\documentclass[preprint]{jmd}
\usepackage{amssymb}
\usepackage{graphicx}
\UseLinks
\theoremstyle{plain}
\newtheorem{theorem}{Theorem}
\newtheorem{lemma}[theorem]{Lemma}
\newtheorem{proposition}[theorem]{Proposition}

\theoremstyle{definition}

\newtheorem{example}[theorem]{Example}
\newtheorem{remark}[theorem]{Remark}

\numberwithin{theorem}{section}
\renewcommand{\phi}{\varphi}

\def\etc.{\emph{et\thinspace c.}}

\renewcommand{\d}{\,\mathrm{d}}

\def\H{\mathbb{H}}

\def\SS{\mathbb{S}}

\let\emptyset\varnothing

\def\conorm#1{|\kern-.2em\lfloor#1\rfloor\kern-.2em|}

\begin{document}
\title{Lecture notes on the dynamics of the Weil-Petersson flow}
\thanks{This work was partially supported by the French ANR grant ``GeoDyM'' (ANR-11-BS01-0004) and the Balzan Research Project of J. Palis.}
\author{Carlos Matheus}
\email{matheus.cmss@gmail.com}
\address{Universit\'e Paris 13, Sorbonne Paris Cit\'e,
LAGA, CNRS (UMR 7539), F-93430, Villetaneuse, France.}
\Subjclass{37D20}{57N10}

\maketitle\tableofcontents

\section{Introduction}\label{s.introduction}

\subsection{Some words on the origin of these notes} This text is an expanded version of some lecture notes prepared by the author in the occasion of a series of three lectures during the workshop \emph{Young mathematicians in dynamical systems} organized by Fran\c coise Dal'bo, Louis Funar, Boris Hasselblatt and Barbara Schapira in November 2013 at Centre International de Rencontres Math\'ematiques (CIRM), Marseille, France.

As it is explained in the introduction of Hasselblatt's text \cite{Hasselblatt} in this volume, the three lectures at the origin of this text were part of a minicourse by Keith Burns, Boris Hasselblatt and the author around the recent theorem of Burns-Masur-Wilkinson \cite{BurnsMasurWilkinson} on the \emph{ergodicity of the Weil-Petersson (WP) geodesic flow}.

Of course, the goal of these notes is the same of the author's lectures: we want to cover some of the aspects related to moduli spaces of Riemann surfaces (and Teichm\"uller theory) in the proofs of the ergodicity of WP flow \cite{BurnsMasurWilkinson} (see also Theorem \ref{t.BMW} below) and the recent results of Burns, Masur, Wilkinson and the author \cite{BurnsMasurMatheusWilkinson} on the rates of mixing of WP flow (see also Theorem \ref{t.BMMW} below).

\subsection{An overview of the dynamics of WP flow} Before giving precise definitions of the terms introduced above (e.g., moduli spaces of Riemann surfaces, Weil-Petersson geodesic flow, etc.), let us list and compare some properties of the WP flow and its close cousin the \emph{Teichm\"uller (geodesic) flow} (see \cite{Zorich}) in order to get a flavor of their dynamical behaviors.

\medskip
\begin{center}
\begin{tabular}{|c|p{5.5cm}|p{5.5cm}|}
\hline
 & Teichm\"uller flow & WP flow \\
\hline
(a) & comes from a Finsler metric & comes from a Riemannian metric \\
\hline
(b) & complete & incomplete \\
\hline
(c) & is part of a $SL(2,\mathbb{R})$-action & is not part of a $SL(2,\mathbb{R})$-action \\
\hline
(d) & non-uniformly hyperbolic & singular hyperbolic \\
\hline
(e) & related to flat geometry of Riemann surfaces & related to hyperbolic geometry of Riemann surfaces \\
\hline
(f) & transitive & transitive \\
\hline
(g) & periodic orbits are dense & periodic orbits are dense \\
\hline
(h) & finite topological entropy & infinite topological entropy \\
\hline
(i) & ergodic for the Liouville measure $\mu_{T}$ & ergodic for the Liouville measure $\mu_{WP}$ \\
\hline
(j) & metric entropy $0<h(\mu_T)<\infty$ & metric entropy $0<h(\mu_{WP})<\infty$ \\
\hline
(k) & exponential rate of mixing & mixing at most polynomial (in general) \\
\hline
\end{tabular}
\end{center}
\medskip

Let us make some comments on both the common features and the significant differences between the Teichm\"uller and WP flows highlighted in the items above.

The Teichm\"uller flow is associated to a \emph{Finsler metric} (i.e., a continuous family of norms) on the fibers of the cotangent bundle of the moduli spaces\footnote{Actually, the Finsler metric corresponding to Teichm\"uller flow is a $C^1$ but \emph{not} $C^2$ family of norms: see, e.g., pages 308 and 309 of Hubbard's book \cite{Hubbard}.}, while the WP flow is associated to a \emph{Riemannian} (and, actually, \emph{K\"ahler}) \emph{metric} called \emph{Weil-Petersson (WP) metric}. In particular, the item (a) says that the WP flow comes from a metric that is \emph{smoother} than the metric generating the Teichm\"uller flow. We will come back to this point later when defining the WP metric.

On the other hand, the item (b) says that the dynamics of WP flow is not so nice because it is \emph{incomplete}, that is, there are certain WP geodesics that ``go to \emph{infinity}''  in \emph{finite} time. In particular, the WP flow is \emph{not} defined for all time $t\in\mathbb{R}$ when we start from \emph{certain} initial data. We will make more comments on this later. Nevertheless, Wolpert \cite{Wolpert2003} showed that the WP flow is defined for all time $t\in\mathbb{R}$ for \emph{almost every} initial data with respect to the Liouville (volume) measure induced by WP metric, and, thus, the WP flow is a legitime flow from the point of view of Ergodic Theory.

The item (c) says that WP flow is less \emph{algebraic} than Teichm\"uller flow because the former is not part of a $SL(2,\mathbb{R})$-action while the latter corresponds to the diagonal subgroup $g_t=\textrm{diag}(e^t,e^{-t})$ of $SL(2,\mathbb{R})$ acting (in a natural way) on the unit cotangent bundle of the moduli spaces of Riemann surfaces. Here, it is worth to mention that the mere fact that the Teichm\"uller flow is part of a $SL(2,\mathbb{R})$-action makes its dynamics very \emph{rich}: for instance, once one shows that the Teichm\"uller flow is ergodic (with respect to some $SL(2,\mathbb{R})$-invariant probability measure), it is possible to apply Howe-Moore's theorem (or variants of it) to improve ergodicity into mixing (and, actually, exponential mixing) of Teichm\"uller flow (see, e.g., \cite{AvilaGouezel} and \cite{AvilaGouezelYoccoz} for more details).

The item (d) says that WP and Teichm\"uller flows (morally) are non-uniformly hyperbolic in the sense of Pesin theory \cite{Pesinthy}, but they are so for \emph{distinct} reasons. The non-uniform hyperbolicity of the Teichm\"uller flow was shown by Veech \cite{Veech86} (for ``volume''/Masur-Veech measure) and Forni \cite{Forni} (for arbitrary invariant probability measures) and it follows from uniform estimates for the derivative of the Teichm\"uller flow on compact sets. On the other hand, the non-uniform hyperbolicity of the WP flow requires a slightly different argument because some sectional curvatures of WP metric approach $-\infty$ or $0$ at certain places near the ``boundary'' of the moduli spaces. We will return to this point in the future.

The item (e) \emph{partly} explains the interest of several authors in Teichm\"uller and WP flows. Indeed, since their introduction by Bernard Riemann in 1851 (in his PhD thesis), the study of Riemann surfaces and their moduli spaces became an important topic of research in both Mathematics and Physics (for reasons whose explanations are beyond the scope of these notes). In particular, the fact that the properties of the Teichm\"uller and WP flows on moduli spaces allows to recover geometrical information about Riemann surfaces motivated part of the literature on the dynamics of these flows. Concerning applications of these flows to the investigation of Riemann surfaces, it is natural to study the Teichm\"uller flow whenever one is interested in the properties of flat metrics with conical singularities on Riemann surfaces (cf. Zorich's survey \cite{Zorich}), while it is more natural to study the WP metric/flow whenever one is interested in the properties of \emph{hyperbolic metrics} on Riemann surfaces: for instance, Wolpert \cite{Wolpert2008} showed that the hyperbolic length of a closed geodesic in a fixed free homotopy class is a convex function along orbits of the WP flow, Mirzakhani \cite{Mirzakhani2008} proved that the growth of the hyperbolic lengths of simple geodesics on hyperbolic surfaces is related to the WP volume of the moduli space, and, after the works of Bridgeman \cite{Bridgeman2010}, McMullen \cite{McMullen2008} and more recently Bridgeman-Canary-Labourie-Sambarino \cite{BCLS2013} (among other authors), we know that the Weil-Petersson metric is intimately related to \emph{thermodynamical invariants} (entropy, pressure, etc.) of the geodesic flow on hyperbolic surfaces.

Concerning items (f) to (h), Pollicott-Weiss-Wolpert \cite{PollicottWeissWolpert2010} showed the transitivity and denseness of periodic orbits of the WP flow in the particular case of the unit cotangent bundle of the moduli space $\mathcal{M}_{1,1}$ (of once-punctured tori). In general, the transitivity, the denseness of periodic orbits and the infinitude of the topological entropy of the WP flow on the unit cotangent bundle of the moduli space $\mathcal{M}_{g,n}$ of genus $g$ Riemann surfaces with $n$ marked points (for any $g\geq 1$, $n\geq 1$) were shown by Brock-Masur-Minsky \cite{BrockMasurMinsky2010}.  Moreover, Hamenst\"adt \cite{Hamenstadt2010} proved the \emph{ergodic version} of the denseness of periodic orbits, i.e., the denseness of the subset of ergodic probability measures supported on periodic orbits in the set of all ergodic WP flow invariant probability measures.

The ergodicity of WP flow (mentioned in item (i)) was first studied by Pollicott-Weiss \cite{PollicottWeiss2009} in the particular case of the unit cotangent bundle $T^1\mathcal{M}_{1,1}$ of the moduli space $\mathcal{M}_{1,1}$ of once-punctured tori: they showed that \emph{if} the first two derivatives of the WP flow on $T^1\mathcal{M}_{1,1}$ are suitably bounded, \emph{then} this flow is ergodic. More recently, Burns-Masur-Wilkinson \cite{BurnsMasurWilkinson} were able to control \emph{in general} the first derivatives of WP flow and they used their estimates to show the following theorem:

\begin{theorem}[Burns-Masur-Wilkinson]\label{t.BMW} The WP flow on the unit cotangent bundle $T^1\mathcal{M}_{g,n}$ of the moduli space $\mathcal{M}_{g,n}$ of Riemann surfaces of genus $g$ with $n$ marked points is ergodic with respect to the Liouville measure $\mu_{WP}$ of the WP metric whenever $3g-3+n\geq 1$. Actually, it is Bernoulli (i.e., it is measurably isomorphic to a Bernoulli shift) and, \emph{a fortiori}, mixing. Furthermore, its metric entropy $h(\mu_{WP})$ is positive and finite.
\end{theorem}

The Teichm\"uller-theoretical aspects of this theorem will occupy the next two sections of this text. For now, we will just try to describe the general lines of Burns-Masur-Wilkinson arguments in Subsection \ref{s.BMW-outline} below.

However, before passing to this topic, let us make some comments about item (k) above on the rate of mixing of Teichm\"uller and WP flows.

Generally speaking, it is expected that the rate of mixing of a system (diffeomorphism or flow) displaying a ``reasonable'' amount of \emph{hyperbolicity} is exponential: for example, the property of exponential rate of mixing was shown  by Dolgopyat \cite{Dolgopyat} (see also this article of Liverani \cite{Liverani}) for a large class of contact Anosov flows\footnote{Including certain geodesic flows on compact Riemannian manifolds with negative curvature.}, and by Avila-Gou\"ezel-Yoccoz \cite{AvilaGouezelYoccoz} and Avila-Gou\"ezel \cite{AvilaGouezel} for the Teichm\"uller flow equipped with ``nice'' measures.

Here, we recall that the \emph{rate of mixing}/\emph{decay of correlations} of a mixing flow $\psi^t$ is the speed of convergence to zero of the correlations functions $C_t(f,g):=\int f\cdot g\circ\psi^t - \left(\int f\right)\left(\int g\right)$ as $t\to\infty$ (for choices of ``sufficiently smooth'' observables $f$ and $g$). Intuitively, the rate of mixing is a quantitative measurement of how fast the flow $\psi^t$ mix distinct regions of the phase space (such as the supports of the observables $f$ and $g$). See, e.g., Subsection 6.16 of Hasselblatt's lecture notes \cite{Hasselblatt} for more comments.

In this context, given the ergodicity and mixing theorem of Burns-Masur-Wilkinson stated above, it is natural to try to ``determine'' the rate of mixing of WP flow. In this direction, we obtained the following result (cf. \cite{BurnsMasurMatheusWilkinson}):

\begin{theorem}[Burns-Masur-M.-Wilkinson]\label{t.BMMW} The rate of mixing of WP flow on $T^1\mathcal{M}_{g,n}$ (for ``reasonably smooth'' observables) is
\begin{itemize}
\item at most \textbf{polynomial} for $3g-3+n>1$ and
\item \textbf{rapid} (super-polynomial) for $3g-3+n=1$.
\end{itemize}
\end{theorem}

We will present a sketch of proof of this result in the last section of this text. For now, we will content ourselves with a vague description of the \emph{geometrical reason} for the difference in the rate of mixing of the Teichm\"uller and WP flows in Subsection \ref{s.BMMW} below.

\subsection{Ergodicity of WP flow: outline of proof}\label{s.BMW-outline}

The initial idea to prove Burns-Masur-Wilkinson theorem is the ``usual'' argument for the proof of ergodicity of a system exhibiting some hyperbolicity, namely, \emph{Hopf's argument}.

\subsubsection{A quick review of Hopf's argument} Traditionally, Hopf's argument runs as follows (cf. Subsection 4.3 of Hasselblatt's lecture notes \cite{Hasselblatt}). Given a smooth flow $(\psi^t)_{t\in\mathbb{R}}:X\to X$ on a compact Riemannian manifold $(X,d)$ preserving the corresponding volume measure $\mu$ and a continuous observable $f:X\to\mathbb{R}$, we consider the future and past \emph{Birkhoff averages}:
$$f^+(x):=\lim\limits_{T\to+\infty}\frac{1}{T}\int_0^T f(\psi^s(x))\,ds \quad \textrm{and} \quad f^-(x):=\lim\limits_{T\to-\infty}\frac{1}{T}\int_0^T f(\psi^s(x))\,ds$$
By \emph{Birkhoff's ergodic theorem} (cf. Subsection 6.3 of \cite{Hasselblatt}),  for $\mu$-almost every $x\in X$, the quantities $f^+(x)$ and $f^-(x)$ exist and, actually, they coincide $f^+(x)=f^-(x):=\widetilde{f}(x)$. In the literature, a point $x$ such that $f^+(x)$, $f^-(x)$ exist and $f^+(x)=f^-(x)=\widetilde{f}(x)$ is called a \emph{Birkhoff generic} point (with respect to $\mu$).

By definition, the ergodicity of $\psi^t$ (with respect to $\mu$) is equivalent to the fact that the functions $f^+$ and $f^-$ are \emph{constant} at $\mu$-almost every point.

In order to show the ergodicity of a flow $\psi^t$ with some hyperbolicity, Hopf \cite{Hopf} observes that the function $f^+$, resp. $f^-$, is constant along stable, resp. unstable, sets
$$W^s(x):=\{y: \lim\limits_{t\to+\infty}d(\psi^t(y),\psi^t(x))=0\}, \textrm{resp.} W^u(x)=\{y: \lim\limits_{t\to-\infty}d(\psi^t(y),\psi^t(x))=0\},$$
i.e., $f^+(x)=f^+(y)$ whenever $y\in W^s(x)$, resp. $f^-(x)=f^-(z)$ whenever $z\in W^u(x)$. We leave the verification of this fact as an exercise to the reader.

In the case of an \emph{Anosov flow} $\psi^t$ on $X$, we know that the stable and unstable sets are immersed \emph{submanifolds} (cf. Subsection 5.5 of Hasselblatt's notes \cite{Hasselblatt}). Moreover, if one forgets about the flow direction, the stable and unstable manifolds have complementary dimensions and intersect transversely. Hence, given two points $p, q\in X$ (lying in distinct orbits of $\psi^t$), we can connect them using pieces of stable and unstable manifolds as shown in the figure below:

\begin{figure}[htb!]
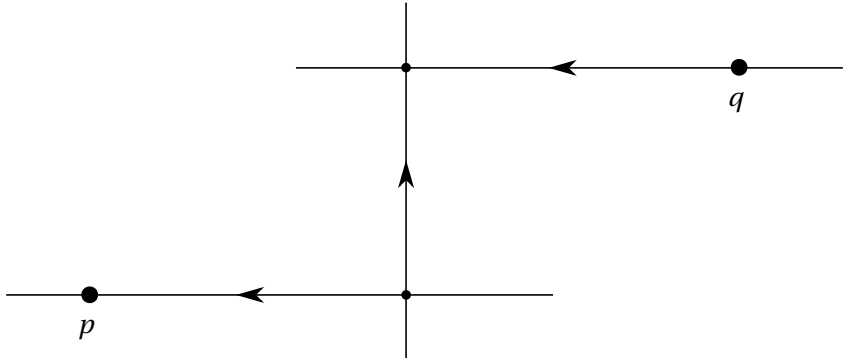
\caption{Connecting $p$ and $q$ with pieces of stable and unstable manifolds.}\label{f.1}
\end{figure}

In particular, this \emph{indicates} that a volume-preserving Anosov flow $\psi^t$ is ergodic because the functions $f^+$ and $f^-$ are constant along stable and unstable manifolds, they coincide almost everywhere and any pair of points can be connected via pieces of stable and unstable manifolds. However, this argument towards ergodicity of $\psi^t$ is \emph{not} complete yet: indeed, one needs to know that the intersection points $z_1,\dots, z_n$ between the pieces of stable and unstable manifolds connecting $p$  and $q$ are Birkhoff generic in order to conlude that $\widetilde{f}(p)=\widetilde{f}(z_1)=\dots=\widetilde{f}(z_n)=\widetilde{f}(q)$.

In the original context of his article, Hopf \cite{Hopf} studies a geodesic flow $\psi^t$ of a compact surface of \emph{constant} negative curvature, and he uses the fact that the stable and unstable manifolds form $C^1$ foliations to deduce that the intersection points $z_1,\dots, z_n$ can be taken to be Birkhoff generic points. Indeed, since the invariant foliations are $C^1$ in his context, Hopf applies \emph{Fubini's theorem} to the set $\mathcal{B}$ of full $\mu$-volume consisting of Birkhoff generic points in order to ensure that almost all stable and unstable manifolds $W^s(x)$ and $W^u(x)$ intersect $\mathcal{B}$ in a subset of total length measure of $W^s(x)$ and $W^u(x)$ (compare with the proof of Proposition 4.10 of \cite{Hasselblatt}).

On the other hand, it is known that the stable and unstable manifolds of a \emph{general} Anosov flow (such as geodesic flows on compact manifolds of \emph{variable} negative curvature) do \emph{not} form necessarily a $C^1$-foliation, but only H\"older continuous foliations (see e.g. the papers of Anosov \cite{Anosov} and/or Hasselblatt \cite{Hasselblatt1994} for concrete examples). In particular, this is an obstacle to the argument \emph{\`a la Fubini} of the previous paragraph. Nevertheless, Anosov \cite{Anosov} showed that the stable and unstable foliations of a smooth Anosov flow are always \emph{absolutely continuous}, so that one can still apply Fubini's theorem to conclude ergodicity along the lines of Hopf's argument presented.

In summary, we know that a smooth ($C^2$) volume-preserving Anosov flow on a compact manifold is ergodic thanks to Hopf's argument and the absolute continuity of stable and unstable foliations.

\begin{remark} Robinson-Young \cite{RobinsonYoung} showed that the stable and unstable foliations of a $C^1$ Anosov system are not necessarily absolutely continuous. In particular, the smoothness ($C^2$) assumption on the Anosov flow is necessary for the ergodicity argument described above.
\end{remark}

\begin{remark} The absolute continuity of a foliation invariant under some system depends on some hyperbolicity. In fact, Shub-Wilkinson \cite{ShubWilkinson} constructed examples of invariant central (along which the dynamics is neutral) foliations of certain partially hyperbolic diffeomorphisms failing to satisfy Fubini's theorem: each leaf of these central foliations intersects a set of full volume exactly at one point! This phenomenon is sometimes referred to as \emph{Fubini's nightmare} in the literature (see, e.g., this article of Milnor \cite{Milnor}) and sometimes a foliation ``failing'' Fubini's theorem is called a \emph{pathological foliation}.
\end{remark}

After this brief sketch of Hopf's argument for the ergodicity of smooth volume-preserving Anosov flows on compact manifolds, let us explain the difficulties of extending this argument to the setting of WP flow.

\subsubsection{Hopf's argument in the context of WP flow} As we already mentioned (cf. item (d) of the table above), the WP flow is singular hyperbolic. In a nutshell, this means that, even though WP flow is not Anosov, it is (morally) non-uniformly hyperbolic in the sense of Pesin theory and it satisfies some hyperbolicity estimates along pieces of orbits staying in compact parts of moduli space.

In particular, thanks to (Katok-Strelcyn \cite{KatokStrelcyn} version of) \emph{Pesin's stable manifold theorem} \cite{Pesinthy}, the stable and unstable sets of almost every point are immersed submanifolds, and, if we forget about the flow direction, the stable and unstable manifolds have complementary dimensions. Furthermore, the stable and unstable manifolds are part of absolutely continuous laminations. Here, it is important that the dynamics is sufficiently smooth (see, e.g., this paper of Pugh \cite{Pugh}, and this preprint of Bonatti-Crovisier-Shinohara \cite{BCS}).

Thus, this gives hopes that Hopf's argument \emph{could be} applied to show the ergodicity of volume-preserving non-uniformly hyperbolic systems.

However, by inspecting the figure \ref{f.1} above, we see that Hopf's argument relies on the fact that stable and unstable manifolds of Anosov flows have a nice, well-controlled, geometry.

For instance, if we start with a point $p$ and we want to connect it with pieces of stable and unstable manifolds to a point $q$ at a \emph{large} distance, we have to make sure that the pieces of stable and unstable manifolds used in figure \ref{f.1} are ``uniform'', e.g., they are graphs of \emph{definite size} and \emph{bounded curvature} with respect to the splitting into stable and unstable directions, and, moreover, the angles between the stable and unstable directions are \emph{uniformly bounded away from zero}.

Indeed, if the pieces of stable and unstable manifolds get shorter and shorter, and/or if they ``curve'' a lot, and/or the angles between stable and unstable directions are not bounded away from zero, one might not be able to reach/access $q$ from $p$ with stable and unstable manifolds:

\begin{figure}[htb!]
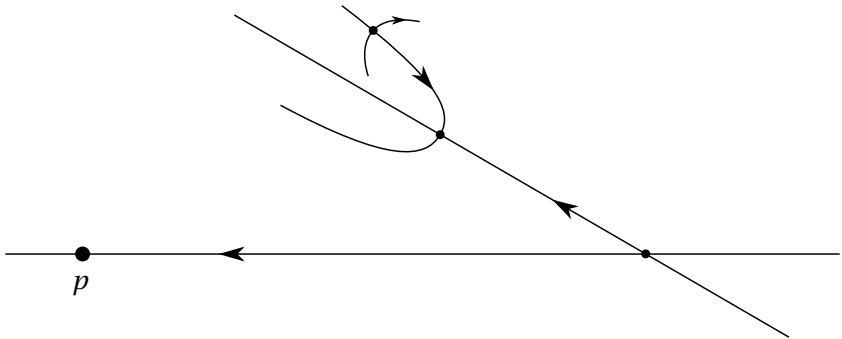
\caption{Pesin stable and unstable manifolds with ``bad'' geometry.}\label{f.2}
\end{figure}

As it turns out, while these kinds of non-uniformity do not occur for Anosov flows, they can actually occur for \emph{certain} non-uniformly hyperbolic systems. More precisely, the sizes and curvatures of stable and unstable manifolds, and the angles between stable and unstable directions of a general non-uniformly hyperbolic system vary only \emph{measurably} from point to point.

In particular, this excludes \emph{a priori} a naive generalization of Hopf's ergodicity argument for non-uniformly hyperbolic systems, and, in fact, there are concrete examples\footnote{As a matter of fact, these examples are ``sharp'': Pugh-Shub \cite{PughShub89} showed that a volume-preserving non-uniformly hyperbolic system has \emph{at most} countably many ergodic components.} by Dolgopyat-Hu-Pesin \cite{DolgopyatHuPesin} of volume-preserving non-uniformly hyperbolic systems with countably many ergodic components consisting of invariant sets of positive volumes that are essentially open.

In summary, the ergodicity of a non-uniformly hyperbolic system \emph{depends} on the particular dynamical features of the given system.

In this direction, there is an important literature dedicated to the construction of large classes of ergodic non-uniformly hyperbolic systems: for example, the ergodicity of several classes of \emph{billiards} was shown by Sinai \cite{Sinai70}, Bunimovich \cite{Bunimovich74}, Bunimovich-Chernov-Sinai \cite{BunimovichChernovSinai91} among others (see also Chernov-Markarian's book \cite{ChernovMarkarian}) and the ergodicity of non-uniformly hyperbolic systems exhibiting \emph{partial hyperbolicity} (or \emph{dominated splitting}) was shown by Pugh-Shub \cite{PughShub00}, Rodriguez-Hertz \cite{RodriguezHertz}, Tahzibi \cite{Tahzibi}, Burns-Wilkinson \cite{BurnsWilkinson}, Rodriguez-Hertz-- Rodriguez-Hertz--Ures \cite{RHRHU} among others.

For the proof of their ergodicity result for the WP flow, Burns-Masur-Wilkinson take part of their inspiration from the work of Katok-Strelcyn \cite{KatokStrelcyn} where Pesin's theory \cite{Pesinthy} (of existence and absolute continuity of stable manifolds) is extended to singular hyperbolic systems.

In a nutshell, the basic philosophy behind Katok-Strelcyn's work is the following. Given a non-uniformly hyperbolic system with some non-trivial singular set, all dynamical features predicted by Pesin theory in virtue of the (non-uniform) \emph{exponential} contraction and expansion are \emph{not} affected \emph{if} the loss of control on the system is at most \emph{polynomial} as one approaches the singular set. In other terms, the \emph{exponential} (hyperbolic) behavior of a singular system is \emph{not} disturbed by the presence of a singular set where the first two derivatives of the system lose control in a \emph{polynomial} way. In particular, this hints that Hopf's argument can be extended to singular hyperbolic systems with polynomially bad singular sets.

In this context, Burns-Masur-Wilkinson shows the following \emph{ergodicity criterion} for singular hyperbolic geodesic flows (cf. Theorem 3.1 of \cite{BurnsMasurWilkinson}).

Let $N$ be the quotient $N=M/\Gamma$ of a contractible, negatively curved, possibly incomplete, Riemannian manifold $M$ by a subgroup $\Gamma$ of isometries of $M$ acting freely and properly discontinuously. By slightly abusing notation, we denote by $d$ the metrics on $N$ and $M$ induced by the Riemannian metric of $M$.

We consider $\overline{N}$ the (Cauchy) \emph{metric completion} of the metric space $(N, d)$, i.e., the (complete) metric space consisting of all equivalence classes of Cauchy sequences $\{x_n\}\subset N$ under the relation $\{x_n\}\sim\{y_n\}$ if and only if $\lim\limits_{n\to\infty} d(x_n,y_n)=0$ equipped with the metric $d(\{x_n\},\{z_n\})=\lim\limits_{n\to\infty} d(x_n, z_n)$, and we define the (Cauchy) \emph{boundary} $\partial N:=\overline{N}-N$.

\begin{theorem}[Burns-Masur-Wilkinson ergodicity criterion for geodesic flows]\label{t.BMW-ergodicity-criterion} Let $N=M/\Gamma$ be a manifold as above. Suppose that:
\begin{itemize}
\item[(I)] the universal cover $M$ of $N$ is \textbf{geodesically convex}, i.e., for every $p,q\in M$, there exists an unique geodesic segment \textbf{in} $M$ connecting $p$ and $q$.
\item[(II)] the metric completion $\overline{N}$ of $(N,d)$ is \textbf{compact}.
\item[(III)] the boundary $\partial N$ is \textbf{volumetrically cusplike}, i.e., for some constants $C>1$ and $\nu>0$, the volume of a $\rho$-neighborhood of the boundary satisfies
$$\textrm{Vol}(\{x\in N: d(x,\partial N)<\rho\})\leq C \rho^{2+\nu}$$
for every $\rho>0$.
\item[(IV)] $N$ has \textbf{polynomially controlled curvature}, i.e., there are constants $C>1$ and $\beta>0$ such that the curvature tensor $R$ of $N$ and its first two derivatives satisfy the following polynomial bound
$$\max\{\|R(x)\|,\|\nabla R(x)\|,\|\nabla^2 R(x)\|\}\leq C d(x,\partial N)^{-\beta}$$
for every $x\in N$.
\item[(V)] $N$ has \textbf{polynomially controlled injectivity radius}, i.e., there are constants $C>1$ and $\beta>0$ such that
$$\textrm{inj}(x)\geq (1/C) d(x,\partial N)^{\beta}$$
for every $x\in N$ (where $\textrm{inj}(x)$ denotes the injectivity radius at $x$).
\item[(VI)] The \textbf{first derivative of the geodesic flow} $\varphi_t$ is \textbf{polynomially controlled}, i.e., there are constants $C>1$ and $\beta>0$ such that, for every infinite geodesic $\gamma$ on $N$ and every $t\in [0,1]$:
$$\|D_{\stackrel{.}{\gamma}(0)}\varphi_t\|\leq C d(\gamma([-t,t]),\partial N)^{\beta}$$
\end{itemize}

Then, the Liouville (volume) measure $m$ of $N$ is finite, the geodesic flow $\varphi_t$ on the unit cotangent bundle $T^1N$ of $N$ is defined at $m$-almost every point for all time $t$, and the geodesic flow $\varphi_t$ is \textbf{non-uniformly hyperbolic} (in the sense of Pesin's theory) and \textbf{ergodic}.

Actually, the geodesic flow $\varphi_t$ is Bernoulli and, furthermore, its metric entropy $h(\varphi_t)$ is positive, finite and $h(\varphi_t)$ is given by Pesin's entropy formula (i.e., $h(\varphi_t)$ is the sum of positive Lyapunov exponents of $\varphi_t$ counted with multiplicities).
\end{theorem}

The proof of this ergodicity criterion for geodesic flows was one of the main motivations of Burns' lectures (see \cite{Burns}) and, for this reason, we will not discuss it here. Instead, we will always \emph{assume} Theorem \ref{t.BMW-ergodicity-criterion} in the sequel, so that the proof of Theorem \ref{t.BMW} (ergodicity of the WP flow) will be complete\footnote{\label{footnote-A}Actually, there is a \emph{subtle} point in the reduction of Theorem \ref{t.BMW} to Theorem \ref{t.BMW-ergodicity-criterion} related to the \emph{orbifoldic} nature of moduli spaces. We will discuss this later in Subsection \ref{ss.ergodicity-orbifold-manifold}.} once we show that the moduli space of Riemann surfaces equipped with the WP metric satisfies the six items (I) to (VI) above.

\subsubsection{A brief comment on the verification of the ergodicity criterion for WP flow} In comparison with previously known results in the literature, some of the main novelties in Burns-Masur-Wilkinson work \cite{BurnsMasurWilkinson} concern the verification of items (IV) and (VI) for the WP metric: in fact, those items are the most delicate to check and their verifications are strongly based on  important previous works of McMullen \cite{McMullen2000} and Wolpert \cite{Wolpert2003}, \cite{Wolpert2008}, \cite{Wolpert2009}, \cite{Wolpert2011}.

In any case, this completes our outline of the proof of Burns-Masur-Wilkinson theorem on the ergodicity of WP flow.

\subsection{Rates of mixing of WP flow}\label{s.BMMW}

As we mentioned above, both Teichm\"uller and WP flows are uniformly hyperbolic in compact parts of the moduli space of curves. Since an uniformly hyperbolic system is (usually) exponentially mixing, the sole obstacle preventing an exponential rate of mixing for these flows is the possibility that a ``big'' set of orbits spends a ``lot'' of time near infinity (or rather the boundary of the moduli space) before coming back to the compact parts.

In the case of Teichm\"uller flow, the volume in Teichm\"uller metric of a $\rho$-neighborhood of the boundary of moduli space is exponentially small\footnote{Its order is $O(e^{-(2-)\rho})$ where $2-$ denotes any fixed  positive real number strictly smaller than $2$; cf. Corollary 2.16 of Avila-Gou\"ezel-Yoccoz paper \cite{AvilaGouezelYoccoz}.}.

Intuitively, this says that the ``probability'' that an orbit spends a long time near the boundary of moduli space is exponentially small (cf. Theorem 2.15 of Avila-Gou\"ezel-Yoccoz paper \cite{AvilaGouezelYoccoz}). In particular, the excursions near infinity of most orbits is not long enough to disrupt the exponential rate of mixing ``imposed'' by hyperbolic dynamics of the Teichm\"uller flow on compact parts. Of course, this is merely a vague intuition behind the exponential mixing of the Techm\"uller flow and the curious reader is encouraged to consult the articles of Avila-Gou\"ezel-Yoccoz \cite{AvilaGouezelYoccoz} and Avila-Gou\"ezel \cite{AvilaGouezel} for detailed explanations.

On the other hand, in the context of the WP flow, we will see that the volume in WP metric of $\rho$-neighborhood of the boundary of moduli space is $\simeq \rho^4$ (compare with Lemma 6.1 of  \cite{BurnsMasurWilkinson}).

Therefore, the ``probability'' that an orbit of WP flow spends a long time near infinity \emph{could} be \emph{only} polynomially small but \emph{not} exponentially small. In particular, this possibility \emph{might} conspire against an exponential mixing of WP flow.

In fact, in our joint work \cite{BurnsMasurMatheusWilkinson} with Burns, Masur and Wilkinson, we construct a subset $A_{\rho}$ of volume $\simeq \rho^{8}$ of orbits of WP flow staying near infinity for a time $\simeq 1/\rho$ (at least). For this sake, we use some estimates of Wolpert \cite{Wolpert2009} (see also Propositions 4.11, 4.12 and 4.13 in Burns-Masur-Wilkinson paper \cite{BurnsMasurWilkinson}) saying that the geometry of WP metric on the moduli space of Riemann surfaces of genus $g\geq 2$ looks like a \emph{product} of the WP metrics on the moduli spaces of curves of lower genera $1\leq g'<g$. In particular, the set $A_{\rho}$ is chosen to correspond to geodesics travelling \emph{almost parallel} to one of the factors of the product for a relatively \emph{long time}.

Of course, the existence of such sets $A_{\rho}$ means that the rate of mixing of WP flow $\psi^t$ can \emph{not} be very fast. Indeed, by taking $g_{\rho}$ a ``smooth approximation'' of the characteristic function of $A_{\rho}$ (i.e., $0\leq g_{\rho}\leq 1$ supported on $A_{\rho}$ and $\int g_{\rho}\simeq\rho^8$), and by letting $f$ be a \emph{fixed} smooth function supported on the compact part (away from infinity), we see that
$$|C_t(f,g_{\rho})|:=\left|\int f\cdot g_{\rho}\circ\psi^t -\left(\int f\right) \left(\int g_{\rho}\right)\right| = \left(\int f\right)\left(\int g_{\rho}\right)\simeq \rho^8$$
for $0\leq t\leq 1/\rho$. In fact, the second equality follows because $f$ is supported in the compact part of the moduli space, $g_{\rho}\circ\psi^t$ is supported on $\psi^{-t}(A_{\rho})$ and the set $\psi^{-t}(A_{\rho})$ is \emph{disjoint} from the compact part for $0\leq t\leq 1/\rho$ (by construction of $A_{\rho}$), so that $f\cdot g_{\rho}\circ\psi^t\equiv 0$ for $0\leq t\leq 1/\rho$. Therefore, at time $t=1/\rho$, we deduce that $C_t(f,g_{\rho})\simeq 1/t^8$, and, hence, the correlation functions associated to WP flow $\psi^t$ can \emph{not} decay faster than a polynomial function of degree $>8$ of $1/t$ as the time $t\to\infty$. In particular, this explains the first part of the statement of Theorem \ref{t.BMMW}.

Finally, let us remark that this argument does not work in genus $g=1$ because the crucial fact (in the construction of the set $A_{\rho}$) that the WP metric looks like the product of WP metrics in moduli spaces of lower genera \emph{breaks down} in genus $g=1$. Indeed, in this situation, the moduli space is naturally compactified by adding a single point (because the moduli space in lower genus $g=0$ is trivial) and so the WP metric does not behave like a product (or, more precisely, \emph{no} sectional curvature approaches zero as we get close to infinity). In this case, one can exploit this ``absence of zero curvatures at infinity'' to show that the rate of mixing of the WP flow on the moduli space of torii is \emph{rapid}, i.e., faster than any polynomial function of $1/t$. In particular, this explains the second part of the statement of Theorem \ref{t.BMMW}.

Concluding this Subsection, let us observe that Theorem \ref{t.BMMW} does not claim that the rate of mixing of the WP flow on moduli space of curves of genus $g\geq 2$ is \emph{genuinely} polynomial.

Indeed, recall that the naive intuition says that the rate of mixing is polynomial if we can show that most orbits do not spend long time near infinity.

Of course, this would \emph{not} be the case if the WP metric is \emph{very close} to a product metric, or, more precisely, if some sectional curvatures of WP metric are very close to zero: in fact, the structure of a product metric near infinity would allow for several orbits to travel almost parallel to the factors of the product (and, hence, near infinity) for a very long time.

So, we need estimates saying how \emph{fast} the sectional curvatures of WP metric approach zero as one gets close to infinity, and, \emph{unfortunately}, the best formulas for the sectional curvatures of WP metric near infinity available so far (due to Wolpert \cite{Wolpert2009}) do not give this type of information (because of certain \emph{potential} cancellations in Wolpert's calculations).

\subsection{Organization of the text} The remainder of these lectures notes are divided into three sections. Section \ref{s.2} contains introductory material on moduli spaces and WP metrics. Section \ref{s.wp-geometry} is dedicated to the proof of Theorem \ref{t.BMW}. Finally, Section \ref{s.mixing-rate-WP} gives a sketch of the proof of Theorem \ref{t.BMMW}. 

\section{Moduli spaces of Riemann surfaces and the Weil-Petersson metric}\label{s.2}

The main purposes of this section are the following. In the next seven subsections below, we recall the definitions and basic properties of the moduli spaces of Riemann surfaces and their cotangent bundles, and we introduce the Weil-Petersson (and Teichm\"uller) metric(s). In particular, the definition of the main actor of these lecture notes, namely the Weil-Petersson geodesic flow, is presented in details in Subsection \ref{ss.Teich-WP-metrics}. The basic reference for these subsections is Hubbard's book \cite{Hubbard}.

Finally, we fulfill in the last subsection the promise made in footnote \ref{footnote-A} to explain the \emph{subtle} point in the reduction of the ergodicity of WP flow (Theorem \ref{t.BMW}) to the ergodicity criterion for geodesic flows (Theorem \ref{t.BMW-ergodicity-criterion}) related to the \emph{orbifoldic} nature of moduli spaces (cf. Subsection \ref{ss.ergodicity-orbifold-manifold}). Of course, this is a technicality about moduli spaces and the reader might wish to skip this subsection in a first reading of this text.

\subsection{Definition and examples of moduli spaces}

Let $S$ be a fixed topological surface of genus $g\geq0$ with $n\geq 0$ punctures. The \emph{moduli space} $\mathcal{M}(S)=\mathcal{M}_{g,n}$ is the set of Riemann surface structures on $S$ \emph{modulo} biholomorphisms (conformal equivalences).

\begin{example}[Moduli space of triply punctured spheres]\label{ex.M03} The moduli space $\mathcal{M}_{0,3}$ of triply punctured spheres consists of a single point
$$\mathcal{M}_{0,3}=\{\overline{\mathbb{C}}-\{0,1,\infty\}\}$$
where $\overline{\mathbb{C}}$ denotes the Riemann sphere. Indeed, this is a consequence of the fact that the group of biholomorphisms (M\"obius transformations) of the Riemann sphere $\overline{\mathbb{C}}$ is \emph{simply 3-transitive}, i.e., given $3$ points $x,y,z\in\overline{\mathbb{C}}$, there exists an unique biholomorphism of $\overline{\mathbb{C}}$ sending $x$, $y$ and $z$ (resp.) to $0$, $1$ and $\infty$ (resp.).
\end{example}

\begin{example}[Moduli space of once punctured torii]\label{ex.moduli-torii} The moduli space $\mathcal{M}_{1,1}$ of once punctured torii is
$$\mathcal{M}_{1,1}=\mathbb{H}/SL(2,\mathbb{Z})$$
where $SL(2,\mathbb{Z})$ acts on the hyperbolic half-plane $\mathbb{H}:=\{z\in\mathbb{C}:\textrm{Im}(z)>0\}$ via M\"obius transformations, i.e., $\left(\begin{array}{cc}a& b \\ c & d \end{array}\right)\in SL(2,\mathbb{Z})$ acts on $\mathbb{H}$ via $$\left(\begin{array}{cc}a& b \\ c & d \end{array}\right)z:=\frac{az+b}{cz+d}$$ Indeed, this follows from the facts that:
\begin{itemize}
\item a complex torus with a marked point is biholomorphic to a ``normalized'' lattice $\mathbb{C}/(\mathbb{Z}\oplus\mathbb{Z}z)$ for some $z\in\mathbb{H}$ (with the marked point corresponding to the origin), and
\item two ``normalized'' lattices $\mathbb{C}/(\mathbb{Z}\oplus\mathbb{Z}z)$ and $\mathbb{C}/(\mathbb{Z}\oplus\mathbb{Z}w)$ are biholomorphic if and only if $w=\frac{az+b}{cz+d}$ for some $\left(\begin{array}{cc}a& b \\ c & d \end{array}\right)\in SL(2,\mathbb{Z})$.
\end{itemize}
\end{example}

The second example reveals an interesting feature of $\mathcal{M}_{1,1}$: it is \emph{not} a manifold, but only an \emph{orbifold}. In fact, the \emph{stabilizer} of the action of $SL(2,\mathbb{Z})$ on $\mathbb{H}$ at a \emph{typical} point is trivial, but it has order $2$ at $i\in\mathbb{H}$ and order $3$ at $\exp(\pi i/3)\in\mathbb{H}$ (this happens because a typical torus has no symmetry, but the square and hexagonal torii have some symmetries). In particular, $\mathcal{M}_{1,1}$ is topologically an once punctured sphere with two conical singularities at $i$ and $\exp(\pi i/3)$. The figure below is a classical fundamental domain of the action of $SL(2,\mathbb{Z})$ on $\mathbb{H}$ together with the actions of the matrices $T=\left(\begin{array}{cc} 1 & 1 \\ 0 & 1 \end{array}\right)$ and $J=\left(\begin{array}{cc} 0 & -1 \\ 1 & 0 \end{array}\right)$:

\begin{figure}[htb!]
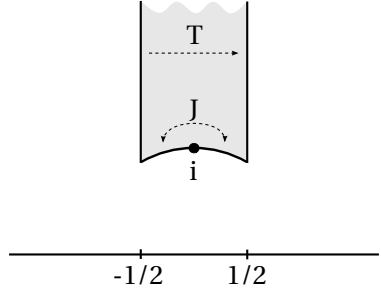
\caption{Fundamental domain $\{z\in\mathbb{H}: |\textrm{Re}(z)|\leq1/2, |z|\geq 1\}$ for $\mathbb{H}/SL(2,\mathbb{Z})$.}\label{f.modular-curve}
\end{figure}

As it turns out, all moduli spaces $\mathcal{M}_{g,n}$ are complex orbifolds. In order to see this fact, we need to introduce some auxiliary structures (including the notions of \emph{Teichm\"uller spaces} and \emph{mapping class groups}).

\begin{remark} \textbf{From now on}, we will restrict our attention to the case of a topological surface $S$ of genus $g\geq 0$ with $n\geq 0$ punctures such that $3g-3+n>0$. In this case, the \emph{uniformization theorem} says that a Riemann surface structure $X$ on $S$ is conformally equivalent to a quotient $\mathbb{H}/\Gamma$ of the hyperbolic upper-half plane $\mathbb{H}$ by a discrete subgroup of $SL(2,\mathbb{R})$ (isomorphic to the fundamental group of $S$). Moreover, the hyperbolic metric $\widetilde{\rho}=\frac{|dz|}{\textrm{Im}(z)}$ on $\mathbb{H}$ descends to a finite area hyperbolic metric $\rho$ on $\mathbb{H}/\Gamma$ and, in fact, $\rho$ is the unique Riemannian metric of constant curvature $-1$ on $X$ inducing the \emph{same} conformal structure. (See, e.g., Hubbard's book \cite{Hubbard} for more details)
\end{remark}

\subsection{Teichm\"uller metric}

Let us start by endowing the moduli spaces with the structure of complete metric spaces.

By definition, a metric on $\mathcal{M}(S)$ corresponds to a way to measure the distance between two points in $\mathcal{M}(S)$. A natural way of telling how far apart are two conformal structures on $S$ is by the means of \emph{quasiconformal maps}.

\emph{Very roughly speaking}, the idea is that even though by definition there is no conformal maps (biholomorphisms) between conformal structures $S_0$ and $S_1$ corresponding two distinct points of $\mathcal{M}(S)$, one has several \emph{quasiconformal maps} between them, that is, $f:S_0\to S_1$ such that the quantity
$$K(f)=\sup\limits_{x\in S_0}\frac{|\partial f(x)/\partial z|+|\partial f(x)/\partial \overline{z}|}{|\partial f(x)/\partial z|-|\partial f(x)/\partial \overline{z}|}\geq 1$$
is \emph{finite}.

Here, it is worth to point out that $K(f)$ is measuring the largest possible eccentricity among all infinitesimal ellipses in the tangent planes $T_{f(x)}S_1$ obtained as images under $Df(x)$ of infinitesimal circles  on the tangent planes $T_x S_0$, and, moreover, $f:S_0\to S_1$ is conformal if and only if $K(f)=1$. See Hubbard's book \cite{Hubbard} for details (including some pictures of the geometrical meaning of $K(f)$).

This motivates measuring the ``distance'' between $S_0$ and $S_1$ via the formula:
$$d_T(S_0,S_1)=\inf_{f:S_0\to S_1 \textrm{ quasiconformal }}\log K(f)$$

This function $d_T(.,.)$ is the so-called \emph{Teichm\"uller metric} and, as the nomenclature suggests, it can be shown that $d_T(.,.)$ is a metric on $\mathcal{M}(S)$.

The moduli space $\mathcal{M}(S)$ endowed with $d_T(.,.)$ is a \emph{complete} metric space.

\begin{example} The Teichm\"uller metric on the moduli space $\mathcal{M}_{1,1}=\mathbb{H}/SL(2,\mathbb{Z})$ of once-punctured torii can be shown to coincide with the hyperbolic metric induced by Poincar\'es metric on $\mathbb{H}$ (see Hubbard's book).
\end{example}

\subsection{Teichm\"uller spaces and mapping class groups}

Once we know that the moduli spaces are topological spaces (and, actually, complete metric spaces), we can start the discussion of its (orbifold) universal cover.

In this direction, we need to describe the ``fiber'' in the universal cover of a point $X$ of $\mathcal{M}(S)$ (i.e., a Riemann surface structure on $S$). In other terms, we need to add ``extra information'' to $X$. As it turns out, this ``extra information'' has topological nature and it is called a \emph{marking}.

More precisely, a \emph{marked complex structure} (on $S$) is the data of a Riemann surface $X$ together with a homeomorphism $f:S\to X$ (called marking).

By analogy with the notion of moduli spaces, we define the \emph{Teichm\"uller space} $Teich(S)$ is the set of Teichm\"uller equivalence classes of marked complex structures, where two marked complex structures $f:S\to X_1$ and $g:S\to X_2$ are \emph{Teichm\"uller equivalent} whenever there exists a conformal map $h:X_1\to X_2$ isotopic to $g\circ f^{-1}$. In other words, the Teichm\"uller space is the ``moduli space of marked complex structures''.

The Teichm\"uller metric $d_T(.,.)$ also makes sense on the Teichm\"uller space $Teich(S)$ and the metric space $(Teich(S), d_T)$ is also complete.

From the definitions, we see that one can recover the moduli space from the Teichm\"uller space by forgetting the ``extra information'' given by the markings. Equivalently, we have that $\mathcal{M}(S)=Teich(S)/MCG(S)$ where $MCG(S)=MCG_{g,n}$ is the so-called \emph{mapping class group} of isotopy classes of orientation-preserving homeomorphisms of $S$.

The mapping class group is a discrete group acting on $Teich(S)$ by isometries of the Teichm\"uller metric $d_T$. Moreover, by  \emph{Hurwitz theorem} (and our standing assumption that $3g-3+n>0$), the $MCG(S)$-stabilizer of any point of $Teich(S)$ is finite (of cardinality $\leq 84(g-1)$ when $g>1$), but it might vary from point to point because some Riemann surfaces are more symmetric than others (see, e.g., the paragraph after Example \ref{ex.moduli-torii} above).

\begin{example} The Teichm\"uller space $Teich_{1,1}$ of once-punctured torii is $$Teich_{1,1}\simeq\mathbb{H}.$$
Indeed, as we already mentioned (cf. Example \ref{ex.moduli-torii}), the set of once-punctured torii is parametrized by normalized lattices $\Lambda(w)=\mathbb{Z}\oplus \mathbb{Z}w$, $w\in\mathbb{H}$, and there is a conformal map between $\mathbb{C}/\Lambda(w)$ and $\mathbb{C}/\Lambda(w')$ if and only if $w'=\frac{aw+b}{cw+d}$, $\left(\begin{array}{cc}a&b\\ c&d\end{array}\right)\in SL(2,\mathbb{Z})$. From this, one can check that $Teich_{1,1}=\mathbb{H}$ and $MCG_{1,1}=SL(2,\mathbb{Z})$ (because the conformal map associated to $\left(\begin{array}{cc}a&b\\ c&d\end{array}\right)$ is isotopic to the identity if and only if $\left(\begin{array}{cc}a&b\\ c&d\end{array}\right)= Id$).
\end{example}

The Teichm\"uller space $Teich(S)$ is the (orbifold) universal cover of $\mathcal{M}(S)$ and $MCG(S)$ is the (orbifold) fundamental group of $\mathcal{M}(S)$ (compare with the example above). A common way to see this fact passes through showing that $Teich(S)$ is \emph{simply connected} (and even \emph{contractible}) because it admits a \emph{global} system of coordinates called \emph{Fenchel-Nielsen coordinates} (providing an homemorphism between $Teich(S)$ and $\mathbb{R}^{6g-6+n}$). The discussion of these coordinates is the topic of the next subsection.

\subsection{Fenchel-Nielsen coordinates}

In order to introduce the Fenchel-Nielsen coordinates, we need the notion of \emph{pants decomposition}. A pants (trouser) decomposition of $S$ is a collection $\{\alpha_1,\dots,\alpha_{3g-3+n}\}$ of $3g-3+n$ simple closed curves on $S$ that are pairwise disjoint, homotopically non-trivial (i.e., not homotopic to a point) and \emph{non-peripheral} (i.e., not homotopic to a small loop around one of the possible punctures of $S$). The picture below illustrates a pants decomposition of a compact surface of genus $2$:

\begin{figure}[htb!]
\includegraphics[scale=0.7]{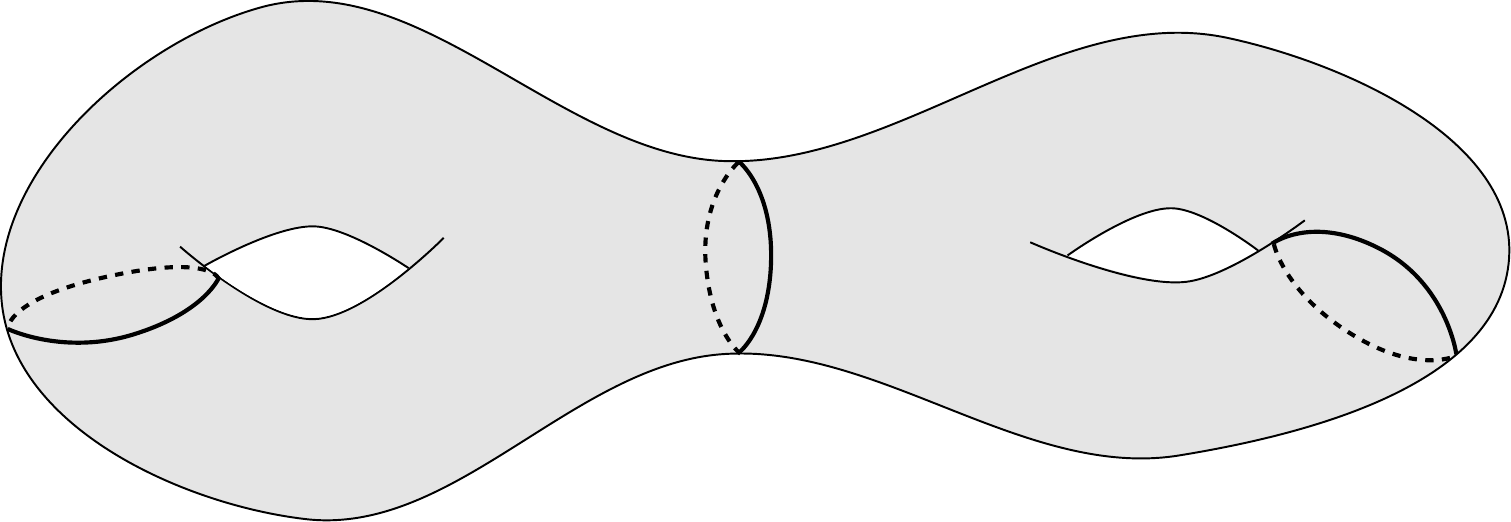}
\end{figure}

The nomenclature ``pants decomposition'' comes from the fact that if we cut $S$ along the curves $\alpha_{j}$, $j=1,\dots,3g-3+n$ (i.e., we consider the connected components of the complement of these curves), then we see ``pairs of pants'' (topologically equivalent to a triply punctured sphere):

\begin{figure}[htb!]
\includegraphics[scale=0.5]{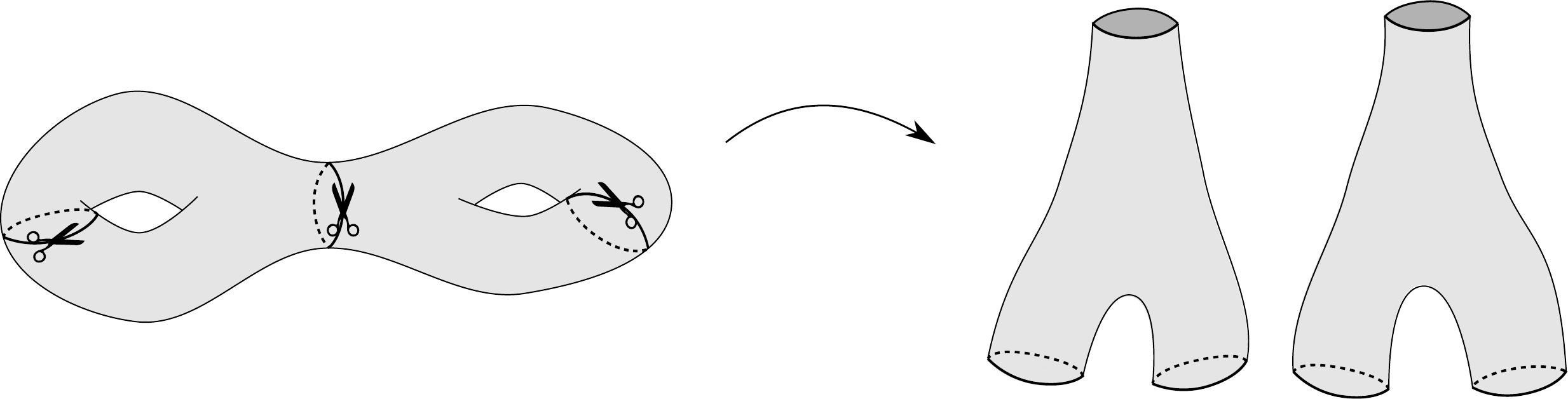}
\end{figure}

A remarkable fact about pair of pants is that hyperbolic structures on them are \emph{uniquely} determined by the lengths of their boundary components. In other terms, a trouser with $j$ boundary circles ($j=1, 2$ or $3$) has a $j$-dimensional space of hyperbolic structures (parametrized by the lenghts of these $j$-circles). Alternatively, one can construct trousers out of right-angled hexagons in the hyperbolic plane (see, e.g., Theorem 3.5.8 in Hubbard's book \cite{Hubbard}).

In this setting, the Fenchel-Nielsen coordinates can be described as follows. We fix $P=\{\alpha_1,\dots,\alpha_{3g-3+n}\}$ a pants decomposition and we consider
$$\mathcal{FN}_{P}: Teich(S)\to(\mathbb{R}_+\times\mathbb{R})^{3g-3+n}$$
defined by $\mathcal{FN}_P(f:S\to X)=(\ell_{\alpha_1},\tau_{\alpha_1},\dots, \ell_{\alpha_{3g-3+n}}, \tau_{\alpha_{3g-3+n}})$, where $\ell_{\alpha}$ is the hyperbolic length of $\alpha\in P$ with respect to the hyperbolic structure associated to the marked complex structure $f:S\to X$, and $\tau_{\alpha}$ is a \emph{twist parameter} measuring the ``relative displacement'' of the pairs of pants glued at $\alpha$.

A detailed description of twist parameters can be found in Section 7.6 of Hubbard's book \cite{Hubbard}, but, for now, let us just make some quick comments about them. First, we fix (in an \emph{arbitrary} way) a collection of simple arcs joining the boundaries of the pairs of pants determined by $P$ such that these arcs land at the same point whenever they come from opposite sides of $\alpha_{j}\in P$.

\begin{figure}[htb!]
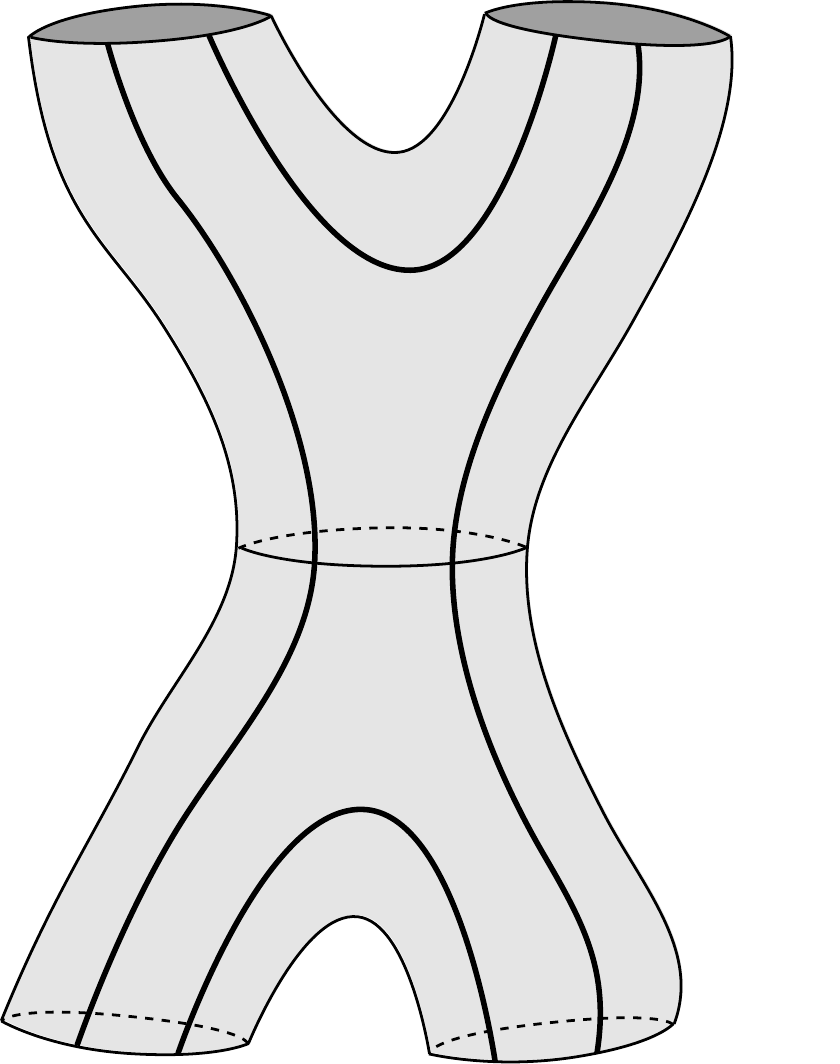
\end{figure}

From these arcs, we get a collection $P^*$ of simple closed curves on $S$ looking like this:

\begin{figure}[htb!]
\includegraphics[scale=0.7]{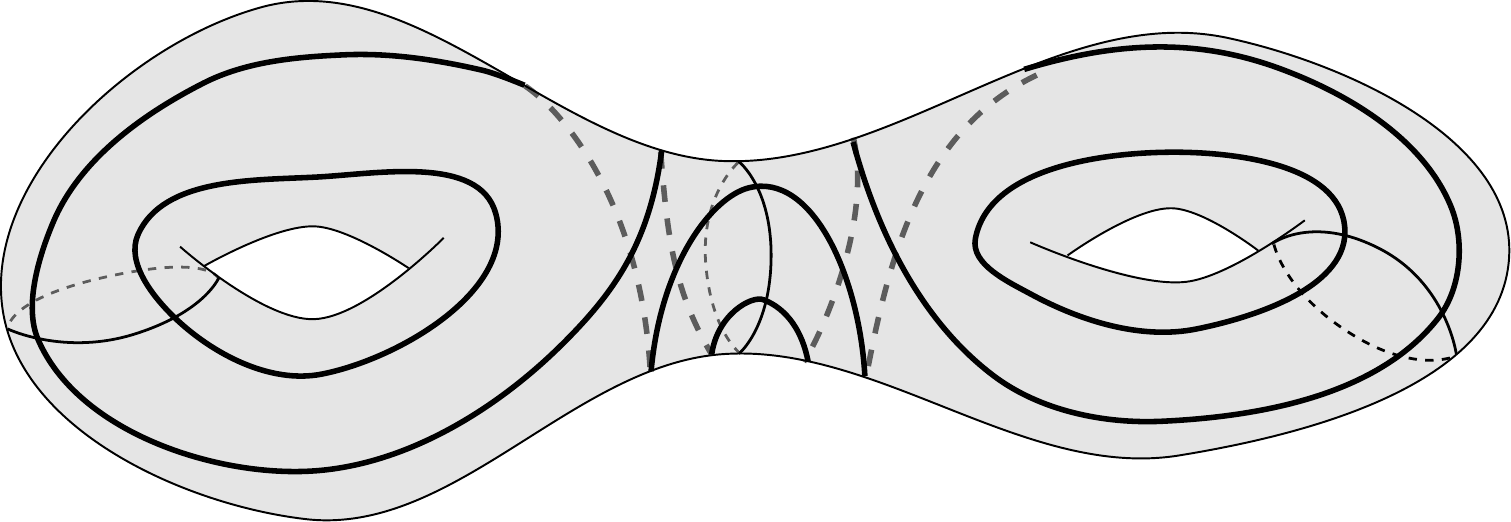}
\end{figure}

Consider now a pair of trousers sharing a curve $\alpha\in P$ (they might be the same trouser) and let $\gamma^*$ be an arc of a curve in $P^*$ joining two boundary components $A(\gamma^*)$ and $B(\gamma^*)$ of the union of these trousers:

\begin{figure}[htb!]
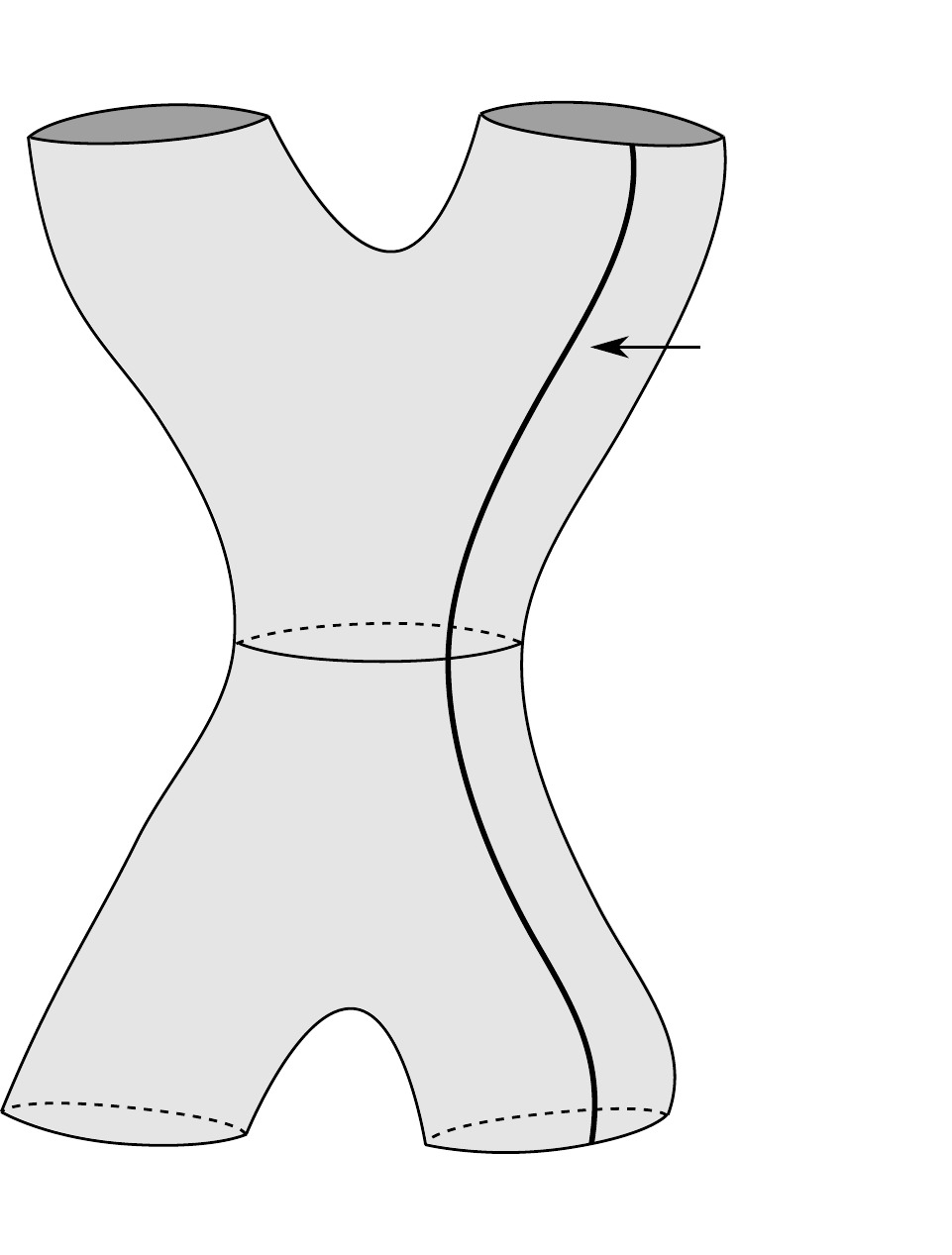
\end{figure}

Given a marked complex structure $f:S\to X$, consider the unique arc $\alpha(\gamma^*)$ on $X$ homotopic to $f(\gamma^*)$ (relative to the boundary of the union of the pair of trousers) consisting of two \emph{minimal} geodesic arcs connecting $\alpha\in P$ to $A(\gamma^*)$ and $B(\gamma^*)$ and an immersed geodesic $\delta(\gamma^*)$ moving inside $\alpha\in P$. We define the twist parameter $\tau_{\alpha}(f:S\to X)$ as the oriented length of $\delta(\gamma^*)$ counted as positive if it turns to the right and negative if it turns to the left.

\begin{remark} Since the definition of twist parameters \emph{depend} on the choice of $P^*$, these parameters are well-defined only \emph{up to an additive constant}. Nevertheless, this technical difficulty does not lead to any serious issue.
\end{remark}

The figure \ref{f.concrete-twist} below illustrates two markings $f:S\to X$ and $g:S\to Y$ whose twist parameters differ by
$$\tau_{\alpha}(g:S\to Y)=\tau_{\alpha}(f:S\to X)+2\ell_{\alpha}(f:S\to X)$$

\begin{figure}[htb!]
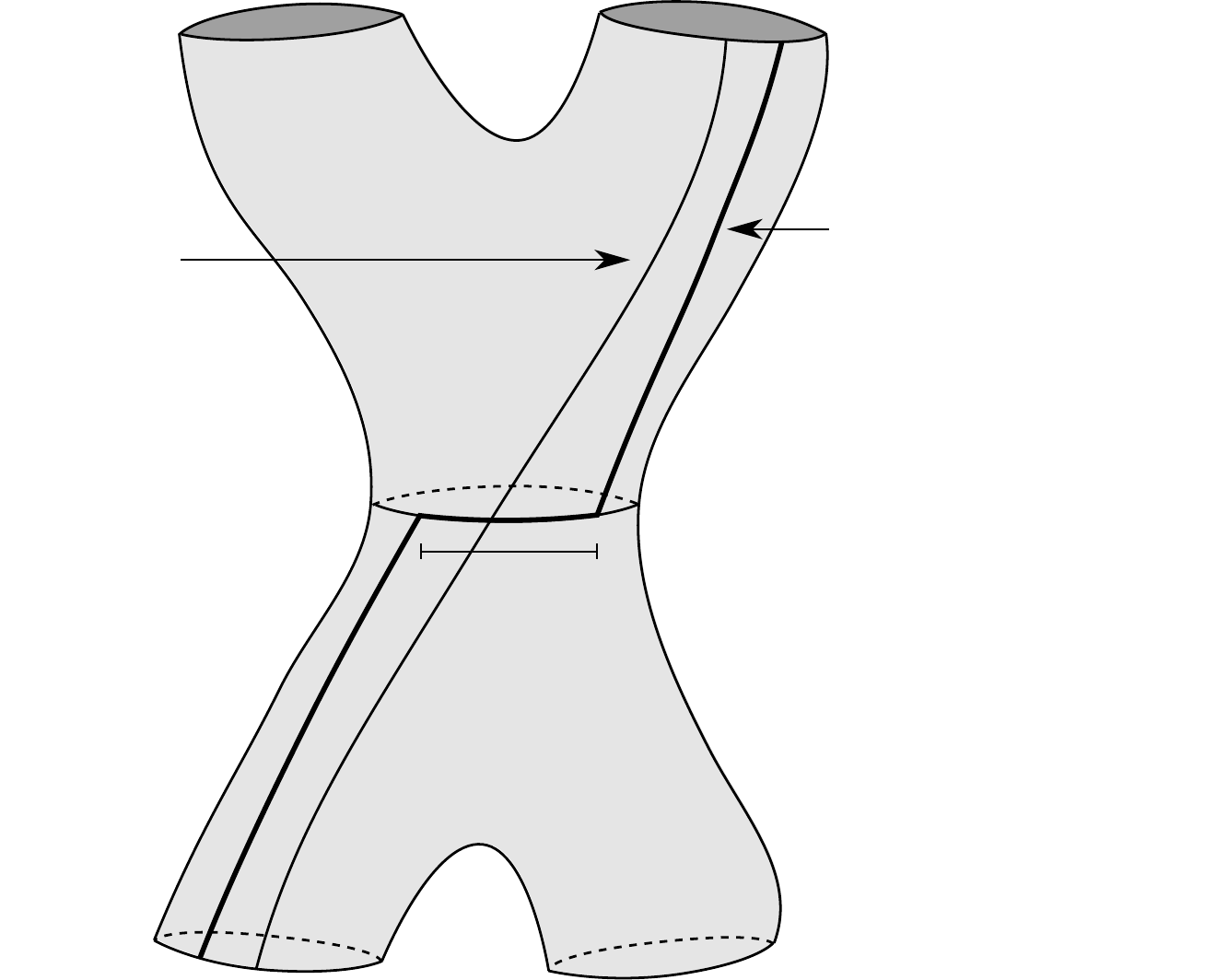
\end{figure}

\begin{figure}[htb!]
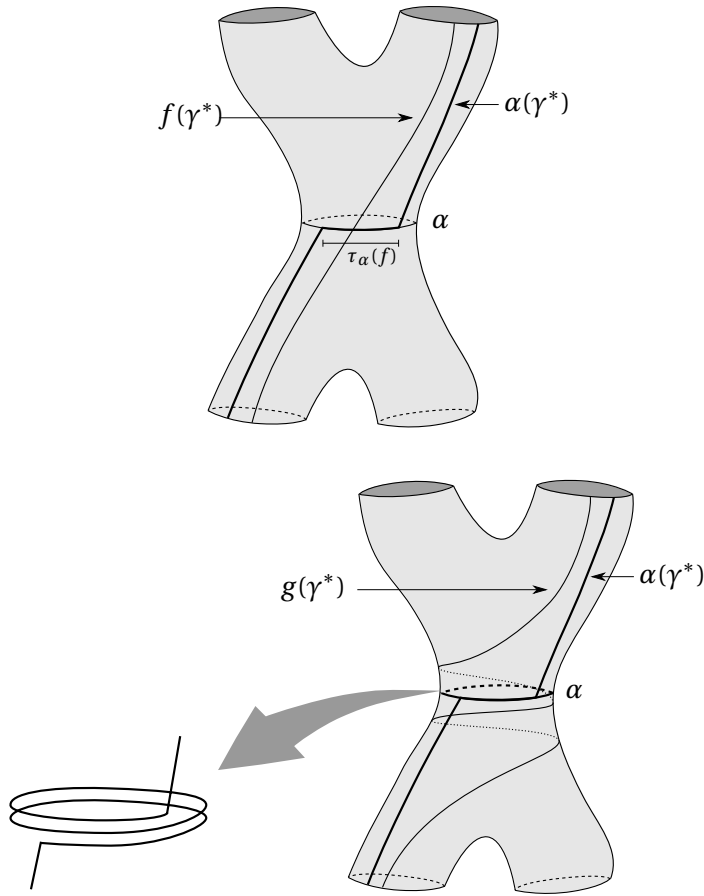
\caption{Concrete calculation of twist parameters.}\label{f.concrete-twist}
\end{figure}

In any case, it is possible to show the Fenchel-Nielsen coordinates $\mathcal{FN}_P$ associated to any pants decomposition $P$ is a \emph{global} homeomorphism (see, e.g., Theorem 7.6.3 in Hubbard's book \cite{Hubbard}). In particular, the Teichm\"uller space $Teich(S)$ is simply connected (as it is homeomorphic to $\mathbb{R}^{6g-6+2n}$). Hence, it is the orbifold universal cover of the moduli space $\mathcal{M}(S)$ (and the mapping class group $MCG(S)$ is the orbifold fundamental group of $\mathcal{M}(S)=Teich(S)/MCG(S)$).

This partly explain why one discusses the properties of $\mathcal{M}(S)$ and $Teich(S)$ at the same time.

\subsection{Cotangent bundle to moduli spaces of Riemann surfaces}

Another reason for studying $\mathcal{M}(S)$ and $Teich(S)$ together is because $Teich(S)$ is a manifold while $\mathcal{M}(S)$ is only an orbifold. In fact, the Teichm\"uller spaces $Teich(S)$ are \emph{real-analytic manifolds}. Indeed, the real-analytic structure on $Teich(S)$ comes from the uniformization theorem. More precisely, given a marked complex structure $f:S\to X$, we can apply the uniformization theorem to write $X=\mathbb{H}/\Gamma$ where $\Gamma\subset SL(2,\mathbb{R})$ is a discrete subgroup isomorphic to the fundamental group $\pi_1(S)$ of $S$. In other words, from a marked complex structure $f:S\to X$, we have a representation of $\pi_1(S)$ on $SL(2,\mathbb{R})$ (well-defined \emph{modulo conjugation}), and this permits to identify $Teich(S)$ with an open component of the \emph{character variety} of homomorphisms from $\pi_1(S)$ to $SL(2,\mathbb{R})$ modulo conjugacy. In particular, the pullback of the real-analytic structure of this representation variety to endow $Teich(S)$ with its own real-analytic structure.

Actually, as it turns out, this real-analytic structure of $Teich(S)$ can be ``upgraded'' to a \emph{complex-analytic structure}. One way of seeing this uses a ``generalization'' of the construction of the real-analytic structure above based on the complex-analytic structure on the representation variety of $\pi_1(S)$ in $SL(2,\mathbb{C})$ and \emph{Bers simultaneous uniformization theorem} \cite{Bers}. We will discuss this point later (in Section \ref{s.wp-geometry}) .

\begin{remark} This should be compared with the following ``toy model'' situation.

Let $E$ be a real vector space of dimension $2n$ and denote by $\mathcal{J}(E)$ the set of \emph{linear complex structures}\footnote{I.e., $\mathbb{R}$-linear maps $J:E\to E$ with $J^2=-Id$.} on $E$. It is possible to check that a linear complex structure on $E$ is \emph{equivalent} to the data of a complex subspace $K\subset \mathbb{C}\otimes_{\mathbb{R}} E$ of the complexification $\mathbb{C}\otimes_{\mathbb{R}}E$ of $E$ such that $\textrm{dim}_{\mathbb{C}}K=n$ and $K\cap \overline{K}=\{0\}$ (i.e., $\mathbb{C}\otimes_{\mathbb{R}} E = K\oplus\overline{K}$) where $\overline{K}$ is the complex conjugate of $K$.

Since the Grassmanian manifold $Gr_n(\mathbb{C}\otimes_{\mathbb{R}}E)$ of complex subspaces of $\mathbb{C}\otimes_{\mathbb{R}}E$ of complex dimension $n$ is naturally a complex manifold and the condition $K\cap \overline{K}=\{0\}$ is \emph{open} in $Gr_n(\mathbb{C}\otimes_{\mathbb{R}}E)$, we obtain that the set $\mathcal{J}(E)$ parametrizing complex structures on $E$ is itself a complex manifold.
\end{remark}

Let us now sketch the relationship between the \emph{quadratic differentials} on Riemann surfaces and the cotangent bundle to Teichm\"uller and moduli spaces.

\subsection{Integrable quadratic differentials}

The Teichm\"uller metric was defined via the notion of quasiconformal mappings $f:S_0\to S_1$. By inspecting the nature of this notion, we see that the quantities $k(f,x)=\frac{|\partial f(x)/\partial\overline z|}{|\partial f(x)/\partial z|}$ (related to the eccentricities of infinitesimal ellipses obtained as the images under $Df(x)$ of infinitesimal circles) play an important role in the definition of the Teichm\"uller distance between $S_0$ and $S_1$.

The \emph{measurable Riemann mapping theorem} of Alhfors and Bers (see, e.g., page 149 of Hubbard's book \cite{Hubbard}) says that the quasiconformal map $f$ can be \emph{recovered} from the quantities $k(f,x)$ \emph{up to composition with conformal maps}. More precisely, by collecting the quantities $k(f,x)$ in a globally defined \emph{tensor} of type $(-1,1)$
$$\mu(x)=\frac{(\partial f(x)/\partial\overline z)d\overline{z}}{(\partial f(x)/\partial z)dz}$$
with $\|\mu\|_{L^{\infty}}<1$ called \emph{Beltrami differential}, one can ``recover'' $f$ by solving \emph{Beltrami's equation}
$$(\partial f/\partial z) = \mu \cdot (\partial f/\partial\overline{z})$$
in the sense that there is always a solution to ths equation and, furthermore, two solutions $f$ and $g$ differ by a conformal map (i.e., $g=f\circ\phi$).

In other terms, the deformations of complex structures are intimately related to Beltrami differentials and it is not surprising that Beltrami differentials can be used to describe the tangent bundle of $Teich(S)$. In this setting, we can obtain the cotangent bundle $T^*Teich(S)$ by noticing that there is a natural pairing between bounded ($L^{\infty}$) Beltrami differentials $\mu$ and integrable ($L^1$) \emph{quadratic differentials} $q$ (i.e., a tensor of type $(2,0)$, $q=q(z)dz^2$):
$$\langle\mu, q\rangle=\int \mu q  = \int \mu(z)q(z)\frac{d\overline{z}}{dz}dz^2 = \int \mu(z)q(z)dz\,d\overline{z}$$
because $dz\,d\overline{z}$ is an area form and $\mu(z)q(z)$ is integrable. In this way, it can be shown that the cotangent space $T^*_XTeich(S)$ at a point $f:S\to X$ of $Teich(S)$ is naturally identified to the space $Q(X)$ of integrable quadratic differentials on $X$.

Note that the space of integrable quadratic differentials $Q(X)$ provides a concrete way of manipulating the complex structure of $Teich(S)$: in this setting, the complex structure is just the multiplication by $i$ on the space of quadratic differentials.

\begin{remark} By a theorem of Royden (see Hubbard's book), the mapping class group $MCG(S)$ is the group of complex-analytic automorphisms of $Teich(S)$. In particular, the moduli space $\mathcal{M}(S)=Teich(S)/MCG(S)$ is a complex orbifold.
\end{remark}

\subsection{Teichm\"uller and Weil-Petersson metrics}\label{ss.Teich-WP-metrics}

The description of the cotangent bundle of $Teich(S)$ in terms of quadratic differentials allows us to define the Teichm\"uller and Weil-Petersson metrics in the following way.

Given a point $f:S\to X$ of $Teich(X)$, we endow the cotangent space $T^*_XTeich(S)\simeq Q(X)$ with the $L^p$-norm:
$$\|\psi\|_{p}:=\left(\int \rho^{2-2p}|\psi|^p\right)^{1/p}$$
where $\rho$ is the hyperbolic metric associated to the conformal structure $X$ and $\psi$ is a quadratic differential (i.e., a tensor of type $(2,0)$).

\begin{remark} More generally, we define the $L^p$-norm of a tensor $\psi$ of type $(r,s)$ (i.e., $\psi=\psi(z)dz^r\,d\overline{z}^s$) as:
$$\|\psi\|_{p}:=\left(\int \rho^{2-p(r+s)}|\psi|^p\right)^{1/p}$$
\end{remark}

In this notation, the \emph{infinitesimal Teichm\"uller metric} is the family of $L^1$-norms on the fibers $T^*_X Teich(S)$ of the cotangent bundle of $Teich(S)$. Here, the nomenclature ``infinitesimal Teichm\"uller metric'' is justified by the fact that the ``global'' Teichm\"uller metric (defined by the infimum of the eccentricity factors $K(f)$ of quasiconformal maps $f:X_0\to X_1$) is the Finsler metric induced by the ``infinitesimal'' Teichm\"uller metric (see, e.g., Theorem 6.6.5 of Hubbard's book).

In a similar vein, the \emph{Weil-Petersson (WP) metric} is the family of $L^2$-norms on the fibers $T^*_X Teich(S)$ of the cotangent bundle of $Teich(S)$.

\begin{remark} In the definition of the WP metric, it was implicit that an integrable quadratic differential has finite $L^2$-norm (and, actually, all $L^p$-norms are finite, $1\leq p\leq\infty$). This fact is obvious when the $S$ is compact, but it requires a (simple) computation when $S$ has punctures. See, e.g., Proposition 5.4.3 of Hubbard's book for the details.
\end{remark}

For later use, we will denote the (infinitesimal) Teichm\"uller metric, resp., Weil-Petersson metric, as $\|.\|_T$, resp. $\|.\|_{WP}$.

The Teichm\"uller metric $\|.\|_{T}$ is a Finsler metric: the family of $L^1$-norms on the fibers of $T^*Teich(S)$ vary in a $C^1$ but not $C^2$ way (cf. Lemma 7.4.3 and Proposition 7.4.4 in Hubbard's book).

\begin{remark} The first derivative of the Teichm\"uller metric is not hard to compute. Given two cotangent vectors $p,q\in Q(X)$ with $\|q\|_{T}\neq 0$, we affirm that
$$D\|.\|_{T}(q)\cdot p=\int_X\textrm{Re}\left(\frac{\overline{q}}{|q|}p\right)$$
Indeed, the first derivative is $D\|.\|_{T}(q)\cdot p:=\lim\limits_{t\to 0}\frac{1}{t}\int_X (|q+tp|-|q|)$. Since $|q+tp|-|q|\leq t|p|$ and $p\in Q(X)$ is bounded (i.e., its $L^{\infty}$ norm is finite), we can use the dominated convergence theorem to obtain that
$$D\|.\|_{T}(q)\cdot p=\int_X\lim\limits_{t\to0}\frac{|q+tp|-|q|}{t} = \int_X\textrm{Re}\left(\frac{\overline{q}}{|q|}p\right)$$
\end{remark}

The Weil-Petersson metric $\|.\|_{WP}$ is induced by the \emph{Hermitian} inner product
$$\langle q_1, q_2\rangle_{WP} := \int_X \frac{\overline{q_1} q_2}{\rho^2}$$

As usual, the real part $g_{WP}:=\textrm{Re}\langle.,.\rangle_{WP}$ induces a \emph{real} inner product (also inducing the WP metric), while the imaginary part $\omega_{WP}:=\textrm{Im}\langle.,.\rangle_{WP}$ induces a \emph{symplectic form} (i.e., an anti-symmetric bilinear form).

By definition, the Weil-Petersson metric $g_{WP}$ relates to the Weil-Petersson symplectic form $\omega_{WP}$ and the complex structure $J$ on $Teich(S)$ (i.e., multiplication by $i$ of elements of $Q(X)$) via:
$$g_{WP}(q_1,q_2)=\omega_{WP}(q_1,Jq_2)$$

Furthermore, as it was firstly discovered by Weil \cite{Weil} by means of a ``simple-minded calculation'' (``\emph{calcul idiot}'') and later confirmed by others, it is possible to show that the Weil-Petersson metric  is \emph{K\"ahler}, i.e., the Weil-Petersson symplectic form $\omega_{WP}$ is closed (that is, its exterior derivative vanishes: $d\omega_{WP}=0$). See, e.g., Section 7.7 of Hubbard's book for more details.

We will come back later (in Section \ref{s.wp-geometry}) to the K\"ahler property of the WP metric, but for now let us just mention that this property enters into the proof of a beautiful theorem of Wolpert \cite{Wolpert1983} saying that the Weil-Petersson symplectic form has a simple expression in terms of Fenchel-Nielsen coordinates:
$$\omega_{WP}=\frac{1}{2}\sum\limits_{\alpha\in P} d\ell_{\alpha}\wedge d\tau_{\alpha}$$
where $P$ is an arbitrary pants decomposition of $S$. Here, it is worth to mention that an important step in the proof of this formula (cf. Step 2 in the proof of Theorem 7.8.1 in Hubbard's book \cite{Hubbard}) is the fact discovered by Wolpert that the infinitesimal generator $\partial/\partial\tau_{\alpha}$ of the Dehn twist about $\alpha$ is of the symplectic gradient of the Hamiltonian function $\frac{1}{2}\ell_{\alpha}$, that is,
$$\frac{1}{2}d\ell_{\alpha} = \omega_{WP}(.,\partial/\partial \tau_{\alpha}) \quad (i.e., \textrm{grad}\,\ell_{\alpha}=-2J(\partial/\partial\tau_{\alpha}))$$
This equation is the starting point of several Wolpert's expansion formulas for the Weil-Petersson metric that we will discuss later in this series of posts.

Before proceeding further, let us briefly discuss the Teichm\"uller and WP metrics on the moduli spaces of once-punctured torii $\mathcal{M}_{1,1}\simeq\mathbb{H}/SL(2,\mathbb{Z})$.

\begin{example}\label{ex.wolpert-asymptotics} The Teichm\"uller metric on $\mathcal{M}_{1,1}\simeq\mathbb{H}/SL(2,\mathbb{Z})$ is the quotient of the hyperbolic metric $\rho(z) = \frac{|dz|}{|\textrm{Im}(z)|}$ of $\mathbb{H}$.

On the other hand, the Fenchel-Nielsen coordinates $(\ell,\tau)$ on $Teich_{1,1}$ have first-order expansion
$$\ell(z)\sim\frac{1}{\textrm{Im}(z)}=\frac{1}{y} \quad \textrm{and} \quad \tau(z)\sim\frac{\textrm{Re}(z)}{\textrm{Im}(z)}=\frac{x}{y}$$
where $z=x+iy$. Thus, we see from Wolpert's formula that
$$\omega_{WP}=\frac{1}{2}d\ell\wedge d\tau\sim\left(-\frac{1}{y}dy\right)\wedge\left(\frac{1}{y}dx-\frac{x}{y^2}dy\right) = \frac{1}{y^3}dx\wedge dy =
\frac{1}{\textrm{Im}(z)^3}dz\wedge d\overline{z}.$$

Since the complex structure on $Teich_{1,1}$ is the standard complex structure of $\mathbb{H}$, we see that the Weil-Petersson metric $g_{WP}$ has asymptotic expansion
$$g_{WP}^2\sim \frac{|dz|^2}{\textrm{Im}(z)^3},$$
that is, the Weil-Petersson $g_{WP}$ on the moduli space $\mathcal{M}_{1,1}\simeq\mathbb{H}/SL(2,\mathbb{Z})$ near the cusp at infinity is modeled\footnote{Recall that, in general, a surface of revolution obtained by rotation of the curve $v=f(u)$ has the metric $g^2=(1+f'(u)^2)du^2+f(u)^2dv^2$.} by the \emph{surface of revolution} obtained by rotating the curve $v=u^3$ (for $0< u\leq 1$ say).
\begin{figure}[htb!]
\includegraphics[scale=0.5]{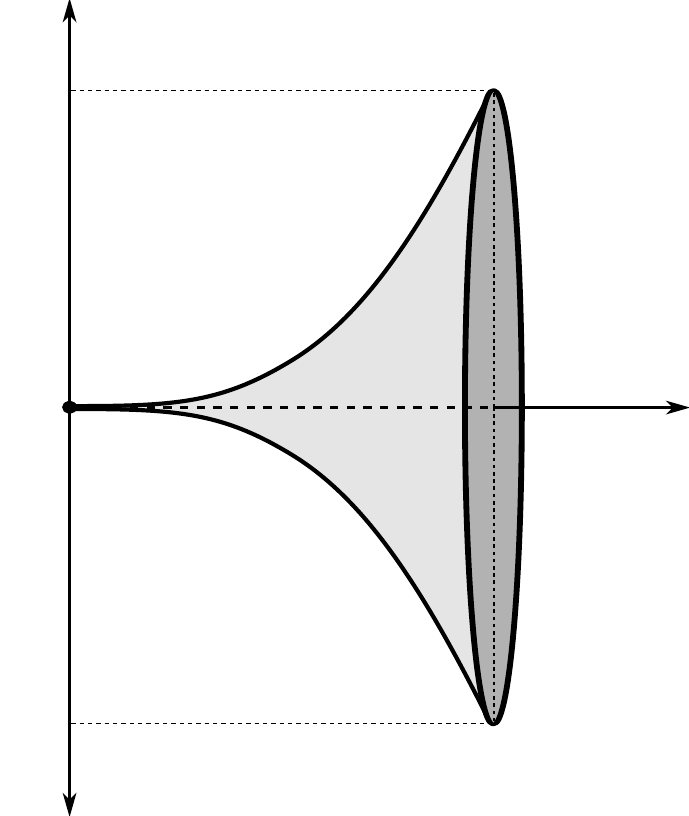}
\end{figure}

This is in contrast with the fact that the Teichm\"uller metric is the hyperbolic metric and hence it is modeled by surface of revolution obtained by rotation the curve $v=e^{-u}$ (for $1<x<\infty$ say).

From this asymptotic expansion of $g_{WP}$, we see that it is \emph{incomplete}: indeed, a vertical ray to the cusp at infinity starting at a point $z$ in the line $\textrm{Im}(z)=y_0$ has Weil-Petersson length $\sim 2y_0^{-1/2}\sim 2\ell(z)^{1/2}$. Moreover, the curvature $K$ satisfies $K(z)\sim -3/2\ell(z)$, and, in particular, $K\to-\infty$ as $\textrm{Im}(z)\to\infty$.
\end{example}

The previous example (WP metric on $\mathcal{M}_{1,1}$) already contains several features of the WP metric on \emph{general} moduli spaces $\mathcal{M}_{g,n}$. For example, we will see later that the Weil-Petersson metric is incomplete because it is possible to shrink a simple closed curve $\alpha$ to a point and leave Teichm\"uller space along a Weil-Petersson geodesic in time $\sim \ell_{\alpha}^{1/2}$. Also, some sectional curvatures might approach $-\infty$ as one leaves Teichm\"uller space.

Nevertheless, an interesting feature of the Weil-Petersson metric in $Teich_{g,n}$ and $\mathcal{M}_{g,n}$ for $3g-3+n>1$ \emph{not} occuring in the case of $\mathcal{M}_{1,1}$ is the fact that some sectional curvatures might also approach $0$ as one leaves Teichm\"uller space. Indeed, as we will see later, this happens because the ``boundary'' of $\mathcal{M}_{g,n}$ is sufficiently ``large'' when $3g-3+n>1$ so that it is possible form some Weil-Petersson geodesics to travel ``almost parallel'' to certain parts of the ``boundary'' for a certain time (while the same is \emph{not} possible for $\mathcal{M}_{1,1}$ because the ``boundary'' consists of a single point).

Concluding this subsection, let us mention that our main dynamical object in these notes -- the \emph{Weil-Petersson geodesic flow} -- is simply the geodesic flow induced by the WP metric on the unit cotangent bundle to $\mathcal{M}_{g,n}$.

\subsection{Ergodicity of WP flow: outline of proof revisited}\label{ss.ergodicity-orbifold-manifold}

By the end of Subsection \ref{s.BMW-outline} above, we mentioned that the proof of Burns-Masur-Wilkinson theorem of ergodicity of the WP geodesic flow (Theorem \ref{t.BMW}) can be essentially reduced to show that the WP metric satisfies the six conditions of Burns-Masur-Wilkinson ergodicity criterion for geodesic flows (Theorem \ref{t.BMW-ergodicity-criterion}).

Indeed, at first sight, it is tempting to say that Theorem \ref{t.BMW} follows from Theorem \ref{t.BMW-ergodicity-criterion} after checking items (I) to (VI) of the latter theorem for the case $M=T^1 Teich_{g,n}$ (the cotangent bundle of $Teich_{g,n}$), $N=T^1\mathcal{M}_{g,n}$ (the cotangent bundle of $\mathcal{M}_{g,n}$) and $\Gamma=MCG_{g,n}$ (the mapping class group).

However, a closer inspection of the statement of the ergodicity criterion (Theorem \ref{t.BMW-ergodicity-criterion}) reveals that this is not quite true: the moduli spaces $\mathcal{M}_{g,n}$ and their unit cotangent bundles $N=T^1\mathcal{M}_{g,n}$ are not \emph{manifolds} but only \emph{orbifolds}, while the ergodicity criterion (Theorem \ref{t.BMW-ergodicity-criterion}) assumes that the phase space $N$ of the geodesic flow is a manifold.

In other words, the orbifoldic nature of moduli spaces imposes a technical difficulty in the reduction of Theorem \ref{t.BMW} to Theorem \ref{t.BMW-ergodicity-criterion}. Fortunately, a solution to this technical issue is very well-known to algebraic geometers and it consists into taking an adequate \emph{finite} cover of the moduli space  in order to ``kill'' the orbifold points (i.e., points with large stabilizers for the mapping class group).

More precisely, for each $k\in\mathbb{N}$, one considers the following \emph{finite-index} subgroup of the mapping class group $MCG(S)$:
$$MCG(S)[k]=\{\phi\in MCG(S): \phi_* = 0 \textrm{ acting on } H_1(S,\mathbb{Z}/k\mathbb{Z})\}$$
where $\phi_*$ is the action on homology of $\phi$. Equivalently, an element $\phi$ of $MCG(S)$ belongs to $MCG(S)[k]$ whenever its action $\phi_*$ on the absolute homology group $H_1(S,\mathbb{Z})$ corresponds to a (symplectic) integral $2g\times 2g$ matrix congruent to the identity matrix modulo $k$.

\begin{example} In the case of once-punctured torii, the mapping class group is $MCG_{1,1}=SL(2,\mathbb{Z})$ and
$$MCG_{1,1}[k]=\left\{\left(\begin{array}{cc} a & b \\ c & d \end{array}\right)\in SL(2,\mathbb{Z}): a\equiv d\equiv 1 (\textrm{mod }k), b\equiv c\equiv 0 (\textrm{mod }k)  \right\}$$
In the literature, $MCG_{1,1}[k]$ is called the \emph{principal congruence subgroup} of $SL(2,\mathbb{Z})$ of level $k$.
\end{example}

\begin{remark} The index of $MCG_{g,n}[k]$ in $MCG_{g,n}$ can be computed explicitly. For instance, the natural map from $MCG_g$ to $Sp(2g,\mathbb{Z})$ is surjective (see, e.g., Farb-Margalit's book), so that the index of $MCG_g[k]$ is the cardinality of $Sp(2g,\mathbb{Z}/k\mathbb{Z})$, and, for $k=p$ prime, one has
$$\# Sp(2g,\mathbb{Z}/k\mathbb{Z}) = p^{g^2}(p^2-1)(p^4-1)\dots(p^{2g}-1)=p^{2g^2+g}+O(p^{2g^2+g-2}),$$
cf. Dickson's paper \cite{Dickson}.
\end{remark}

It was shown by Serre (see \cite{Serre} for the original proof or Farb-Margalit's book \cite{FarbMargalit} for an alternative exposition) that $MCG(S)[k]$ is torsion-free for $k\geq 3$ and, \emph{a fortiori}, it acts freely and properly discontinuous on $Teich(S)$ for $k\geq 3$. In other terms, the finite cover of $\mathcal{M}(S)=Teich(S)/MCG(S)$ given by
$$\mathcal{M}(S)[k]=Teich(S)/MCG(S)[k]$$
is a \emph{manifold} for $k\geq 3$.

\begin{remark} Serre's result is sharp: the principal congruence subgroup $MCG_{1,1}[2]$ of level $2$ of $SL(2,\mathbb{Z})$ contains the torsion element $-Id$.
\end{remark}

Once one disposes of an appropriate manifold $\mathcal{M}(S)[3]$ finitely covering the moduli space $\mathcal{M}(S)$, the reduction of Theorem \ref{t.BMW} to Theorem \ref{t.BMW-ergodicity-criterion} consists into two steps:
\begin{itemize}
\item[(a)] the verification of items (I) to (VI) in the statement of Theorem \ref{t.BMW-ergodicity-criterion} in the case of the unit cotangent bundle $N=T^1\mathcal{M}(S)[3]$ of $\mathcal{M}(S)[3]$.
\item[(b)] the deduction of the ergodicity (and mixing, Bernoullicity, and positivity and finiteness of metric entropy) of the Weil-Petersson geodesic flow on $T^1\mathcal{M}(S)$ from the corresponding fact(s) for the Weil-Petersson geodesic flow on $T^1\mathcal{M}(S)[3]$.
\end{itemize}

For the remainder of this section, we will discuss item (b) while leaving item (a) (i.e., items (I) to (VI) of Theorem \ref{t.BMW-ergodicity-criterion} for $N=T^1\mathcal{M}(S)[3]$) for the next section.

For ease of notation, we will denote $Teich(S)=\mathcal{T}$, $\mathcal{M}(S)=\mathcal{M}$ and $\mathcal{M}(S)[3]=\mathcal{M}[3]$. Assuming that the Weil-Petersson flow is ergodic (and Bernoulli, and its metric entropy is positive and finite) on $T^1\mathcal{M}[3]$, the ``obstruction'' to show the same fact(s) for the Weil-Petersson flow on $T^1\mathcal{M}$ is the possibility that the orbifold points of $\mathcal{M}$ form a ``large'' set.

Indeed, if we can show that the set of orbifold points of $\mathcal{M}$ is ``small'' (e.g., they form a set of zero measure), then the geodesic flow on $T^1\mathcal{M}[3]$ covers the geodesic flow on $T^1\mathcal{M}$ on a set of full measure. In particular, if $E$ is a (Weil-Petersson flow) invariant set of positive measure on $T^1\mathcal{M}$, then its lift $\widetilde{E}$ to $T^1\mathcal{M}[3]$ is also a (Weil-Petersson flow) invariant set of positive measure. Therefore, by the ergodicity of the Weil-Petersson flow on $T^1\mathcal{M}[3]$, we have that $\widetilde{E}$ has full measure, and, \emph{a fortiori}, $E$ has full measure. Moreover, the fact that the Weil-Petersson flow on $T^1\mathcal{M}[3]$ covers the Weil-Petersson flow on $T^1\mathcal{M}$ on a full measure set also allows to deduce Bernoullicity and positivity and finiteness of metric entropy of the latter flow from the corresponding properties for the former flow.

At this point, this subsection is complete once we check that the orbifold points of $\mathcal{M}(S)$ form a subset of zero measure (for the Liouville/volume measure of the Weil-Petersson metric). This is an immediate consequence of the following lemma:

\begin{lemma}
Let $F$ be the subset of $Teich(S)$ corresponding to orbifoldic points, i.e., $F$ is the (countable) union of the subsets $F(h)$ of fixed points of the natural action on $Teich(S)$ of all elements $h\in MCG(S)$ of finite order, \textbf{excluding} the genus $2$ hyperelliptic involution. Then, $F$ is a closed subset of real codimension $\geq 2$.
\end{lemma}

\begin{proof} For each $h\in MCG(S)$ of finite order, $F(h)$ is the Teichm\"uller space of the quotient orbifold $X/\langle h\rangle$. From this, one can show that:
\begin{itemize}
\item if $S$ is compact and $h$ is not the hyperelliptic involution in genus $2$, then $F(h)$ has complex dimension $\leq 3g-5$;
\item if $S$ has punctures, then $F(h)$ has complex dimension $\leq 3g-4$;
\item if $h$ is the hyperelliptic involution in genus $2$, then $F(h)=Teich(S)$.
\end{itemize}
See, e.g., Rauch's paper \cite{Rauch} for more details.

In particular, the proof of the lemma is complete once we verify that $F$ is a \emph{locally finite} union of the real codimension $\geq 2$ subsets $F(h)$, $h\in MCG(S)$.

Keeping this goal in mind, we fix a compact subset $K$ of $Teich(S)$ and we recall that the mapping class group $MCG(S)$ acts in a properly discontinuous manner on $Teich(S)$. Therefore, it is not possible for an infinite sequence $(h_n)_{n\in\mathbb{N}}\subset MCG(S)$ of distinct finite order elements to satisfy $F(h_n)\cap K\neq\emptyset$ for all $n\in\mathbb{N}$. In other words, $F\cap K$ is the subset of finitely many $F(h)$, i.e., $F$ is a locally finite union of $F(h)$, $h\in MCG(S)$.
\end{proof}

\begin{example} In the case of once-punctured torii, the subset $F\subset Teich_{1,1}$ consists of the $SL(2,\mathbb{Z})$-orbits of the points $i\in\mathbb{H}$ and $j=\exp(2\pi i/3)\in\mathbb{H}$.
\end{example}

\section{Geometry of the Weil-Petersson metric}\label{s.wp-geometry}

This section is devoted to the verification of items (I) to (VI) of Burns-Masur-Wilkinson ergodicity criterion (Theorem \ref{t.BMW-ergodicity-criterion}) in the context of the Weil-Petersson metric on $Teich(S)$ and $\mathcal{M}(S)[3]$. In other terms, as it was explained in Subsection \ref{ss.ergodicity-orbifold-manifold} above, this section covers (some of) the Teichm\"uller-theoretical aspects of the proof of Burns-Masur-Wilkinson theorem on the ergodicity of the WP geodesic flow on moduli spaces (Theorem \ref{t.BMW}) \emph{assuming} Burns-Masur-Wilkinson  ergodicity criterion (Theorem \ref{t.BMW-ergodicity-criterion}).

\subsection{Items (I) and (II) of Theorem \ref{t.BMW-ergodicity-criterion} for WP metric}

The item (I) in the statement of Theorem \ref{t.BMW-ergodicity-criterion} in the context of the Weil-Petersson metric (i.e., the geodesic convexity of the WP metric on $Teich(S)$) was proved by Wolpert \cite{Wolpert2008}, but we will not attempt to discuss this topic here (for the sake of making comments on other aspects of the geometry of WP metric).

Next, let us discuss the item (II) of Theorem \ref{t.BMW-ergodicity-criterion} in the context of the WP metric, that is, the compactness of the metric completions of moduli spaces $\mathcal{M}(S)$ equipped with WP metrics.

We start by recalling that the metric completion of the Teichm\"uller space $Teich(S)$ with respect to the WP metric was determined by Masur \cite{Masur}. Indeed, Masur exploited the fact that we can \emph{leave} $Teich(S)$ along a WP geodesic in \emph{finite time} of order $\sim\ell_{\alpha}^{1/2}$ by pinching a closed geodesic $\alpha$ of hyperbolic length $\ell_{\alpha}$ to show that the WP metric completion of the $Teich(S)$ is the so-called \emph{augmented Teichm\"uller space} $\overline{Teich}(S)$.

The augmented Teichm\"uller space $\overline{Teich}(S)$ is a stratified space obtained by adjoining lower-dimensional Teichm\"uller spaces of \emph{noded} Riemann surfaces. The combinatorial structure of the stratification of $\overline{Teich}(S)$ is encoded by the \emph{curve complex} $\mathcal{C}(S)$ (sometimes also called \emph{complex of curves} or \emph{graph of curves}).

More precisely, the curve complex $\mathcal{C}(S)$ is a $(3g-4+n)$-simplicial complex defined as follows. The vertices of $\mathcal{C}(S)$ are homotopy classes of homotopically non-trivial, non-peripheral, simple closed curves on $S$. We put an edge between two vertices whenever the corresponding homotopy classes have \emph{disjoint} representatives. In general, a $k$-simplex $\sigma\in\mathcal{C}(S)$ consists of $k+1$ distinct vertices possessing mutually disjoint representatives.

\begin{remark} $\mathcal{C}(S)$ is a $(3g-4+n)$-simplicial complex because a maximal collection $P$ of distinct vertices possessing disjoint representatives is a pants decomposition of $S$ and, hence, $\#P=3g-3+n$.
\end{remark}

\begin{example} In the case of once-punctured torii, the curve complex $\mathcal{C}(S)$ consists of an infinite discrete set of vertices (because there is no pair of disjoint homotopically distinct curves).
However, some authors define the curve complex $\mathcal{C}(S)$ of once-punctured torii by putting an edge between vertices corresponding to curves intersecting minimally (i.e., only once). In this alternative setting, the curve complex of once-punctured torii becomes the \emph{Farey graph}.
\end{example}

The curve complex $\mathcal{C}(S)$ is a connected locally infinite complex, except for the cases $(g,n)=(0,4)$ or $(1,1)$. Also, the mapping class group $MCG(S)$ naturally acts on $\mathcal{C}(S)$. Moreover, Masur-Minsky \cite{MasurMinsky} showed that $\mathcal{C}(S)$ is a $\delta$-\emph{hyperbolic metric space} for some $\delta=\delta(S)>0$.

Using the curve complex $\mathcal{C}(S)$, we can define the augmented Teichm\"uller space $\overline{Teich}(S)$ as follows.

A \emph{noded Riemann surface} is a compact topological surface equipped with the structure of a complex space with at most isolated singularities called \emph{nodes} such that each of these singularities possess a neighborhood biholomorphic to a neighborhood of $(0,0)$ in the singular curve
$$\{(z,w)\in\mathbb{C}^2: zw=0\}$$

Removing the nodes of a noded Riemann surface $Y$ yields to a possibly disconnected Riemann surface denoted by $\widehat{Y}$. The connected components of $\widehat{Y}$ are called the \emph{pieces} of $Y$.

For example, the noded Riemann surface of genus $g$ of the figure below has two pieces (of genera $g-1$ and 1 resp.).

\begin{figure}[htb!]
\includegraphics[scale=0.6]{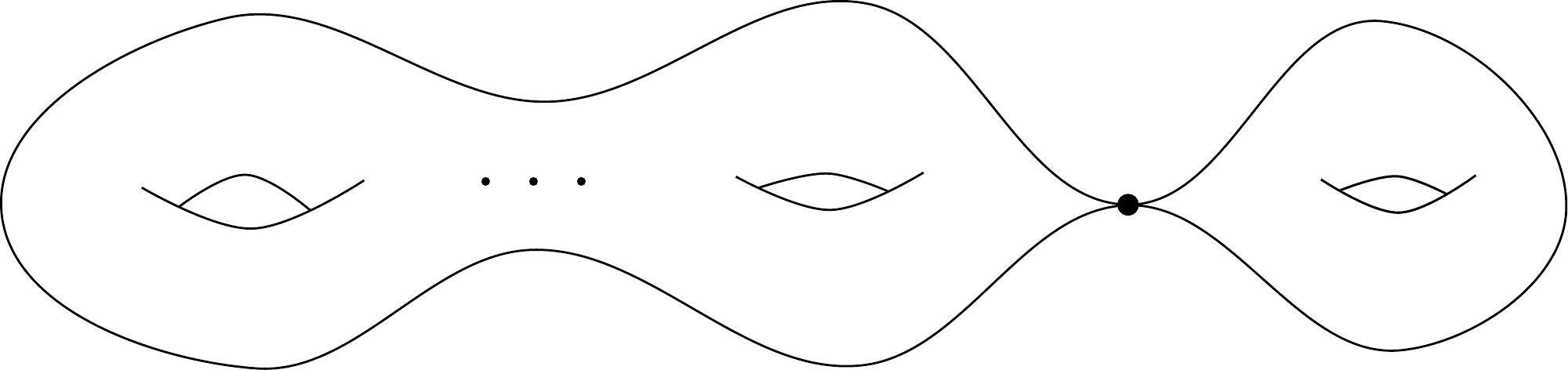}
\end{figure}

Given a simplex $\sigma\in\mathcal{C}(S)$, we will adjoint a Teichm\"uller space $\mathcal{T}_{\sigma}$ to $Teich(S)$ in the following way. A \emph{marked noded Riemann surface} with nodes at $\sigma$ is a noded Riemann surface $X_{\sigma}$ equipped with a continuous map $f:S\to X_{\sigma}$ such that the restriction of $f$ to $S-\sigma$ is a homeomorphism to $\widehat{X_{\sigma}}$. We say that two marked noded Riemann surfaces $f:S\to X_{\sigma}^1$ and $g:S\to X_{\sigma}^2$ are \emph{Teichm\"uller equivalent} if there exists a biholomorphic node-preserving map $h:X_{\sigma}^1\to X_{\sigma}^2$ such that $f\circ h$ is isotopic to $g$. The Teichm\"uller space $\mathcal{T}_{\sigma}$ associated to $\sigma$ is the set of Teichm\"uller equivalence classes $f:S\to X_{\sigma}$ marked noded Riemann surfaces with nodes at $\sigma$.

In this context, the augmented Teichm\"uller space is
$$\overline{Teich}(S)=Teich(S)\cup\bigcup\limits_{\sigma\in\mathcal{C}(S)}\mathcal{T}_{\sigma}$$

The topology on $\overline{Teich}(S)$ is given by the following neighborhoods of points $f:S\to X_{\sigma}$. Given $\sigma\in\mathcal{C}(S)$, we consider $P$ a maximal simplex (pants decomposition of $S$) containing $\sigma$ and we let $(\ell_{\alpha},\tau_{\alpha})_{\alpha\in P}$ be the corresponding Fenchel-Nielsen coordinates on $Teich(S)$. We \emph{extend} these coordinates by allowing $\ell_{\alpha}=0$ whenever $\alpha$ is pinched in a node and we take the quotient by identifying noded Riemann surfaces corresponding to parameters $(\ell_{\alpha},\tau_{\alpha})=(0,t)$ and $(\ell_{\alpha},\tau_{\alpha})=(0,t')$ whenever $\alpha\in\sigma$.

\begin{remark} The augmented Teichm\"uller space $\overline{Teich}(S)$ is not locally compact: indeed, a neighborhood of a noded Riemann surface allows for arbitrary twists $\tau_{\alpha}$ corresponding to curves $\alpha\in\sigma$.
\end{remark}

The quotient of $\overline{Teich}(S)$ by the natural action of $MCG(S)$ (through the corresponding action on $\mathcal{C}(S)$) is the so-called \emph{Deligne-Mumford compactification} $\overline{\mathcal{M}}(S)=\overline{Teich}(S)/MCG(S)$ of the moduli space of $\mathcal{M}(S)$. The space $\overline{\mathcal{M}}(S)$ was originally introduced by Deligne-Mumford \cite{DeligneMumford} and, as the nomenclature suggests, $\overline{\mathcal{M}}(S)$ is compact (see also Hubbard-Koch's paper \cite{HubbardKoch} for more details).

Since $\overline{Teich}(S)$ is the metric completion of $Teich(S)$ with respect to the WP metric and $MCG(S)[k]$ is a \emph{finite-index} subgroup of $MCG(S)$, it follows from the compactness of $\overline{\mathcal{M}}(S)$ that the the metric completion $\overline{Teich}(S)/MCG(S)[k]$ of $\mathcal{M}(S)[k]$ with respect to the WP metric is also compact (because it is a \emph{finite} cover of $\overline{\mathcal{M}}(S)$).

In particular, $\mathcal{M}(S)[3]$ satisfies the item (II) in the statement of Theorem \ref{t.BMW-ergodicity-criterion}.

\begin{remark} It is worth to notice that the Deligne-Mumford compactification in the case of the once-punctured torii is just one point\footnote{Because geometrically by pinching one curve in a punctured torus we get a thrice-punctured sphere in the limit and the moduli space of thrice-punctured spheres is trivial (cf. Example \ref{ex.M03}).} while it is stratified in non-trivial lower-dimensional moduli spaces in general. Moreover, as we will see later, some asymptotic formulas of Wolpert tells that the WP metric ``looks'' like a product of the WP metrics on these lower-dimensional moduli spaces.

In particular, as we will discuss in the last section of this text, some WP geodesics to travel ``almost parallel'' to these lower-dimensional moduli spaces for a long time and this will give a polynomial rate of mixing for this flow in general. On the other hand, since it is not possible to travel almost parallel to a point for a long time, this arguments breaks down in the case of the WP metric in the case of the moduli space of once-punctured torii.
\end{remark}

\subsection{Item (III) of Theorem \ref{t.BMW-ergodicity-criterion} for WP metric}

Let us now quickly check that $\mathcal{M}(S)[3]$ also satisfies the item (III) in the statement of Theorem \ref{t.BMW-ergodicity-criterion}, i.e., its boundary $\partial\mathcal{M}(S)[3]$ is volumetrically cusp-like.

In this direction, given $X\in Teich(S)$, let us denote by $\rho_0(X)$ the Weil-Petersson distance between $X$ and $\partial Teich(S):=\overline{Teich}(S)-Teich(S)$. Our current task is to prove that there are constants $C>0$ and $\nu>0$ such that
$$\textrm{vol}(E_{\rho})\leq C\rho^{2+\nu}$$
where $E_{\rho}:=\{X\in Teich(S)/MCG(S)[3]: \rho_0(X)\leq \rho\}$.

As we are going to see now, one can actually take $\nu=2$ in the estimate above thanks to some asymptotic formulas of Wolpert for the Weil-Petersson metric near the boundary $\partial\mathcal{T}=\bigcup\limits_{\sigma\in\mathcal{C}(S)}\mathcal{T}_{\sigma}$ of augmented Teichm\"uller space.

\begin{lemma}\label{l.wp-DM-bdry-nbhd-vol} One has $\textrm{vol}(E_{\rho})\simeq \rho^4$.
\end{lemma}

\begin{proof} It was shown by Wolpert (in page 284 of \cite{Wolpert2008}) that the Weil-Petersson metric $g_{WP}$ has asymptotic expansion
$$g_{WP}\sim \sum\limits_{\alpha\in\sigma} (4\, dx_{\alpha}^2 + x_{\alpha}^6 d\tau_{\alpha}^2)$$
near $\mathcal{T}_{\sigma}$, where $x_{\alpha}=\ell_{\alpha}^{1/2}/\sqrt{2\pi^2}$ and $\ell_{\alpha}$, $\tau_{\alpha}$ are the Fenchel-Nielsen coordinates associated to $\alpha\in\sigma$.

This gives that the volume element $\sqrt{\det(g_{WP})}$ of the Weil-Petersson metric near $\mathcal{T}_{\sigma}$ is $\sim\prod\limits_{\alpha\in\sigma} x_{\alpha}^3$. Furthermore, this aymptotic expansion of $g_{WP}$ also says that the distance $\rho_0(X)$ between $X$ and $\mathcal{T}_{\sigma}$ is comparable to $\min_{\alpha\in\sigma} x_{\alpha}(X)$. By putting these two facts together, we see that
$$\textrm{vol}(E_{\rho})\simeq \rho^4$$
This proves the lemma.
\end{proof}

\begin{remark} The properties that $\overline{\mathcal{M}}(S)$ is compact and $\mathcal{M}(S)$ is volumetrically cusp-like imply that the Liouville measure (volume) is finite.

Recently, Mirzakhani \cite{Mirzakhani2013} studied the total mass $V_{g,n}$ of $\mathcal{M}(S)$ with respect to the WP metric and she showed that there exists a constant $M>0$ such that
$$g^{-M}\leq \frac{V_{g,n}}{(4\pi^2)^{2g+n-3}(2g+n-3)!}\leq g^M$$
\end{remark}

\subsection{Item (IV) of Theorem \ref{t.BMW-ergodicity-criterion} for WP metric}

Recall that the item (IV) of Theorem \ref{t.BMW-ergodicity-criterion} asks for polynomial bounds in the sectional curvatures and their first two derivatives.

In the context of the Weil-Petersson (WP) metric, the desired polynomial bounds on the sectional curvatures themselves follow from the work of Wolpert.

\subsubsection{Wolpert's formulas for the curvatures of the WP metric}

We will give now a \emph{compte rendu} of some estimates of Wolpert for the behavior of the WP metric near the boundary $\partial \mathcal{T}$ of the Teichm\"uller space $\mathcal{T}=Teich(S)$.

Before stating Wolpert's formulas, we need an \emph{adapted} system of coordinates (called \emph{combined length basis} in the literature) near the strata $\mathcal{T}_{\sigma}$, $\sigma\in\mathcal{C}(S)$, of $\partial\mathcal{T}$, where $\mathcal{C}(S)$ is the curve complex of $S$.

Denote by $\mathcal{B}$ the set of pairs (``basis'') $(\sigma,\chi)$ where $\sigma\in\mathcal{C}(S)$ is a simplex of the curve complex and $\chi$ is a collection of simple closed curves such that each $\beta\in\chi$ is disjoint from all $\alpha\in\sigma$. Here, we allow that two curves $\beta, \beta'\in\chi$ \emph{intersect} (i.e., one might have $\beta\cap\beta'\neq\emptyset$) and also the case $\chi=\emptyset$ is \emph{not} excluded.

Following the nomenclature introduced by Wolpert, we say that $(\sigma,\chi)\in\mathcal{B}$ is a \emph{combined length basis} at a point $X\in\mathcal{T}$ whenever the set of tangent vectors
$$\{\lambda_{\alpha}(X), J\lambda_{\alpha}(X), \textrm{grad}\ell_{\beta}(X)\}_{\alpha\in\sigma, \beta\in\chi}$$
is a basis of $T_X\mathcal{T}$, where $\ell_{\gamma}$ is the length parameter in the Fenchel-Nielsen coordinates and $\lambda_{\alpha}:=\textrm{grad} \ell_{\alpha}^{1/2}$.

\begin{remark} The length parameters $\ell_{\gamma}$ and their square-roots $\ell_{\gamma}^{1/2}$ are natural for the study of the WP metric: for instance, Wolpert showed that these functions are convex along WP geodesics (see, e.g., Wolpert \cite{Wolpert2008}, \cite{Wolpert2009a} and Wolf \cite{Wolf2012}).
\end{remark}

The name \emph{combined length basis} comes from the fact that we think of $(\sigma,\chi)$ as a combination of a collection $\sigma\in\mathcal{C}(S)$ of \emph{short} curves (indicating the boundary stratum that one is close to), and a collection $\chi$ of \emph{relative} curves to $\sigma$ allowing to complete the set $\{\lambda_{\alpha}\}_{\alpha\in\sigma}$ into a basis of the tangent space to $\mathcal{T}$ in which one can write nice formulas for the WP metric.

This notion can be ``extended'' to a stratum $\mathcal{T}_{\sigma}$ of $\mathcal{T}$ as follows. We say $\chi$ is a \emph{relative basis} at a point $X_{\sigma}\in\mathcal{T}_{\sigma}$ whenever $(\sigma,\chi)\in\mathcal{B}$ and the length parameters $\{\ell_{\beta}\}_{\beta\in\chi}$ is a \emph{local} system of coordinates for $\mathcal{T}_{\sigma}$ near $X_{\sigma}$.

\begin{remark} The stratum $\mathcal{T}_{\sigma}$ is (isomorphic to) a product of the Teichm\"uller spaces of the pieces of $X_{\sigma}\in\mathcal{T}_{\sigma}$. In particular, $\mathcal{T}_{\sigma}$ carries a ``WP metric'', namely, the product of the WP metrics on the Teichm\"uller spaces of the pieces of $X_{\sigma}$. In this setting, $\chi$ is a relative basis at $X_{\sigma}\in\mathcal{T}_{\sigma}$ if and only if $\{\textrm{grad} \ell_{\beta}\}_{\beta\in\chi}$ is a basis of $T_{X_{\sigma}}\mathcal{T}_{\sigma}$.
\end{remark}

\begin{remark} Contrary to the Fenchel-Nielsen coordinates, the length parameters $\{\ell_{\beta}\}_{\beta\in\chi}$ associated to a relative basis $\chi$ might not be a \emph{global} system of coordinates for $\mathcal{T}_{\sigma}$. Indeed, this is so because we allow the curves in $\chi$ to intersect non-trivially: geometrically, this means that there are points $X_0$ in $\mathcal{T}_{\sigma}$ where the geodesic representatives of such curves meet orthogonally, and, at such points, the system of coordinates induced by $\{\ell_{\beta}\}_{\beta\in\chi}$ hits a singularity.
\end{remark}

The relevance of the concept of combined length basis to the study of the WP metric is explained by the following theorem of Wolpert \cite{Wolpert2008}:

\begin{theorem}[Wolpert]\label{t.Wolpert2008} For any point $X_{\sigma}\in\mathcal{T}_{\sigma}$, $\sigma\in\mathcal{C}(S)$ , there exists a relative length basis $\chi$. Furthermore, the WP metric $\langle.,.\rangle_{WP}$ can be written as
$$\langle.,.\rangle_{WP}\sim \sum\limits_{\alpha\in\sigma}\left((d\ell_{\alpha}^{1/2})^2+(d\ell_{\alpha}^{1/2}\circ J)^2\right) + \sum\limits_{\beta\in\chi}(d\ell_{\beta})^2$$
where the implied comparison constant is uniform in a neighborhood $U\subset\overline{\mathcal{T}}$ of $X_{\sigma}$.

In particular, there exists a neighborhood $V\subset\overline{\mathcal{T}}$ of $X_{\sigma}$ such that $(\sigma,\chi)$ is a combined length basis at any $X\in V\cap\mathcal{T}$.
\end{theorem}

The statement above is just the beginning of a series of formulas of Wolpert for the WP metric and its sectional curvatures written in terms of the local system of coordinates induced by a combined length basis $(\sigma,\chi)$.

In order to write down the next list of formulas of Wolpert, we need the following notations. Given $\mu$ an arbitrary collection of simple closed curves on $S$, we define
$$\underline{\ell}_{\mu}(X):=\min\limits_{\alpha\in\mu}\ell_{\alpha}(X)\quad \textrm{and}\quad \overline{\ell}_{\mu}(X):=\max\limits_{\alpha\in\mu}\ell_{\alpha}(X)$$
where $X\in\mathcal{T}=Teich(S)$. Also, given a constant $c>1$ and a basis $(\sigma,\chi)\in\mathcal{B}$, we will consider the following (Bers) region of Teichm\"uller space:
$$\Omega(\sigma,\chi,c):=\{X\in\mathcal{T}: 1/c<\underline{\ell}_{\chi}(X) \textrm{ and } \overline{\ell}_{\sigma\cup\chi}(X)<c\}$$

Wolpert \cite{Wolpert2009} provides several estimates for the WP metric $\langle.,.\rangle_{WP}=\langle.,.\rangle$ and its sectional curvatures in terms of the basis $\lambda_{\alpha}=\textrm{grad}\ell_{\alpha}^{1/2}$, $\alpha\in\sigma$ and $\textrm{grad}\ell_{\beta}$, $\beta\in\chi$, which are uniform on the regions $\Omega(\sigma,\chi,c)$.

\begin{theorem}[Wolpert]\label{t.Wolpert-expansions} Fix $c>1$. Then, for any $(\sigma,\chi)\in\mathcal{B}$, and any $\alpha,\alpha'\in\sigma$ and $\beta,\beta'\in\chi$, the following estimates hold uniformly on $\Omega(\sigma,\chi,c)$
\begin{itemize}
\item $\langle\lambda_{\alpha},\lambda_{\alpha'}\rangle = \frac{1}{2\pi}\delta_{\alpha,\alpha'}+O((\ell_{\alpha}\ell_{\alpha'})^{3/2}) = \langle J\lambda_{\alpha}, J\lambda_{\alpha'}\rangle$
where $\delta_{\ast,\ast\ast}$ is Kronecker's delta.
\item $\langle\lambda_{\alpha}, J\lambda_{\alpha'}\rangle=\langle J\lambda_{\alpha}, \textrm{grad} \ell_{\beta}\rangle=0$
\item $\langle\textrm{grad} \ell_{\beta}, \textrm{grad} \ell_{\beta'}\rangle\sim 1$
and, furthermore, $\langle\textrm{grad} \ell_{\beta}, \textrm{grad} \ell_{\beta'}\rangle$ extends continuosly to the boundary stratum $\mathcal{T}_{\sigma}$.
\item $\langle\lambda_{\alpha},\textrm{grad}\ell_{\beta}\rangle = O(\ell_{\alpha}^{3/2})$
\item the distance from $X\in\Omega(\sigma,\chi,c)$ to the boundary stratum $\mathcal{T}_{\sigma}$ is
$$d(X,\mathcal{T}_{\sigma}) = \sqrt{2\pi\sum\limits_{\alpha\in\sigma}\ell_{\alpha}(X)} + O\left(\sum\limits_{\alpha\in\sigma}\ell_{\alpha}^{5/2}(X)\right)$$
\item for any vector $v\in T\Omega(\sigma,\chi,c)$,
$$\left\|\nabla_v \lambda_{\alpha} - \frac{3}{2\pi\ell_{\alpha}^{1/2}}\langle v, J\lambda_{\alpha}\rangle J\lambda_{\alpha}\right\|_{WP} =
O(\ell_{\alpha}^{3/2}\|v\|_{WP})$$
\item $\|\nabla_{\lambda_{\alpha}}\textrm{grad} \ell_{\beta}\|_{WP} = O(\ell_{\alpha}^{1/2})$ and
$\|\nabla_{\lambda_{\alpha}}\textrm{grad} \ell_{\beta}\|_{WP}=O(\ell_{\alpha}^{1/2})$
\item $\nabla_{\textrm{grad}\ell_{\beta}}\textrm{grad}\ell_{\beta'}$ extends continuously to the boundary stratum
$\mathcal{T}_{\sigma}$
\item the sectional curvature of the complex line (real two-plane) $\{\lambda_{\alpha}, J\lambda_{\alpha}\}$ is
$$\langle R(\lambda_{\alpha}, J\lambda_{\alpha})J\lambda_{\alpha}, \lambda_{\alpha}\rangle = \frac{3}{16\pi^2\ell_{\alpha}} + O(\ell_{\alpha})$$
\item for any quadruple $(v_1, v_2, v_3, v_4)$, $v_i\in\{\lambda_{\alpha}, J\lambda_{\alpha}, \textrm{grad}\ell_{\beta}\}_{\alpha\in\sigma, \beta\in\chi}$ distinct from a curvature-preserving permutation of $(\lambda_{\alpha}, J\lambda_{\alpha}, J\lambda_{\alpha}, \lambda_{\alpha})$, one has
$$\langle R(v_1,v_2)v_3, v_4\rangle=O(1),$$
and, moreover, each $v_i$ of the form $\lambda_{\alpha}$ or $J\lambda_{\alpha}$ introduces a multiplicative factor $O(\ell_{\alpha})$ in the estimate above.
\end{itemize}
\end{theorem}

These estimates of Wolpert give a very good understanding of the geometry of the WP metric in terms of combined length basis. For instance, one infers from the last two items above that, as one approaches the boundary stratum $\mathcal{T}_{\sigma}$, the sectional curvatures of the WP metric along the complex lines $\{\lambda_{\alpha}, J\lambda_{\alpha}\}$ converge to $-\infty$ with speed $\sim -\ell_{\alpha}^{-1}\sim -d(X,\mathcal{T}_{\sigma})^{-2}$, while the sectional curvatures of the WP metric associated to quadruples of the form $(\lambda_{\alpha}, J\lambda_{\alpha}, J\lambda_{\alpha'}, \lambda_{\alpha'})$ with $\alpha, \alpha'\in\sigma$, $\alpha\neq\alpha'$, converge to $0$ with speed $\sim O(\ell_{\alpha}^2 \ell_{\alpha'}^2) = O(d(X,\mathcal{T}_{\sigma})^8)$ \emph{at least}.

In particular, these formulas of Wolpert allow to show ``one third of item (IV) of Theorem \ref{t.BMW-ergodicity-criterion}'' for the WP metric, that is,
\begin{equation}\label{e.(IV)-1/3-WP}
\|R_{WP}(x)\|_{WP}\leq C d(x,\partial\mathcal{T})^{-2}
\end{equation}
for all $x\in\mathcal{T}$.

\begin{remark} Observe that the formulas of Wolpert provide \emph{asymmetric} information on the sectional curvatures of the WP metric: indeed, while we have precise estimates on how these sectional curvarutures can approach $-\infty$, the same is not true for the sectional curvatures approaching zero (where one disposes of lower bounds but no upper bounds for the speed of convergence).
\end{remark}

\begin{remark}\label{r.wp-almost-zero-curvatures} From the discussion above, we see that there are sectional curvatures of the WP metric on $Teich(S)$ approaching zero whenever $\sigma\in\mathcal{C}(S)$ contains two distinct curves. In other words, the WP metric has sectional curvatures approaching zero whenever the genus $g$ and the number of punctures $n$ of $S=S_{g,n}$ satisfy $3g-3+n>1$, i.e., except in the cases of once-punctured torii $S_{1,1}$ and four-times puncture spheres $S_{0,4}$. This qualitative difference on the geometry of the WP metric on $Teich_{g,n}$ in the cases $3g-3+n>1$ and $3g-3+n=1$ (i.e., $(g,n)=(0,4)$ or $(1,1)$) will be important in the last post of this series when we will discuss the rates of mixing of the WP geodesic flow.
\end{remark}

\begin{remark}\label{r.wp-m11-surface-revolution} As it was pointed out by Wolpert \cite{Wolpert2011}, these estimates permit to think of the WP metric on the moduli space $\mathcal{M}_{1,1}\simeq\mathbb{H}^2/PSL(2,\mathbb{Z})$ in a $\varepsilon$-neighborhood of the cusp at infinity as a $C^2$-pertubation of the metric $\pi^3(4dr^2+r^6d\theta)$ of the surface of revolution of the profile $\{y=x^3\}$ modulo multiplicative factors of the form $1+O(r^4)$.
\end{remark}

Now, we will investigate the remaining ``two thirds of item (IV) of Theorem \ref{t.BMW-ergodicity-criterion}'' for the WP metric, i.e., polynomial bounds for the first two derivatives $\nabla R$ and $\nabla^2 R$ of the curvature operator $R$ of the WP metric.

\subsubsection{Bounds for the first two derivatives of WP metric: overview}

As it was \emph{recently} pointed out to us by Wolpert (in a private communication), it is possible to deduce very good bounds for the derivatives of the WP metric (and its curvature tensor) by refining the formulas for the WP metric in some of his works.

Nevertheless, by the time Burns-Masur-Wilkinson's paper \cite{BurnsMasurWilkinson} was written, it was not clear at all that Wolpert's delicate calculations for the WP metric could be extended to provide useful information about the derivatives of this metric.

For this reason, Burns-Masur-Wilkinson decided to implement the following alternative strategy.

At first sight, our task \emph{reminds} the setting of \emph{Cauchy's inequality} in Complex Analysis where one estimates the derivatives of a holomorphic function in terms of given bounds for the $C^0$-norm of this function via the \emph{Cauchy integral formula}. In fact, our current goal is to estimate the first two derivatives of a ``function'' (actually, the curvature tensor of the WP metric) defined on the complex-analytic manifold $Teich(S)$ knowing that this ``function'' already has nice bounds (cf. Equation \eqref{e.(IV)-1/3-WP}).

However, one can \emph{not} apply the argument described in the previous paragraph \emph{directly} to the curvature tensor of the WP metric because this metric is \emph{only} a real-analytic (but \emph{not} a complex-analytic/holomorphic) object on the complex-analytic manifold $Teich(S)$.

Fortunately, as it was observed by Burns-Masur-Wilkinson, this idea of using the Cauchy inequalities can still be shown to work \emph{after} one adds some results of McMullen \cite{McMullen2000} into the picture. In a nutshell, McMullen showed that the WP metric is closely related to a \emph{holomorphic} object: very roughly speaking, using the so-called \emph{Bers simultaneous uniformization theorem}, one can think of the Teichm\"uller space $Teich(S)$ as a \emph{totally real} submanifold of the so-called quasi-Fuchsian locus $QF(S)$, and, in this setting, the Weil-Petersson symplectic $2$-form $\omega_{WP}$ is the restriction to $Teich(S)$ of the differential of a \emph{holomorphic} $1$-form $\theta_{WP}$ globally defined on the quasi-Fuchsian locus $QF(S)$. In particular, it is possible to use Cauchy's inequalities to the holomorphic object $\theta_{WP}$ to get some estimates for the first two derivatives of the WP metric.

\begin{remark} A \emph{caricature} of the previous paragraph is the following. We want to estimate the first two derivatives of a real-analytic function $f:\mathbb{R}\to\mathbb{C}$ (``WP metric'') knowing some bounds for the values of $f$. In principle, we can not do this by simply applying Cauchy's estimates to $f$, but in our context we know (``by the results of McMullen'') that the natural embedding $\mathbb{R}\subset \mathbb{C}=\mathbb{R}\oplus i\mathbb{R}$ of $\mathbb{R}$ as a totally real submanifold of $\mathbb{C}$ allows to think of $f:\mathbb{R}\to\mathbb{C}$ as the restriction of a holomorphic function $g:\mathbb{C}\to\mathbb{C}$ and, thus, we can apply Cauchy inequalities to $g$ to get some estimates for $f$.
\end{remark}

In what follows, we will explain the ``Cauchy inequality idea'' of Burns-Masur-Wilkinson in two steps. Firstly, we will describe the embedding of $Teich(S)$ into the quasi-Fuchsian locus $QF(S)$ and the holomorphic $1$-form $\theta_{WP}$ of McMullen whose differential restricts to the WP symplectic $2$-form on $Teich(S)$. After that, we will show how the Cauchy inequalities can be used to give the remaining ``two thirds of item (IV) of Theorem \ref{t.BMW-ergodicity-criterion}'' for the WP metric.

\subsubsection{Quasi-Fuchsian locus $QF(S)$ and McMullen's $1$-forms $\theta_{WP}$}

Given a hyperbolic Riemann surface $S=\mathbb{H}/\Gamma$, $\Gamma<PSL(2,\mathbb{R})$, the \emph{quasi-Fuchsian locus} $QF(S)$ is defined as
$$QF(S) = Teich(S)\times Teich(\overline{S})$$
where $\overline{S}$ is the \emph{conjugate} Riemann surface of $S$, i.e., $\overline{S}$ is the quotient $\overline{S}=\mathbb{L}/\Gamma$ of the \emph{lower-half plane} $\mathbb{L}=\{z\in\mathbb{C}:\textrm{Im}(z)<0\}$ by $\Gamma$. The \emph{Fuchsian locus} $F(S)$ is the image of $Teich(S)$ under the \emph{anti-diagonal} embedding
$$\widehat{\alpha}:Teich(S)\to QF(S), \quad \widehat{\alpha}(X)=(X,\overline{X})$$

Geometrically, we can think of elements $(X,Y)\in QF(S)$ as follows. Recall that $X$ and $Y$ are related to $S$ and $\overline{S}$ via (extremal) quasiconformal mappings determined by the solutions of Beltrami equations associated to $\Gamma$-invariant Beltrami differentials (coefficients) $\mu_X$ and $\mu_{Y}$ on $\mathbb{H}$ and $\mathbb{L}$. Now, we observe that $\mathbb{H}$ and $\mathbb{L}$ live naturally on the Riemann sphere $\overline{\mathbb{C}}=\mathbb{C}\cup\{\infty\}$. Since the real axis/circle at infinity/equator $\mathbb{R}_{\infty}=\overline{\mathbb{C}}-(\mathbb{H}\cup\mathbb{L})$ has zero Lebesgue measure, we see that $\mu_X$ and $\mu_Y$ induce a Beltrami differential $\mu_{(X,Y)}$ on $\overline{\mathbb{C}}$. By solving the corresponding Beltrami equation, we obtain a quasiconformal map $f_{X,Y}$ on $\overline{\mathbb{C}}$ and, by conjugating, we obtain a \emph{quasi-Fuchsian subgroup}
$$\Gamma(X,Y)=\{f_{(X,Y)}\circ\gamma\circ f_{(X,Y)}^{-1}: \gamma\in\Gamma\}<PSL(2,\mathbb{C}),$$
i.e., a \emph{Kleinian subgroup} whose domain of discontinuity $\Omega(X,Y)\subset\overline{\mathbb{C}}$ consists of two connected components $A$ and $B$ such that $X\simeq A/\Gamma(X,Y)$ and $Y\simeq B/\Gamma(X,Y)$.

The following picture summarizes the discussion of the previous paragraph:


\begin{figure}[htb!]
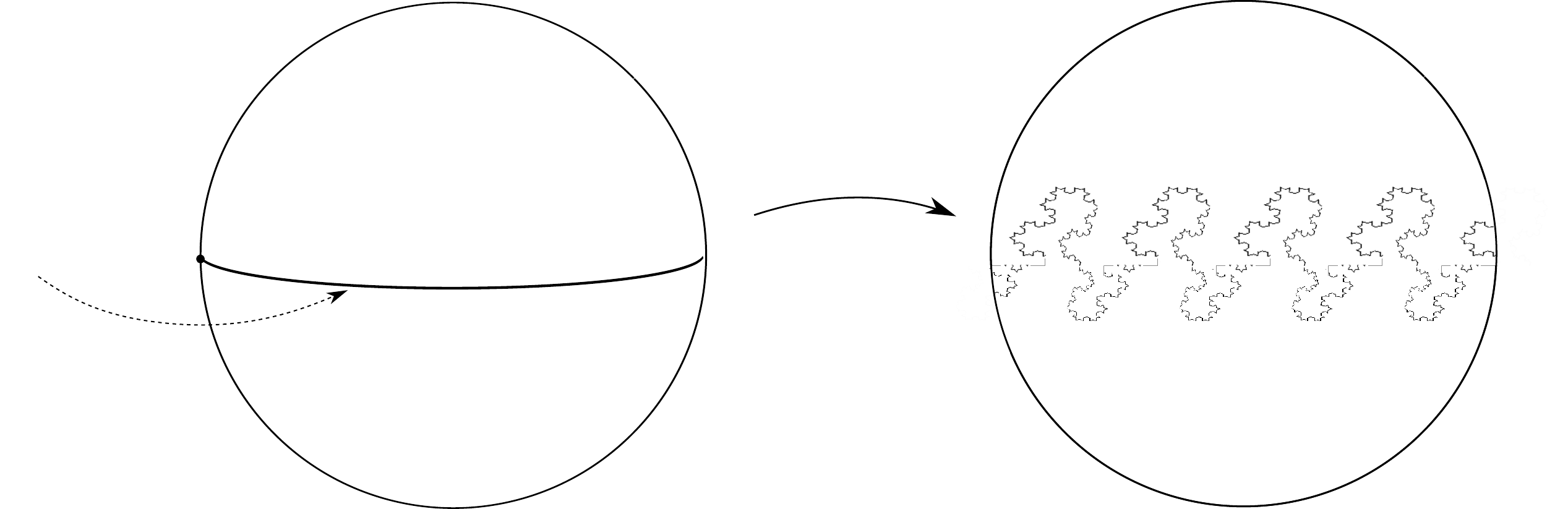
\end{figure}

\begin{remark} The Jordan curve given by the image $f_{(X,Y)}(\mathbb{R}_{\infty})$ of the equator $\mathbb{R}_{\infty}$ under the quasiconformal map $f_{(X,Y)}$ is ``wild'' in general, e.g., it has Hausdorff dimension $>1$ (as the picture above tries to represent). In fact, this happens because a typical quasiconformal map is merely a H\"older continuous, and, hence, it might send ``nice'' curves (such as the equator) into curves with ``intricate geometries''.
\end{remark}

The data of the quasi-Fuchsian subgroup $\Gamma(X,Y)$ attached to $(X,Y)\in QF(S)=Teich(S)\times Teich(\overline{S})$ permits to assign (marked) \emph{projective structures} to $X$ and $Y$. More precisely, by writing $X\simeq A/\Gamma(X,Y)$ and $Y\simeq B/\Gamma(X,Y)$ with $A, B\subset\overline{\mathbb{C}}$ and $\Gamma(X,Y)<PSL(2,\mathbb{C})$, we are equipping $X$ and $Y$ with projective structures, that is, atlases of charts to $\mathbb{C}$ whose changes of coordinates are M\"oebius transformations (i.e., elements of $PSL(2,\mathbb{C})$). Furthermore, by recalling that $X$ and $Y$ come with markings $f:S\to X$ and $g:\overline{S}\to Y$ (because they are points in Teichm\"uller spaces), we see that the projective structures above are marked.

In summary, we have a natural \emph{quasi-Fuchsian uniformization} map
$$\sigma:QF(S)\to Proj(S)\times Proj(\overline{S})$$
assigning to $(X,Y)$ the marked projective structures
$$\sigma(X,Y):=(\sigma_{QF}(X,Y),\overline{\sigma}_{QF}(X,Y))$$ Here, $Proj(S)$ is the ``Teichm\"uller space of projective structures'' on $S$, i.e., the space of ``Teichm\"uller'' equivalence classes of marked projective structures $f:S\to X$ where two marked projective structures $f_1:S\to X_1$ and $f_2:S\to X_2$ are ``Teichm\"uller'' equivalent whenever there is a projective isomorphism $h:X_1\to X_2$ homotopic to $f_2\circ f_1^{-1}$.

\begin{remark} This procedure due to Bers \cite{Bers} of attaching a quasi-Fuchsian subgroup $\Gamma(X,Y)$ to a pair of hyperbolic surfaces $X$ and $Y$ is called \emph{Bers simultaneous uniformization} because the knowledge of $\Gamma(X,Y)$ allows to equip \emph{at the same time} $X$ and $Y$ with natural projective structures.
\end{remark}

Note that $\sigma$ is a section of the natural projection
$$Proj(S)\times Proj(\overline{S})\to QF(S)=Teich(S)\times Teich(\overline{S})$$
obtained by sending each pair of (marked) projective structures $(X,Y)$, $X\in Proj(S)$, $Y\in Proj(\overline{S})$, to the unique pair of (marked) compatible conformal structures $(\pi(X), \overline{\pi}(Y))$, $\pi(X)\in Teich(S)$, $\overline{\pi}(Y)\in Teich(\overline{S})$.

We will now describe how the (affine) structure of the fibers $Proj_X(S)=\pi^{-1}(X)$ of the projection $\pi: Proj(S)\to Teich(S)$ and the section $\sigma$ can be used to construct McMullen's primitives/potentials of the Weil-Petersson symplectic form $\omega_{WP}$.

Given two projective structures $p_1, p_2\in Proj_X(S)$ in the same of the projection $\pi: Proj(S)\to Teich(S)$, one can measure how far apart from each other are $p_1$ and $p_2$ using the so-called \emph{Schwarzian derivative}.

More precisely, the fact that $p_1$ and $p_2$ induce the same conformal structure means that the charts of atlases associated to them can be thought as some families of maps $f_1:U\to\overline{\mathbb{C}}$ and $f_2:U\to\overline{\mathbb{C}}$ from (small) open subsets $U\subset X$ to the Riemann sphere $\overline{\mathbb{C}}$, and we can measure the ``difference'' $p_2-p_1$ by computing how ``far'' from a M\"obius transformation (in $PSL(2,\mathbb{C})$) is $f_2\circ f_1^{-1}$.

Here, given a point $z\in U$, one observes that there exists an \emph{unique} M\"oebius transformation $A\in PSL(2,\mathbb{C})$ such that $f_2$ and $A\circ f_1$ \emph{coincide} at $z$ up to \emph{second order} (i.e., $f_2$ and $A\circ f_1$ have the same value and the same first and second derivatives at $z$). Hence, it is natural to measure how far from a M\"obius transformation is $f_2\circ f_1^{-1}$ by understanding the difference between the \emph{third derivatives} of $f_2$ and $A\circ f_1$ at $z\in U$, i.e., $D^3(f_2-A\circ f_1)(z)$.

Actually, this is \emph{almost} the definition of the \emph{Schwarzian derivative}: since the derivatives of $f_2$ and $A\circ f_1$ map $T_z U$ to $T_{f_2(z)}\overline{\mathbb{C}}$, in order to recover an object from $T_zU$ to \emph{itself}, it is a better idea to ``correct'' $D^3(f_2-A\circ f_1)(z)$ with $Df_2^{-1}(z)$, i.e., we define the Schwarzian derivative $S\{f_2,f_1\}(z)$ of $f_2$ and $f_1$ at $z$ as
$$S\{f_2,f_1\}(z):=6\left(Df(z)^{-1}\circ D^3(f_2-A\circ f_1)(z)\right)$$
Here, the factor $6$ shows up for historical reasons\footnote{That is, this factor makes $S\{f_2,f_1\}(z)$ coincide with the classical definition of Schwarzian derivative in the literature.}.

By definition, the Schwarzian derivative $S\{f_2, f_1\}$ is a field of quadratic forms on $U$ (since its definition involves taking third order derivatives). In other terms, $S\{f_2, f_1\}$ is a \emph{quadratic differential} on $U$, that is, the ``difference'' $p_2-p_1$ between two projective structures $p_1, p_2\in Proj_X(S)$ in the same fiber of the projection $\pi:Proj(S)\to Teich(S)$ is given by a quadratic differential $p_2-p_1=S\{p_2, p_1\}\in Q(X)$. In particular, the fibers $Proj_X(S)$ are affine spaces modeled by the space $Q(X)$ of quadratic differentials on $X$.

\begin{remark} The reader will find more explanations about the Schwarzian derivative in Section 6.3 of Hubbard's book \cite{Hubbard}.
\end{remark}

\begin{remark} The idea of ``measuring'' the distance between projective structures (inducing the same conformal structure) by computing how far they are from M\"obius transformations via the Schwarzian derivative is close in some sense to the idea of measuring the distance between two points in Teichm\"uller space $Teich(S)$ by computing the eccentricities of quasiconformal maps between these points.
\end{remark}

Using this affine structure on $Proj_X(S)$ and the fact that $Q(X)\simeq T_X^*Teich(S)$ is the cotangent space of $Teich(S)$ at $X$, we see that, for each $Y, Z\in Teich(\overline{S})$, the map
$$X\in Teich(S)\mapsto \sigma_{QF}(X,Y)-\sigma_{QF}(X,Z)\in Q(X)$$
defines a (holomorphic) $1$-form on $Teich(S)$. Note that, by letting $Y\in Teich(\overline{S})$ vary and by fixing $Z\in Teich(\overline{S})$, we have a map $\tau_Z=\tau$ given by
$$(X,Y)\in Teich(S)\times Teich(\overline{S})\mapsto \tau(X,Y):=\sigma_{QF}(X,Y)-\sigma_{QF}(X,Z)\in Q(X)$$
Since $QF(S)=Teich(S)\times Teich(\overline{S})$ (so that $T^*QF(S)=T^*Teich(S)\oplus T^*Teich(\overline{S})$) and $Q(X)\simeq T^*_X Teich(S)$, we can think of $\tau$ as a (holomorphic) $1$-form on $QF(S)$.

For later use, let us notice that the $1$-form $\tau:QF(S)\to T^*Teich(S)$ is \emph{bounded} with respect to the Teichm\"uller metric on $Teich(S)$. Indeed, this is a consequence of \emph{Nehari's bound} stating that if $U\subset\overline{\mathbb{C}}$ is a round disc (i.e., the image of the unit disc $\mathbb{D}\subset \mathbb{C}\subset\overline{\mathbb{C}}$ under a M\"oebius transformation) equipped with its hyperbolic metric $\rho$ and $f:U\to\mathbb{C}$ is an injective complex-analytic map, then
$$\|S\{f,z\}\|_{L^{\infty}}\leq 3/2.$$

In this setting, McMullen constructed primitives/potentials for the WP symplectic form $\omega_{WP}$ as follows. The Teichm\"uller space $Teich(S)$ sits in the quasi-Fuchsian locus $QF(S)$ as the \emph{Fuchsian locus} $F(S)=\widehat{\alpha}(Teich(S))$ where $\widehat{\alpha}$ is the \emph{anti-diagonal} embedding
$$\widehat{\alpha}:Teich(S)\to QF(S), \quad \widehat{\alpha}(X)=(X,\overline{X})$$
By pulling back the $1$-form $\tau$ under $\widehat{\alpha}$, we obtain a bounded $1$-form
$$\theta_{WP}(X):=\widehat{\alpha}^*(\tau)(X)=\sigma_{QF}(X,\overline{X})-\sigma_{QF}(X,Z)$$
\begin{remark} This form $\theta_{WP}=\widehat{\alpha}^*(\tau)$ is closely related to a classical object in Teichm\"uller theory called \emph{Bers embedding}: in our notation, the Bers embedding is $$\beta_X(\overline{Z})=\sigma_{QF}(X,Z)-\sigma_{QF}(X,\overline{X})=-\widehat{\alpha}^*(\tau)(X)=-\theta_{WP}(X)$$
\end{remark}
McMullen \cite{McMullen2000} showed that the bounded $1$-forms $i\theta_{WP}$ are primitives/potentials of the WP symplectic $2$-form $\omega_{WP}$, i.e.,
$$d(i\theta_{WP})=\omega_{WP}$$
See also Section 7.7 of Hubbard's book \cite{Hubbard} for a nice exposition of this theorem of McMullen. Equivalently, the restriction of the holomorphic $1$-form $\tau$ to the Fuchsian locus $F(S)$ (a \emph{totally real} sublocus of $QF(S)$) permits to construct (Teichm\"uller bounded) primitives for the WP symplectic form on $F(S)$.

At this point, we are ready to implement the ``Cauchy estimate'' idea of Burns-Masur-Wilkinson to deduce bounds for the first two derivatives of the curvature operator of the WP metric.

\subsubsection{``Cauchy estimate'' of $\omega_{WP}$ after Burns-Masur-Wilkinson}

Following Burns-Masur-Wilkinson, we will need the following local coordinates in $Teich(S)$:
\begin{proposition}[McMullen \cite{McMullen2000}]\label{p.McMullen-coordinates} There exists an universal constant $C_0=C_0(g,n)\geq 1$ such that, for any $X_0\in Teich(S)=Teich_{g,n}$, one has a holomorphic embedding
$$\psi=\psi_{X_0}:\Delta^N\to Teich(S)$$
of the Euclidean unit polydisc $\Delta^N:=\{(z_1,\dots,z_N)\in \mathbb{C}^N: |z_j|<1\,\,\,\,\forall\,j=1,\dots,N\}$ (where $N=3g-3+n=\textrm{dim}(Teich(S))$) sending $0\in\Delta^N$ to $X_0=\psi(0)$ and satisfying
$$\frac{1}{C_0}\|v\|\leq \|D\psi(v)\|_T\leq C_0\|v\|, \quad\forall v\in T\Delta^N,$$
where $\|.\|_T$ is the Teichm\"uller norm and $\|.\|$ is the Euclidean norm on $\Delta^N$.
\end{proposition}

Also, since the statement of Proposition \ref{p.McMullen-coordinates} involves the Teichm\"uller norm
$\|.\|_T$ and we are interested in the Weil-Petersson norm $\|.\|_{WP}$, the following comparison (from Lemma 5.4 of Burns-Masur-Wilkinson paper \cite{BurnsMasurWilkinson}) between $\|.\|_T$ and $\|.\|_{WP}$ will be helpful:
\begin{lemma}\label{l.comparison-T-WP} There exists an universal constant $C=C(g,n)\geq 1$ such that, for any $X\in Teich(S)$ and any cotangent vector $\phi\in Q(X)\simeq T_X^*Teich(S)$, one has
$$\|\phi\|_{WP}\leq C\frac{1}{\underline{\ell}(X)}\|\phi\|_T$$
where $\underline{\ell}(X)$ is the systole of $X$(i.e., the length of the shortest closed simple hyperbolic geodesics on $X$). In particular, for any $X\in Teich(S)$ and any tangent vector $\mu\in T_X Teich(S)$, one has
$$\|\mu\|_{T}\leq C\frac{1}{\underline{\ell}(X)}\|\mu\|_{WP}$$
\end{lemma}

\begin{proof} Given $X\in Teich(S)$, let us write $X\simeq \mathbb{H}^2/\Gamma$ where $\Gamma<PSL(2,\mathbb{R})$ is ``normalized'' to contain the element $T(z)=\lambda z$ where $\lambda=\log\underline{\ell}(X)$.

Fix $D\subset\mathbb{H}$ a \emph{Dirichlet fundamental domain} of the action of $\Gamma$ centered at the point $i\in\mathbb{H}$. By the \emph{collaring theorem}\footnote{Saying that a closed simple hyperbolic geodesic $\gamma$ of length $\ell$ has a collar $A(\gamma,\eta(\ell))$ (tubular neighborhood) of radius $\eta(\ell):=(1/2)\log((\cosh(\ell/2)+1)/(\cosh(\ell/2)-1))$ isometrically embedded in $X$, and two of these collars $A(\gamma_1,\eta(\ell_1))$ and $A(\gamma_2,\eta(\ell_2))$ are disjoint whenever $\gamma_1$ and $\gamma_2$ are disjoint (see, e.g., Theorem 3.8.3 in Hubbard's book \cite{Hubbard}).}, we have that the union of $1/\underline{\ell}(X)$ isometric copies of $D$ contains a ball $B$ of fixed (universal) radius $c=c(g,n)>0$ around any point $z\in D$.

By combining the Cauchy integral formula with the fact stated in the previous paragraph, we see that
$$|\phi(z)|\leq \frac{1}{2\pi c}\int_B |\phi|\leq \frac{1}{2\pi c\underline{\ell}(X)} \int_D |\phi|=\frac{1}{2\pi c\underline{\ell}(X)}\|\phi\|_T$$

Since the hyperbolic metric $\rho$ is bounded away from $0$ on $D$, we can use the $L^{\infty}$-norm estimate on $\phi$ above to deduce that
$$\|\phi\|_{WP}^2:=\int_D\frac{|\phi|^2}{\rho^2} \leq \frac{C}{\underline{\ell}(X)^2}\|\phi\|_T^2$$
for some constant $C=C(g,n)>0$. This completes the proof of the lemma.
\end{proof}

\begin{remark}\label{r.Wolpert-factor-1/2} The factor $1/\underline{\ell}(X)$ in the previous lemma can be replaced by $1/\sqrt{\underline{\ell}(X)}$ via a refinement of the argument above. However, we will not prove this here because this refined estimate is not needed for the proof of the main results of Burns-Masur-Wilkinson.
\end{remark}

Using the local coordinates from Proposition \ref{p.McMullen-coordinates} (and the comparison between Teichm\"uller and Weil-Petersson norms in the previous lemma), we are ready to use Cauchy's inequalities to estimate ``$g_{ij}$'s'' of the WP metric. More concretely, denoting by $\psi=\psi_{X_0}$ ``centered at some $X_0\in Teich(S)$'' in Proposition \ref{p.McMullen-coordinates}, let $z_k=x_k+iy_k$, $k=1,\dots, N$ and consider the vector fields
$$e_{\ell}:=\left\{\begin{array}{cl} \partial/\partial x_{\ell}, & \textrm{if } \ell=1,\dots, N \\
\partial/\partial y_{\ell-N}, & \textrm{if } \ell=N+1,\dots, 2N\end{array}\right.$$
on $\Delta^N$. In setting, we denote by $G_{ij}(z)=\psi^*g_{WP}(z)(e_i,e_j)$ the ``$g_{ij}$'s'' of the WP metric $g_{WP}$ in the local coordinate $\psi$ and by $G^{-1}(z)=(G^{ij}(z))_{1\leq i,j\leq 2N}$ the inverse of the matrix $(G_{ij}(z))_{1\leq i,j\leq 2N}$.

\begin{proposition}\label{p.WP-metric-coordinates} There exists an universal constant $C=C(g,n)\geq 1$ such that, for any $X_0\in Teich(S)$, the pullback $G=\psi^*g_{WP}$ of the WP metric $g_{WP}$ local coordinate $\psi=\psi_{X_0}:\Delta^N\to Teich(S)$ ``centered at $X_0$'' in Proposition \ref{p.McMullen-coordinates} verifies the following estimates:
$$\|G^{-1}(z)\|\leq C/\underline{\ell}(X_0)^2 \quad \forall z\in\Delta^N, \, \|z\|<1/2,$$
and
$$\max\limits_{(\xi_1,\dots,\xi_k)\in\{x_1,\dots, x_N, y_1,\dots, y_N\}^k}\frac{1}{k!}\left|\frac{\partial^k G_{ij}}{\partial\xi_1\dots\partial\xi_k}(z)\right|\leq C$$
for all $1\leq i,j\leq 2N$, $k\geq 0$ and $z\in\Delta^N$, $\|z\|<1/2$.
\end{proposition}

\begin{proof} The first inequality
$$\|G^{-1}(z)\|\leq C/\underline{\ell}(X_0)^2$$
follows from Proposition \ref{p.McMullen-coordinates} and Lemma \ref{l.comparison-T-WP}. Indeed, by letting $v=\sum\limits_{i=1}^{2N} v_i e_i$, we see from Proposition \ref{p.McMullen-coordinates} and Lemma \ref{l.comparison-T-WP} that
$$\|v\|^2\leq C_0^2\|D\psi(v)\|_T^2\leq \frac{C}{\underline{\ell}(X_0)^2}\|D\psi(v)\|_{WP}^2$$
Since
\begin{eqnarray*}\|D\psi(v)\|_{WP}^2&=&\langle D\psi(v), D\psi(v)\rangle_{WP}=\sum v_i v_j\langle D\psi(e_i), D\psi(e_j)\rangle_{WP} \\ &=& \sum v_i v_j G_{ij}=\langle v, Gv\rangle \\ &\leq& \|v\|\cdot\|Gv\|,
\end{eqnarray*}
we deduce that
$$\|v\|^2\leq \frac{C}{\underline{\ell}(X_0)^2}\|v\|\cdot\|Gv\|,$$
i.e., $\|G^{-1}\|\leq C/\underline{\ell}(X_0)^2$.

For the proof of second inequality (estimates of the $k$-derivatives of $G_{ij}$'s), we begin by ``rephrasing'' the construction of McMullen's $\theta_{WP}$-form in terms of the local coordinate $\psi=\psi_{X_0}$ introduced in Proposition \ref{p.McMullen-coordinates}.

The composition $\widehat{\alpha}\circ\psi$ of the local coordinate $\psi:\Delta^N\to Teich(S)$ with the anti-diagonal embedding $\widehat{\alpha}:Teich(S)\to QF(S)$ of the Teichm\"uller space in the quasi-Fuchsian locus can be rewritten as
$$\widehat{\alpha}\circ\psi = \Psi\circ \alpha$$
where $\alpha:\Delta^N\to\Delta^N\times\Delta^N$ is the anti-diagonal embedding
$$\alpha(z)=(z,\overline{z})$$
and the local coordinate $\Psi:\Delta^N\times\Delta^N\to QF(S)$ given by
$$\Psi(z,w)=(\psi(z),\overline{\psi(\overline{w})}).$$

In this setting, the pullback by $\Psi$ of the holomorphic $1$-form
$\tau(X,Y)=\sigma_{QF}(X,Y)-\sigma_{QF}(X,Z)$ gives a holomorphic $1$-form $\kappa=\Psi^*\tau$ on $\Delta^N\times\Delta^N$. Moreover, since the Euclidean metric on $\Delta^N\times\Delta^N$ is comparable to the pullback by $\Psi$ of the Teichm\"uller metric (cf. Proposition \ref{p.McMullen-coordinates}), $\tau$ is bounded in Teichm\"uller metric and $d(i\theta_{WP})=\omega_{WP}$ where $\theta_{WP}=\widehat{\alpha}^*\tau$, we see that
$$\alpha^*\Omega=\psi^*\omega_{WP}$$
where $\Omega:=d(i\kappa)$ and $\kappa:=\Psi^*\tau$ is a holomorphic bounded (in the Euclidean norm) $1$-form on $\Delta^N\times\Delta^N$.

Let us write $\kappa=\sum\limits_{j=1}^N a_j dz_j$ in complex coordinates $(z_1,\dots,z_N,w_1,\dots,w_N)\in\Delta^N\times\Delta^N$, where $a_j:\Delta^N\times\Delta^N\to\mathbb{C}$ are bounded holomorphic functions. Hence,
$$\Omega=d(i\kappa)=i\left(\sum\limits_{j,k=1}\frac{\partial a_j}{\partial z_k} dz_k\wedge dz_j + \sum\limits_{j,k=1}\frac{\partial a_j}{\partial w_k} dw_k\wedge dz_j\right)$$
and, a fortiori,
$$\psi^*\omega_{WP}=\alpha^*\Omega=i\left(\sum\limits_{j,k=1}\frac{\partial a_j}{\partial z_k} dz_k\wedge dz_j + \sum\limits_{j,k=1}\frac{\partial a_j}{\partial \overline{z}_k} d\overline{z}_k\wedge dz_j\right)$$

Since $\psi^*\omega_{WP}$ is the K\"ahler form of the metric $G=\psi^*g_{WP}$, we see that the coefficients of $G$ are linear combinations of the $\alpha$-pullbacks of $\partial a_j/\partial z_k$ and $\partial a_j/\partial w_k$. Because $a_j$ are (universally) bounded holomorphic functions, we can use Cauchy's inequalities to see that the derivatives of $a_j$ are (universally) bounded at any $(z,w)\in\Delta^N$ with $\|(z,w)\|<1/2$. It follows from the boundedness of the (\emph{non-holomorphic}) anti-diagonal embedding $\alpha$ that the $k$-derivatives of $G_{ij}$'s satisfy the desired bound.
\end{proof}

The estimates in Proposition \ref{p.WP-metric-coordinates} (controlling the WP metric in the local coordinates constructed in Proposition \ref{p.McMullen-coordinates}) permit to deduce the remaining ``two thirds of item (IV) of Theorem \ref{t.BMW-ergodicity-criterion}'' for the WP metric:

\begin{theorem}[Burns-Masur-Wilkinson]\label{t.BMW-Cauchy} There are constants $C>0$ and $\beta>0$ such that, for any $X_0\in \mathcal{T}=Teich(S)$, the curvature tensor $R_{WP}$ of the WP metric satisfies
$$\max\{\|\nabla R_{WP}(X_0)\|, \|\nabla^2 R_{WP}(X_0)\|\}\leq C d(X_0,\partial\mathcal{T})^{-\beta}$$
\end{theorem}

\begin{proof} Fix $X_0\in Teich(S)$ and consider the local coordinate $\psi=\psi_{X_0}$ provided by Proposition \ref{p.McMullen-coordinates}. Since $\|D\psi\|$ and $\|D\psi^{-1}\|$ are uniformly bounded, our task is reduced to estimate the first two derivatives of the curvature tensor $R$ of the metric $G(z)=\psi^*g_{WP}(z)=(G_{ij}(z))$ at the origin $0\in \Delta^N$.

Recall that the \emph{Christoffel symbols} of $G_{ij} = G_{ij}(z)$ are
$$\Gamma^{m}_{ij}=\frac{1}{2}\sum\limits_{k}G^{mk}\left(\frac{\partial G_{ki}}{\partial \xi_j} + \frac{\partial G_{kj}}{\partial \xi_i} - \frac{\partial G_{ij}}{\partial \xi_k}\right)$$
or
$$\Gamma^{m}_{ij}=\frac{1}{2}G^{mk}(G_{ki,m} + G_{kj,i} - G_{ij,k})$$
in \emph{Einstein summation convention}, and, in terms of the Christoffel symbols, the coefficients of the \emph{curvature tensor} are
$$R^{l}_{ijk}=\frac{\partial\Gamma^{l}_{ik}}{\partial \xi_j} - \frac{\partial\Gamma^{l}_{ij}}{\partial \xi_k} + \Gamma^{l}_{js}\Gamma^{s}_{ik} - \Gamma^{l}_{ks}\Gamma^{s}_{ij}$$

Therefore, we see that the coefficients of the $k$-derivative $\nabla^k R$ is a polynomial function of $G^{ij}$  and the first $k+2$ partial derivatives $G_{ij}$ whose ``degree''\footnote{Because of the formula $DG^{-1}(0)=-G^{-1}(0)\cdot DG(0)\cdot G^{-1}(0)$.} in the ``variables'' $G^{ij}$ is $\leq k+2$.

By Proposition \ref{p.WP-metric-coordinates}, each $G^{ij}(0)$ has order $O(\underline{\ell}(X_0)^{-2})$ and the first $k+2$ partial derivatives of $G_{ij}$ at $0$ are bounded by a constant depending only on $k$. It follows that
\begin{eqnarray*}
\|\nabla^k R(0)\|^2&\leq& C(k)\sum\limits_{i_1,\dots,i_{k+3},j_1,\dots,j_{k+3},l,m}(\nabla^k R)^{l}_{i_1\dots i_{k+3}} (\nabla^k R)^{m}_{j_1\dots j_{k+3}} G^{i_1j_1}\dots G^{i_{k+3}j_{k+3}} G_{lm} \\ &\leq& C(k)
\frac{1}{\underline{\ell}(X_0)^{2(k+2)}}\frac{1}{\underline{\ell}(X_0)^{2(k+2)}}\frac{1}{\underline{\ell}(X_0)^{2(k+3)}} = C(k)\frac{1}{\underline{\ell}(X_0)^{6k+14}},
\end{eqnarray*}
and, consequently,
$$\max\{\|\nabla R_{WP}(X_0)\|,\|\nabla^2 R_{WP}(X_0)\|\}\leq C/\underline{\ell}(X_0)^{26/2}=C/d(X_0,\partial{T})^{26}.$$
This completes the proof.
\end{proof}

At this point, we have that Theorem \ref{t.Wolpert-expansions} (or, more precisely, its consequence in Equation \eqref{e.(IV)-1/3-WP}) and Theorem \ref{t.BMW-Cauchy} imply the validity of item (IV) of Theorem \ref{t.BMW-ergodicity-criterion} for the WP metric.

\begin{remark} The estimates for the derivatives of the curvature tensor $R_{WP}$ appearing in the proof of Theorem \ref{t.BMW-Cauchy} are \emph{not} sharp with respect to the exponent $\beta$. For instance, the WP metric on the moduli space $\mathcal{M}_{1,1}$ of once-punctured torii has curvature $\sim -1/\ell\sim -1/d^2$ where $d=d(X_0,\infty)$ is the WP distance between $X_0$ and the boundary  $\partial \mathcal{M}_{1,1}=\{\infty\}$, so that one expects tha the $kth$-derivatives of the curvature behave like $\sim -1/d^{k+2}$ (i.e., the exponent $6k+14$ above should be $k+2$).

In a very recent private communication, Wolpert indicated that it is possible to derive the \emph{sharp} estimates of the form
$$\|\nabla^k R_{WP}(X_0)\|\leq C(k)/d(X_0,\partial\mathcal{T})^{k+2}$$
for the derivatives of the curvature tensor of the WP metric from his works.
\end{remark}

\subsection{Item (V) of Theorem \ref{t.BMW-ergodicity-criterion} for WP metric}

The main result of this subsection is the following theorem implying item (V) of Theorem \ref{t.BMW-ergodicity-criterion} for WP metric.
\begin{theorem}\label{t.injectivity-WP} There exists a constant $c>0$ such that for all $X\in\mathcal{M}[k]=\mathcal{T}/MCG[k]$, $k\geq 3$, one has the following polynomial lower bound
$$inj(X)\geq c \cdot d_{WP}(X,\partial\mathcal{M}[k])^3$$
on the injectivity radius of the WP metric at $X$.
\end{theorem}

The proof of this result also relies on the work of Wolpert. More precisely, Wolpert \cite{Wolpert2003} showed that there exists a constant $c>0$ such that, for any $\sigma\in\mathcal{C}(S)$ and $X\in\mathcal{T}$ with $\overline{\ell}(X)\ll 1$,
$$d_{WP}(X,\Gamma(\sigma)(X))\geq c d(X,\partial\mathcal{T})^3$$
where $\Gamma(\sigma)\subset MCG(S)[k]$ is the Abelian subgroup of the ``level $k$'' mapping class group $MCG(S)[k]$ generated by the \emph{Dehn twists} $\tau_{\alpha}$ about the curves $\alpha\in\sigma$.

This reduces the proof of Theorem \ref{t.injectivity-WP} to the following lemma:

\begin{lemma} There exists an universal constant $J_0=J_0(g,n)\geq 1$ with the following property. For each $\varepsilon>0$, there exists $\delta>0$ such that, for any $X\in\mathcal{T}$ with
$$d_{WP}(X,\phi(X))<\delta$$
for some non-trivial $\phi\in MCG(S)[k]$, one can find $\sigma\in\mathcal{C}(S)$ so that $\overline{\ell}_{\sigma}(X)<\varepsilon$ and $\phi^j\in\Gamma(\sigma)$ for some $1\leq j\leq J_0$.
\end{lemma}

\begin{proof} We begin the proof of the lemma by recalling that the mapping class group $MCG(S)[k]$ acts on $\mathcal{T}$ in a properly discontinuous way with no fixed points. Therefore, for each $\varepsilon>0$, there exists $\delta>0$ such that if $d_{WP}(X,\phi(X))<\delta$ for some non-trivial $\phi\in MCG(S)[k]$ (i.e., some non-trivial element of the mapping class group has an ``almost fixed point''), then $\overline{\ell}_{\sigma}(X)<\varepsilon$ (i.e., the ``almost fixed point'' is close to the boundary of $\mathcal{T}$).

Let us show now that in the setting of the previous paragraph, $\phi^j\in \Gamma(\sigma)$ for some $1\leq j\leq J_0$.

In this direction, let $J_0=J_0(g,n)\in\mathbb{N}$ be the product of $(3g-3+n)!$ and the maximal orders of all finite order elements of the mapping class groups of ``lower complexity'' surfaces. By contradiction, let us assume that there exist infinite sequences $X_m\in\mathcal{T}$, $\phi_m\in MCG(S)[k]$, $m\in\mathbb{N}$, such that $\overline{\ell}_{\sigma}(X_m)\ll 1$ for some $\sigma\in\mathcal{C}(S)$ and
$$\lim\limits_{m\to\infty}d(X_m,\phi_m(X_m))=0$$
but $\phi_m^j\notin\Gamma(\sigma)$ for all $m\in\mathbb{N}$, $1\leq j\leq J_0$.

Passing to a subsequence (and applying appropriate elements of $\phi_m\in\Gamma(\sigma)$), we can assume that the sequence $X_m\in\mathcal{T}$ converges to some noded Riemann surface $X_{\sigma}\in\partial\mathcal{T}_{\sigma}$. Because $d(X_m,\phi_m(X_m))\to 0$ as $m\to\infty$, we see that ,for each $\beta\in\sigma$,
$$\ell_{\phi_m(\beta)}(\phi_m(X))=\ell_{\beta}(X_m)\to 0.$$
It follows that, for all $m$ sufficiently large, $\phi_m$ sends any curve $\beta\in\sigma$ to another curve $\phi_m(\beta)\in\sigma$. Therefore, for each $m$ sufficiently large, there exists
$$1\leq j=j(m) \leq \#\sigma!\leq (3g-3+n)!\leq J_0$$
such that $\phi_m^j$ \emph{fixes} each $\beta\in\sigma$ (i.e., $\phi_m^j$ is a \emph{reducible} element of the mapping class group). By the \emph{Nielsen-Thruston classification} of elements of the mapping class groups, the restrictions of $\phi_m^j$ to each piece of $X_{\sigma}$ are given by compositions of Dehn twists about the boundary curves with either a \emph{pseudo-Anosov} or a periodic (finite order) element (in a surface of ``lower complexity'' than $S$).

It follows that we have only two possibilities for $\phi_m^j$: either the restrictions of $\phi_m^j$ to \emph{all} pieces of $X_{\sigma}$ are compositions of Dehn twists about certain curves in $\sigma$ and \emph{finite order} elements, or the restriction of $\phi_m^j$ to \emph{some} piece of $X_{\sigma}$ is the composition of Dehn twists about certain curves in $\sigma$ and a \emph{pseudo-Anosov} element.

In the first scenario, by the definition of $J_0$, we can replace $\phi_m^j$ by an adequate power $\phi_m^{J}$ with $1\leq J\leq J_0$ to ``kill'' the finite order elements and ``keep'' the Dehn twists. In other terms, $\phi_m^{J}\in \Gamma(\sigma)$ (with $1\leq J\leq J_0$), a contradiction with our choice of the sequence $\phi_m$.

This leaves us with the second scenario. In this case, by definition of $J_0$, we can replace $\phi_m^j$ by an adequate power $\phi_m^J$ with $1\leq J\leq J_0$ such that the restriction of $\phi_m^J$ to some piece of $X_{\sigma}$ is pseudo-Anosov. However, Daskalopoulos-Wentworth \cite{DaWe} showed that there exists an \emph{uniform} positive lower bound for
$$d_{WP}(X_{\sigma}, \phi_m^J(X_{\sigma}))$$
when $\phi_m^J$ is pseudo-Anosov on some piece of $X_{\sigma}$. Since $1\leq J\leq J_0$ and $J_0$ is an universal constant, it follows that there exists an uniform positive lower bound for
$$d_{WP}(X_m,\phi_m(X_m))$$
for all $m$ sufficiently large, a contradiction with our choice of the sequences $X_m\in\mathcal{T}$ and $\phi_m\in MCG(S)[k]$.

These contradictions show that the sequences $X_m\in\mathcal{T}$ and $\phi_m\in MCG(S)[k]$ with the properties described above can't exist.

This completes the proof of the lemma.
\end{proof}

\subsection{Item (VI) of Theorem \ref{t.BMW-ergodicity-criterion} for WP flow}

We complete in this subsection our discussion of the proof of Theorem \ref{t.BMW} \emph{modulo} Theorem \ref{t.BMW-ergodicity-criterion} by verifying the item (VI) of Theorem \ref{t.BMW-ergodicity-criterion} for the WP geodesic flow $\varphi_t$. More precisely, we will show the following result:

\begin{theorem}\label{t.WPflow-derivative} There are constants $C\geq 1$, $\beta>0$, $\delta>0$ and $\rho_0>0$ such that
$$\|D_v\varphi_{\tau}\|_{WP}\leq C/\rho_{\tau}(v)^{\beta}$$
for any $0\leq\tau\leq\delta$ and any $v\in T^1\mathcal{T}$ with
$$0<\rho_{\tau}(v):=\min\{d_{WP}(\varphi_t(v),\partial\mathcal{T}): t\in[-\tau,\tau]\}<\rho_0.$$
\end{theorem}

The proof of this result in \cite{BurnsMasurWilkinson} is naturally divided into two steps.

In the first step, one shows a \emph{general} result providing an estimate for the first derivative of the geodesic flow $\varphi_t$ on \emph{arbitrary} negatively curved manifold:

\begin{theorem}\label{t.geodesic-flow-derivative} Let $M$ be a negatively curved manifold. Consider $\gamma:[-\tau,\tau]\to M$ a geodesic where $0\leq \tau\leq 1$ and suppose that for every $-\tau\leq t\leq\tau$ the sectional curvatures of any plane containing $\dot{\gamma}(t)\in T^1M$ is greater than $-\kappa(t)^2$ for some Lipschitz function $\kappa:[-\tau,\tau]\to\mathbb{R}_+$.

Then,
$$\|D_{\dot{\gamma}(0)}\varphi_{\tau}\|\leq 1+2(1+u(0)^2)(1+\sqrt{1+u(\tau)^2})\exp\left(\int_0^\tau u(s) ds\right)$$
where $u:[-\tau,\tau]\to[0,\infty)$ is the solution of Riccati equation
$$u'+u^2=\kappa^2$$
with initial data $u(-\tau)=0$.
\end{theorem}

\begin{remark} The proof of this theorem involves classical objects in Differential Geometry (e.g., Jacobi fields,  matrix Riccati equation, Sasaki metric, etc.), but we will not make more comments on this topic because it is not directly related to the geometry of moduli spaces of Riemann surfaces. Instead, we refer the curious reader to the original article \cite{BurnsMasurWilkinson} of Burns-Masur-Wilkinson (or the paper \cite{Burns} in this volume).
\end{remark}

In the second step, one uses the works of Wolpert to exhibit an adequate bound $\kappa(t)$ for the sectional curvatures of the WP metric along WP geodesics $\gamma(t)$. More concretely, one has the following theorem:

\begin{theorem}\label{t.BMW-Prop-4-22} There are constants $Q, P, L\geq 2$ and $0<\delta<1$ such that for any $0<\delta'<\delta$ and any geodesic segment $\gamma:(-\delta',\delta')\to\mathcal{T}$ there exists a positive Lipschitz function $\kappa:(-\delta',\delta)\to \mathbb{R}_+$ with
\begin{itemize}
\item[(a)] $\sup\limits_{v\in T^1_{\gamma(t)}\mathcal{T}} -\langle R_{WP}(v,\dot{\gamma}(t)) \dot{\gamma}(t), v\rangle_{WP}\leq \kappa^2(t)$ for all $t\in(-\delta',\delta')$;
\item[(b)] $\kappa$ is $Q$-controlled in the sense that $\kappa$ has a right-derivative $D^+\kappa$ satisfying
$$D^+\kappa\geq \frac{1-Q^2}{Q}\kappa^2$$
\item[(c)] $\int_{-\delta'}^{\delta'} \kappa(s) ds\leq L|\log\rho_{\delta'}(\dot{\gamma}(0))|$;
\item[(d)] $\max\{\kappa(0), \kappa(\delta')\}\leq P/\rho_{\delta'}(\dot{\gamma}(0))$.
\end{itemize}
where $\rho_{\delta'}(\dot{\gamma}(0))$ is the distance between the geodesic segment $\gamma([-
\delta', \delta])$ and $\partial\mathcal{T}$.
\end{theorem}

Using Theorems \ref{t.geodesic-flow-derivative} and \ref{t.BMW-Prop-4-22}, we can easily complete the proof of Theorem \ref{t.WPflow-derivative} (i.e., the verification of item (VI) of Theorem \ref{t.BMW-ergodicity-criterion} for the WP metric):

\begin{proof}[Proof of Theorem \ref{t.WPflow-derivative}] Denote by $\kappa$ the ``WP curvature bound'' function provided by Theorem \ref{t.BMW-Prop-4-22} and let $u:[-\delta,\delta]\to\mathbb{R}_+$ be the solution of Riccati's equation
$$u'+u^2=\kappa^2$$
with initial data $u(-\delta)=0$.

Since $\kappa$ is $Q$-controlled (in the sense of item (b) of Theorem \ref{t.BMW-Prop-4-22}), it follows that $u(t)\leq Q\kappa(t)$ for all $t\in[-\delta,\delta]$: indeed, this is so because $u(-\delta)=0\leq Q\kappa(-\delta)$, and, if $u(t_0)=Q\kappa(t_0)$ for some $t_0\in[-\delta,\delta]$, then
$$u'(t_0)=\kappa(t_0)^2-u(t_0)^2=(1-Q^2)\kappa(t_0)^2\leq Q\cdot D^+\kappa(t_0).$$

Therefore, by applying Theorem \ref{t.geodesic-flow-derivative} in this setting, we deduce that
$$\|D_{\dot{\gamma}(0)}\varphi_{\tau}\|_{WP}\leq C/\rho_{\tau}(\dot{\gamma}(0))^{\beta}$$
for $\beta=L+3$ and some constant $C=C(P,Q)\geq 1$. This completes the proof of Theorem \ref{t.WPflow-derivative}.
\end{proof}

Closing this subsection, let us \emph{sketch} the proof of Theorem \ref{t.BMW-Prop-4-22} while referring to Subsection 4.4 of Burns-Masur-Wilkinson paper \cite{BurnsMasurWilkinson} (especially Proposition 4.22 of this article) for more details.

We start by describing how the function $\kappa$ is defined. For this sake, we will use Wolpert's formulas in Theorem \ref{t.Wolpert-expansions} above.

More precisely, since the sectional curvatures of the WP metric approach $0$ or $-\infty$ only near the boundary, we can assume\footnote{Formally, as Burns-Masur-Wilkinson explain in page 883 of \cite{BurnsMasurWilkinson}, one must use Proposition 4.7 of their article to produce a nice ``thick-thin'' decomposition of the Teichm\"uller space $\mathcal{T}$.} that our geodesic segment $\gamma:[-\delta',\delta']\to\mathcal{T}$ in the statement of Theorem \ref{t.BMW-Prop-4-22} is ``relatively close'' to a boundary stratum $\mathcal{T}_{\sigma}$, $\sigma\in\mathcal{C}(S)$.

In this setting, for each $\alpha\in\sigma$, we consider the functions $f_{\alpha}(t):=\sqrt{\ell_{\alpha}(t)}$ and
$$r_{\alpha}(t):=\sqrt{\langle \lambda_{\alpha}, \dot{\gamma}(t)\rangle^2 + \langle J\lambda_{\alpha}, \dot{\gamma}(t)\rangle^2}$$
(where $\lambda_{\alpha}:=\textrm{grad}\,\ell_{\alpha}^{1/2}$) along our geodesic segment $\gamma:I\to\mathcal{T}$, $I=(-\delta',\delta')$. Notice that it is natural to consider these functions in view of the statements in Wolpert's formulas in Theorem \ref{t.Wolpert-expansions}.

The WP sectional curvatures of planes containing the tangent vectors to $\gamma(I)$ are controlled in terms of $r_{\alpha}$ and $f_{\alpha}$. Indeed, given $v\in T_{\gamma(t)}^1\mathcal{T}$, we can use a combined length basis $(\sigma,\chi)\in\mathcal{B}$ to write
$$v:=\sum\limits_{\alpha\in\sigma}(a_{\alpha}\lambda_{\alpha}+b_{\alpha} J\lambda_{\alpha}) + \sum\limits_{\beta\in\chi} c_{\beta}\textrm{grad}\,\ell_{\beta}$$
Similarly, let us write
$$\dot{\gamma}(t)=\dot{\gamma}:=\sum\limits_{\alpha\in\sigma}(A_{\alpha}\lambda_{\alpha}+B_{\alpha} J\lambda_{\alpha}) + \sum\limits_{\beta\in\chi} C_{\beta}\textrm{grad}\,\ell_{\beta}$$

By Theorem \ref{t.Wolpert-expansions}, we obtain the following facts. Firstly, since $v$ and $\dot{\gamma}$ are WP-unit vectors, the coefficients $a_{\alpha}, b_{\alpha}, c_{\alpha}, A_{\alpha}, B_{\alpha}, C_{\alpha}$ are
$$a_{\alpha}, b_{\alpha}, c_{\alpha}, A_{\alpha}, B_{\alpha}, C_{\alpha} = O(1)$$
Secondly, by definition of $r_{\alpha}$, we have that
$$r_{\alpha}^2=\frac{1}{4\pi^2}(A_{\alpha}^2+B_{\alpha}^2)+O(f_{\alpha}^3)$$
Finally,
\begin{eqnarray*}
-\langle R_{WP}(v,\dot{\gamma})\dot{\gamma}, v\rangle_{WP} &=& \sum\limits_{\alpha\in\sigma} (a_{\alpha}^2B_{\alpha}^2+A_{\alpha}^2 b_{\alpha}^2)
\langle R_{WP}(\lambda_{\alpha}, J\lambda_{\alpha})J\lambda_{\alpha}, \lambda_{\alpha}\rangle_{WP}+O(1) \\
&=&\sum\limits_{\alpha\in\sigma} O\left(\frac{r_{\alpha}^2}{f_{\alpha}^2}\right)+O(1)
\end{eqnarray*}
In summary, Wolpert's formulas (Theorem \ref{t.Wolpert-expansions}) imply that
\begin{equation}\label{e.BMW-Prop-4-22-a}\sup\limits_{v\in T^1_{\dot{\gamma}(t)}\mathcal{T}} -\langle R_{WP}(v,\dot{\gamma}(t))\dot{\gamma}(t), v\rangle_{WP} = \sum\limits_{\alpha\in\sigma} O\left(\frac{r_{\alpha}(t)^2}{f_{\alpha}(t)^2}\right)
\end{equation}
(cf. Lemma 4.17 in \cite{BurnsMasurWilkinson}).

Now, we want convert the expressions $r_{\alpha}(t)/f_{\alpha}(t)$ into a positive Lipschitz function satisfying the properties described in items (b), (c), and (d) of Theorem \ref{t.BMW-Prop-4-22}, i.e., a $Q$-controlled function with appropriately bounded total integral and values at $0$ and $\delta'$. We will not give full details on this (and we refer the curious reader to Subsection 4.4 of \cite{BurnsMasurWilkinson}), but, as it turns out, the function
$$\kappa(t):=C\max_{\alpha\in\sigma}\left\{1,\frac{r_{\alpha}(t_{\alpha})}{r_{\alpha}(t_{\alpha})|t-t_{\alpha}|+f_{\alpha}(t_{\alpha})}\right\}$$
where $t_{\alpha}\in[-\delta',\delta']$ is the (unique) time with $f_{\alpha}(t)\geq f_{\alpha}(t_{\alpha})$ for all $t\in[-\delta',\delta']$ and $C\geq 1$ is a sufficiently large constant satisfies the conditions in items (a), (b), (c) and (d) of Theorem \ref{t.BMW-Prop-4-22}. Here, the basic idea is these properties are consequences of the features of two ODE's (cf. Lemmas 4.15 and 4.16 in \cite{BurnsMasurWilkinson}) for $r_{\alpha}$ and $f_{\alpha}$. For instance, the verification of item (a) (i.e., the fact that $\kappa$ controls certain WP sectional curvatures along $\gamma$) relies on the fact that these two ODE's permit to prove that
$$\frac{r_{\alpha}(t)}{f_{\alpha}(t)}\leq A\max\left\{1,\frac{r_{\alpha}(t_{\alpha})}{r_{\alpha}(t_{\alpha})|t-t_{\alpha}|+f_{\alpha}(t_{\alpha})}\right\}$$
for some sufficiently large constant $A\geq 1$. In particular, by plugging this into \eqref{e.BMW-Prop-4-22-a}, we obtain that
$$\sup\limits_{v\in T^1_{\dot{\gamma}(t)}\mathcal{T}} -\langle R_{WP}(v,\dot{\gamma}(t))\dot{\gamma}(t), v\rangle_{WP}\leq \kappa^2(t),$$
i.e.,, the estimate required by item (a) of Theorem \ref{t.BMW-Prop-4-22}.

Concluding this sketch of proof of Theorem \ref{t.BMW-Prop-4-22}, let us indicate the two ODE's on $r_{\alpha}$ and $f_{\alpha}$.

\begin{lemma}[Lemma 4.15 of \cite{BurnsMasurWilkinson}]\label{l.Clairaut-ODE} $r_{\alpha}'(t)=O(f_{\alpha}^3(t))$.
\end{lemma}

\begin{proof} By differentiating $r_{\alpha}(t)^2=\langle \lambda_{\alpha}, \dot{\gamma}(t)\rangle^2 + \langle J\lambda_{\alpha}, \dot{\gamma}(t)\rangle^2$, we see that
$$2r_{\alpha}(t)r_{\alpha}'(t) = 2 \langle \lambda_{\alpha}, \dot{\gamma}(t)\rangle \langle \nabla_{\dot{\gamma}(t)}\lambda_{\alpha}, \dot{\gamma}(t)\rangle + 2 \langle J\lambda_{\alpha}, \dot{\gamma}(t)\rangle \langle J\nabla_{\dot{\gamma}(t)}\lambda_{\alpha}, \dot{\gamma}(t)\rangle.$$
Here, we used the fact that the WP metric is K\"ahler, so that $J$ is \emph{parallel} (``commutes with $\nabla$'').

Now, we observe that, by Wolpert's formulas (cf. Theorem \ref{t.Wolpert-expansions}), one can write $\nabla_{\dot{\gamma}(t)}\lambda_{\alpha}$ and $J\nabla_{\dot{\gamma}(t)}\lambda_{\alpha}$ that
$$\nabla_{\dot{\gamma}(t)}\lambda_{\alpha}=\frac{3\langle\dot{\gamma}(t), J\lambda_{\alpha}\rangle}{2\pi f_{\alpha}(t)} J\lambda_{\alpha} + O(f_{\alpha}(t)^3)$$
and
$$J\nabla_{\dot{\gamma}(t)}\lambda_{\alpha}=-\frac{3\langle\dot{\gamma}(t), J\lambda_{\alpha}\rangle}{2\pi f_{\alpha}(t)} \lambda_{\alpha} + O(f_{\alpha}(t)^3)$$

Since $\max\{|\langle\dot{\gamma}(t), \lambda_{\alpha}|, |\langle\dot{\gamma}(t), J\lambda_{\alpha}\rangle|\}\leq r_{\alpha}(t)$ (by definition), we conclude from the previous equations that
\begin{eqnarray*}
2r_{\alpha}(t)r_{\alpha}'(t) &=& \frac{3}{\pi f_{\alpha}(t)}
(\langle \lambda_{\alpha}, \dot{\gamma}(t)\rangle \langle J\lambda_{\alpha}, \dot{\gamma}(t)\rangle^2 - \langle \lambda_{\alpha}, \dot{\gamma}(t)\rangle \langle J\lambda_{\alpha}, \dot{\gamma}(t)\rangle^2) \\ &+& O(r_{\alpha}(t) f_{\alpha}(t)^3) \\
&=& 0 + O(r_{\alpha}(t)f_{\alpha}(t)^3).
\end{eqnarray*}
This proves the lemma.
\end{proof}

\begin{remark} This ODE is an analogue for the WP metric of \emph{Clairaut's relation} for the ``model metric'' on the surface of revolution of the profil $y=x^3$.
\end{remark}

\begin{lemma}[Lemma 4.16 of \cite{BurnsMasurWilkinson}] $$r_{\alpha}(t)^2 = f_{\alpha}'(t)^2 + \frac{2\pi}{3}f_{\alpha}(t) f_{\alpha}''(t)+O(f_{\alpha}(t)^4)$$
\end{lemma}

\begin{proof} By definition, $\lambda_{\alpha}=\textrm{grad}\,\ell_{\alpha}^{1/2}$, so that
$$f_{\alpha}'(t)=\langle\lambda_{\alpha},\dot{\gamma}(t)\rangle.$$
Differentiating this equality and using Wolpert's formulas (Theorem \ref{t.Wolpert-expansions}), we see that
$$f_{\alpha}''(t)=\langle\nabla_{\dot{\gamma}(t)}\lambda_{\alpha}, \dot{\gamma}(t)\rangle =
\frac{3}{2\pi f_{\alpha}(t)}\langle\dot{\gamma}(t), J\lambda_{\alpha}\rangle^2 + O(f_{\alpha}(t)^3)$$
(Here, we used in the first equality the fact that $\gamma$ is a geodesic, i.e., $\ddot{\gamma}(t)=0$.)

It follows that
\begin{eqnarray*}
\frac{2\pi}{3}f_{\alpha}(t)f_{\alpha}''(t) + f_{\alpha}'(t)^2 &=& \langle\dot{\gamma}(t), J\lambda_{\alpha}\rangle^2 +\langle\dot{\gamma}(t), \lambda_{\alpha}\rangle^2 + O(f_{\alpha}(t)^4) \\
&=:& r_{\alpha}(t)^2+O(f_{\alpha}(t)^4).
\end{eqnarray*}
This proves the lemma.
\end{proof}

At this point, the conclusion is that the WP metric (on $\mathcal{M}(S)[3]$) satisfies items (I) to (VI) of Theorem \ref{t.BMW-ergodicity-criterion}, so that the desired ergodicity (and mixing) result of Theorem \ref{t.BMW} follows. 

\section{Decay of correlations for the Weil-Petersson geodesic flow}\label{s.mixing-rate-WP}

Our goal in this section is to discuss the proof of Theorem \ref{t.BMMW} on the rates of mixing of the Weil-Petersson (WP) geodesic flow on the unit tangent bundle $T^1\mathcal{M}_{g,n}$ of the moduli space $\mathcal{M}_{g,n}$ of Riemann surfaces of genus $g\geq 0$ with $n\geq 0$ punctures for $3g-3+n\geq 1$. 

Let us recall that, by Burns-Masur-Wilkinson theorem (cf. Theorem \ref{t.BMW}), the WP flow $\varphi_t$ on $T^1\mathcal{M}_{g,n}$ is \emph{mixing} with respect to the Liouville measure $\mu$ whenever $3g-3+n\geq 1$. 

By definition of the mixing property, this means that the correlation function $C_t(f,g):=\int f \cdot g\circ\varphi_t d\mu - \left(\int f d\mu\right)\left(\int g d\mu\right)$ converges to $0$ as $t\to\infty$ for any given $L^2$-integrable observables $f$ and $g$. (See, e.g., Hasselblatt's text \cite{Hasselblatt})

Given this scenario, it is natural to ask how \emph{fast} the correlation function $C_t(f,g)$ converges to zero. In general, the correlation function $C_t(f,g)$ can decay to $0$ (as a function of $t\to\infty$) in a very slow way depending on the choice of the observables. Nevertheless, it is often the case (for mixing flows with some hyperbolicity) that the correlation function $C_t(f,g)$ decays to $0$ with a \emph{definite} (e.g., polynomial, exponential, etc.) speed when restricting the observables to appropriate spaces of ``reasonably smooth'' functions. 

In other words, given a mixing flow (with some hyperbolicity), it is usually possible to choose appropriate functional (e.g., H\"older, $C^r$, Sobolev, etc.) spaces $X$ and $Y$ such that 
\begin{itemize}
\item $|C_t(f,g)|\leq C\|f\|_{X} \|g\|_{Y} t^{-n}$ for some constants $C>0$, $n\in\mathbb{N}$ and for all $t\geq 1$ (\emph{polynomial decay}), 
\item or $|C_t(f,g)|\leq C\|f\|_{X} \|g\|_{Y} e^{-ct}$ for some constants $C>0$, $c>0$ and for all $t\geq 1$ (\emph{exponential decay}).
\end{itemize}

Evidently, the ``precise'' rate of mixing of the flow (i.e., the sharp values of the constants $C>0$, $n\in\mathbb{N}$ and/or $c>0$ above) depend on the choice of the functional spaces $X$ and $Y$ (as they might change if we replace $C^1$ observables by $C^2$ observables say). On the other hand, the \emph{qualitative} speed of decay of $C_t(f,g)$, that is, the fact that $C_t(f,g)$ decays polynomially or exponentially as $t\to\infty$ whenever $f$ and $g$ are ``reasonably smooth'', tends to remain \emph{unchanged} if we select $X$ and $Y$ from a well-behaved scale of functional (like $C^r$ spaces, $r\in\mathbb{N}$, or $H^s$ spaces, $s>0$). In particular, this partly explains why in the Dynamical Systems literature one simply says that a given mixing flow $\varphi_t$ has ``polynomial decay'' or ``exponential decay'': usually we are interested in the qualitative behavior of the correlation function for reasonably smooth observables, but the particular choice of functional spaces $X$ and $Y$ is normally treated as a ``technical detail''.

After this brief description of the notion of rate of mixing (speed of decay of correlation functions), let us re-state Theorem \ref{t.BMMW} as two separate results (for the sake of convenience)

\begin{theorem}\label{t.BMMW-1} The rate of mixing of the WP flow $\varphi_t$ on $T^1\mathcal{M}_{g,n}$ is at most polynomial when $3g-3+n>1$.
\end{theorem}

\begin{theorem}\label{t.BMMW-2} The rate of mixing of the WP flow $\varphi_t$ on $T^1\mathcal{M}_{g,n}$ is rapid (faster than any polynomial) when $3g-3+n=1$.
\end{theorem}

\begin{remark} These results were announced in \cite{BurnsMasurMatheusWilkinson}. Since then, Burns, Masur, Wilkinson and myself found some evidence indicating that the Weil-Petersson geodesic flow on $T^1\mathcal{M}_{g,n}$ is actually exponentially mixing when $3g-3+n=1$. The details will hopefully appear in the forthcoming paper (currently still in preparation).
\end{remark}

\begin{remark} An open problem left by Theorem \ref{t.BMMW-1} is to determine the rate of mixing of the WP flow on $T^1\mathcal{M}_{g,n}$ for $3g-3+n>1$. Indeed, while this theorem provides a polynomial upper bound for the rate of mixing in this setting, it does not rule out the possibility that the actual rate of mixing of the WP flow is sub-polynomial (even for reasonably smooth observables). Heuristically speaking, we believe that the sectional curvatures of the WP metric control the time spend by WP geodesics near the boundary of $\overline{\mathcal{M}}_{g,n}$. In particular, it seems that the problem of determining the rate of mixing of the WP flow (when $3g-3+n>1$) is somewhat related to the issue of finding suitable (polynomial?) bounds for how close to zero the sectional curvatures of the WP metric can be (in terms of the distance to the boundary 
of $\overline{\mathcal{M}}_{g,n}$). Unfortunately, the best available bounds for the sectional curvatures of the WP metric (due to Wolpert) do not rule out the possibility that some of these quantities get extremely close to zero (see Remark \ref{r.wp-almost-zero-curvatures} above).
\end{remark}

The difference in the rates of mixing of the WP flow on $T^1\mathcal{M}_{g,n}$ when $3g-3+n>1$ or $3g-3+n=1$ in Theorem \ref{t.BMMW} reflects the following simple (yet important) feature of the WP metric near the boundary of the Deligne-Mumford compactification of $\mathcal{M}_{g,n}$. 

In the case $3g-3+n=1$, e.g., $g=1=n$, the moduli space $\mathcal{M}_{1,1}\simeq\mathbb{H}/PSL(2,\mathbb{Z})$ equipped with the WP metric looks like the surface of revolution of the profile $\{v=u^3: 0 < u \leq 1\}$ near the cusp at infinity (see Remark \ref{r.wp-m11-surface-revolution} above). Thus, even though a $\varepsilon$-neighborhood of the cusp is ``polynomially large'' (with area $\sim \varepsilon^4$), the Gaussian curvature approaches only $-\infty$ near the cusp and, as it turns out, this strong negative curvature near the cusp makes that all geodesic not pointing directly towards the cusp actually come back to the compact part in bounded (say $\leq 1$) time. In other words, the excursions of infinite WP geodesics on $\mathcal{M}_{1,1}$ near the cusp are so quick that the WP flow on $T^1\mathcal{M}_{1,1}$ is ``close'' to a classical Anosov geodesic flow on negatively curved compact surface. In particular, it is not entirely surprising that the WP flow on $T^1\mathcal{M}_{1,1}$ is rapid. 

On the other hand, in the case $3g-3+n>1$, the WP metric on $\mathcal{M}_{g,n}$ has \emph{some} sectional curvatures close to \emph{zero} near the boundary of the Deligne-Mumford compactification $\overline{\mathcal{M}}_{g,n}$ of $\mathcal{M}_{g,n}$ (cf. Remark \ref{r.wp-almost-zero-curvatures}). By exploiting this feature of the WP metric on $\mathcal{M}_{g,n}$ for $3g-3+n>1$ (that has no counterpart for $\mathcal{M}_{1,1}$ or $\mathcal{M}_{0,4}$), we will build a \emph{non-neglegible} set of WP geodesics spending a \emph{long} time near the boundary of $\overline{\mathcal{M}}_{g,n}$ before eventually getting into the compact part. In this way, we will deduce that the WP flow on $\mathcal{M}_{g,n}$ takes a fair (polynomial) amount of time to mix certain parts of the boundary of $\overline{\mathcal{M}}_{g,n}$ with fixed compact subsets of $\mathcal{M}_{g,n}$. 

In the remainder of this post, we will give some details of the proof of Theorem \ref{t.BMMW} (or, equivalently, Theorems \ref{t.BMMW-1} and \ref{t.BMMW-2}). In the next subsection, we give a fairly complete proof of the polynomial upper bound on the rate of mixing of the WP flow on $T^1\mathcal{M}_{g,n}$ when $3g-3+n>1$. After that, in the final subsection, we provide a \emph{sketch} of the proof of the rapid mixing property of the WP flow on $T^1\mathcal{M}_{1,1}$. In fact, we decided (for pedagogical reasons) to explain some key points of the rapid mixing property \emph{only} in the \emph{toy model} case of a negatively curved surface with one cusp corresponding \emph{exactly} to a surface of revolution of a profile $\{v=u^r\}$, $r>3$. In this way, since the WP metric near the cusp of $\mathcal{M}_{1,1}\simeq \mathbb{H}/PSL(2,\mathbb{Z})$ can be thought as a ``perturbation'' of the surface of revolution of the ``borderline profile'' $\{v=u^3\}$ with $r=3$ (thanks to Wolpert's asymptotic formulas), the reader hopefully will get a flavor of the main ideas behind the proof of rapid mixing of the WP flow on $\mathcal{M}_{1,1}$ without getting into the (somewhat boring) technical details needed to check that the arguments used in the toy model case are ``sufficiently robust'' so that they can be ``carried over'' to the ``perturbative setting'' of the WP flow on $T^1\mathcal{M}_{1,1}$. 

\subsection{Rates of mixing of the WP flow on $T^1\mathcal{M}_{g,n}$ I: Proof of Theorem \ref{t.BMMW-1}}

In this subsection, our notations are the same as in Section \ref{s.wp-geometry}.

Given $\varepsilon>0$, let us consider the portion of $\mathcal{M}_{g,n}$ consisting of $X\in\mathcal{M}_{g,n}$ such that a non-separating (homotopically non-trivial, non-peripheral) simple closed curve $\alpha$ has hyperbolic length $\ell_{\alpha}(X)\leq (2\varepsilon)^2$. The following picture illustrates this portion of $\mathcal{M}_{g,n}$ as a $(2\varepsilon)^2$-neighborhood of the stratum $\mathcal{T}_{\alpha}/MCG_{g,n}$ of the boundary of the Deligne-Mumford compactification $\overline{\mathcal{M}}_{g,n}$ where $\alpha$ gets pinched (i.e., $\ell_{\alpha}$ becomes zero).

\begin{figure}[htb!]
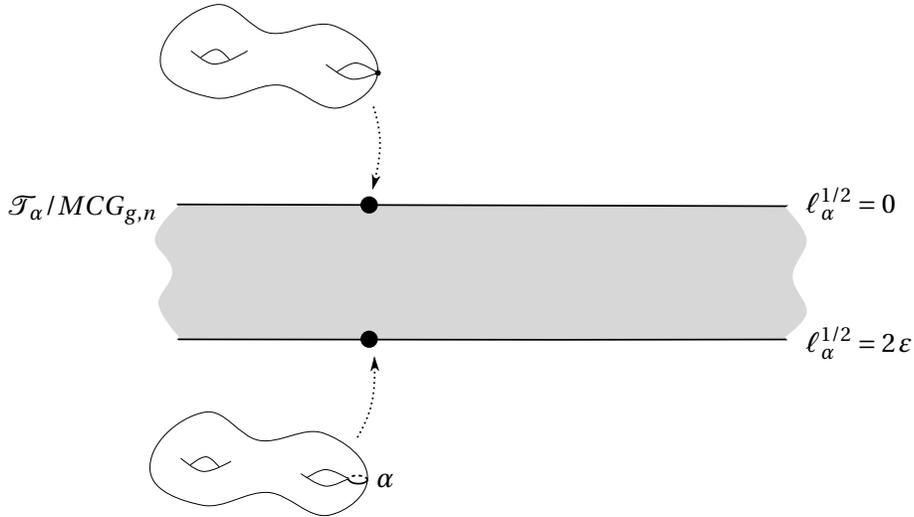
\caption{A portion of the boundary of $\mathcal{M}_{g,n}$ (when $3g-3+n>1$).}\label{f.DM-boundary}
\end{figure}

Note that the stratum $\mathcal{T}_{\alpha}/MCG_{g,n}$ is non-trivial (that is, not reduced to a single point) when $3g-3+n>1$. Indeed, by pinching $\alpha$ as above and by disconnecting the resulting node, we obtain Riemann surfaces of genus $g-1$ with $n+2$ punctures whose moduli space is isomorphic to $\mathcal{T}_{\alpha}/MCG_{g,n}$. It follows that $\mathcal{T}_{\alpha}/MCG_{g,n}$ is a complex orbifold of dimension $3(g-1)+(n+2)=3g-3+n-1>0$, and, a fortiori, $\mathcal{T}_{\alpha}/MCG_{g,n}$ is not trivial. Evidently, this argument breaks down when $3g-3+n=1$: for example, by pinching a curve $\alpha$ as above in a once-punctured torus and by removing the resulting node, we obtain thrice punctured spheres (whose moduli space $\mathcal{M}_{0,3}=\{\overline{\mathbb{C}}-\{0,1,\infty\}\}$ is trivial). In particular, our Figure \ref{f.DM-boundary} concerns \emph{exclusively} the case $3g-3+n>1$. 

We want to locate certain regions near $\mathcal{T}_{\alpha}/MCG_{g,n}$ taking a long time to mix with the compact part of $\mathcal{M}_{g,n}$. For this sake, we will exploit the geometry of the WP metric near $\mathcal{T}_{\alpha}/MCG_{g,n}$ -- e.g., Wolpert's formulas in Theorem \ref{t.Wolpert-expansions}-- to build nice sets of unit vectors traveling in an ``almost parallel'' way to $\mathcal{T}_{\alpha}/MCG_{g,n}$ for a significant amount of time. 

More precisely, we consider the vectors $\lambda_{\alpha} := \textrm{grad}(\ell_{\alpha}^{1/2})$ and $J\lambda_{\alpha}$ (where $J$ is the complex structure). By definition, they span a complex line $L=\textrm{span}\{\lambda_{\alpha}, J\lambda_{\alpha}\}$. Intuitively, the complex line $L$  points in the normal direction to a ``copy'' of $\mathcal{T}_{\alpha}/MCG_{g,n}$ inside a level set of the function $\ell_{\alpha}^{1/2}$ as indicated below:

\begin{figure}[htb!]
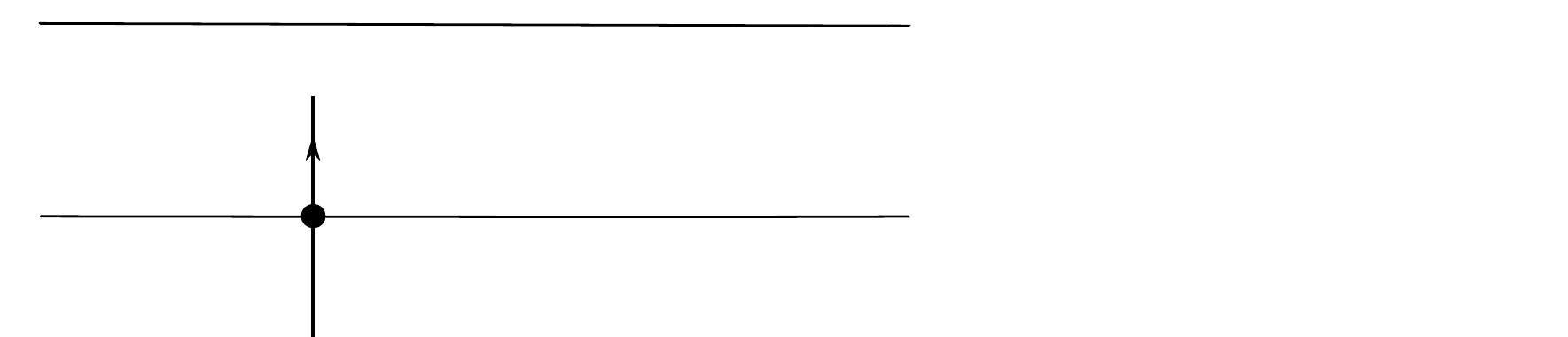
\end{figure}

Using the complex line $L$, we formalize the notion of ``almost parallel'' vector to $\mathcal{T}_{\alpha}/MCG_{g,n}$. Indeed, given $v\in T^1\mathcal{M}_{g,n}$, let us denote by $r_{\alpha}(v)$ the quantity $r_{\alpha}(v):=\sqrt{\langle v, \lambda_{\alpha} \rangle^2 + \langle v, J\lambda_{\alpha}\rangle^2}$ (where $\langle.,.\rangle$ is the WP metric). By definition, $r_{\alpha}(v)$ measures the size of the projection of the unit vector $v$ in the complex line $L$. In particular, we can think of $v$ as ``almost parallel'' to $\mathcal{T}_{\alpha}/MCG_{g,n}$ whenever the quantity $r_{\alpha}(v)$ is very close to zero.

In this setting, we will show that unit vectors almost parallel to $\mathcal{T}_{\alpha}/MCG_{g,n}$ whose footprints are close to $\mathcal{T}_{\alpha}/MCG_{g,n}$ always generate geodesics staying near $\mathcal{T}_{\alpha}/MCG_{g,n}$ for a long time. More concretely, given $\varepsilon>0$, let us define the set 
$$V_{\varepsilon} := \{v\in T^1\mathcal{M}_{g,n}: f_{\alpha}(v)\leq \varepsilon, \, r_{\alpha}(v)\leq \varepsilon^2\}$$
where $f_{\alpha}(v) := \ell_{\alpha}^{1/2}(p)$ and $p\in\mathcal{M}_{g,n}$ is the footprint of the unit vector $v\in T^1\mathcal{M}_{g,n}$. Equivalently, $V_{\varepsilon}$ is the disjoint union of the pieces of spheres $S_{\varepsilon}(p):=\{v\in T^1_p\mathcal{M}_{g,n}: r_{\alpha}(v)\leq \varepsilon^2\}$ attached to points $p\in\mathcal{M}_{g,n}$ with $\ell_{\alpha}(p)\leq \varepsilon^2$. The following figure summarizes the geometry of $S_{\varepsilon}(p)$:

\begin{figure}[htb!]
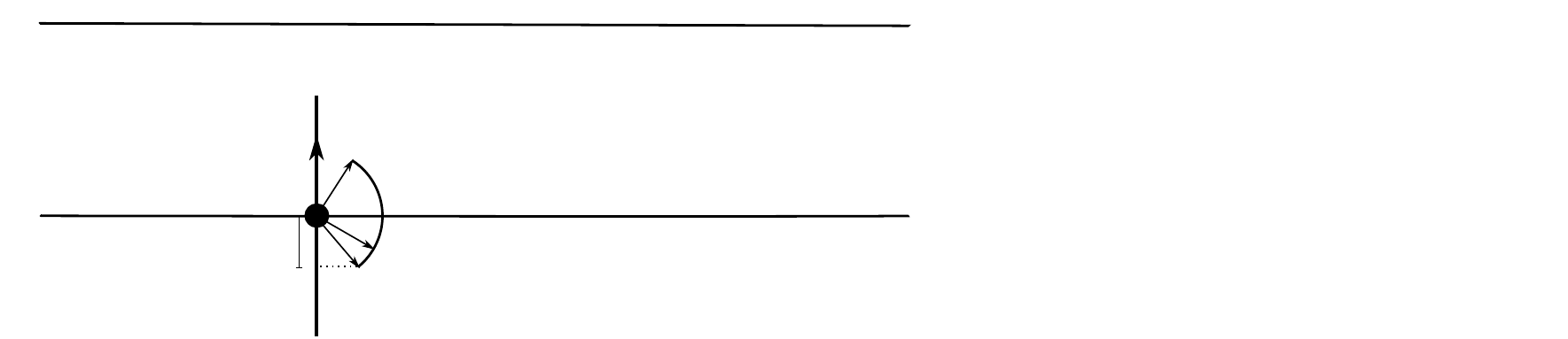
\end{figure}

We would like to prove that a geodesic $\gamma_v(t)$ originating at any $v\in V_{\varepsilon}$ stays in a $(2\varepsilon)^2$-neighborhood of $\mathcal{T}_{\alpha}/MCG_{g,n}$ for an interval of  time $[0, T]$ of size of order $1/\varepsilon$, so that the WP geodesic flow does \emph{not} mix $V_{\varepsilon}$ with any fixed ball $U$ in the compact part of $\mathcal{M}_{g,n}$ of Riemann surfaces with systole $> (2\varepsilon)^2$: 

\begin{figure}[htb!]
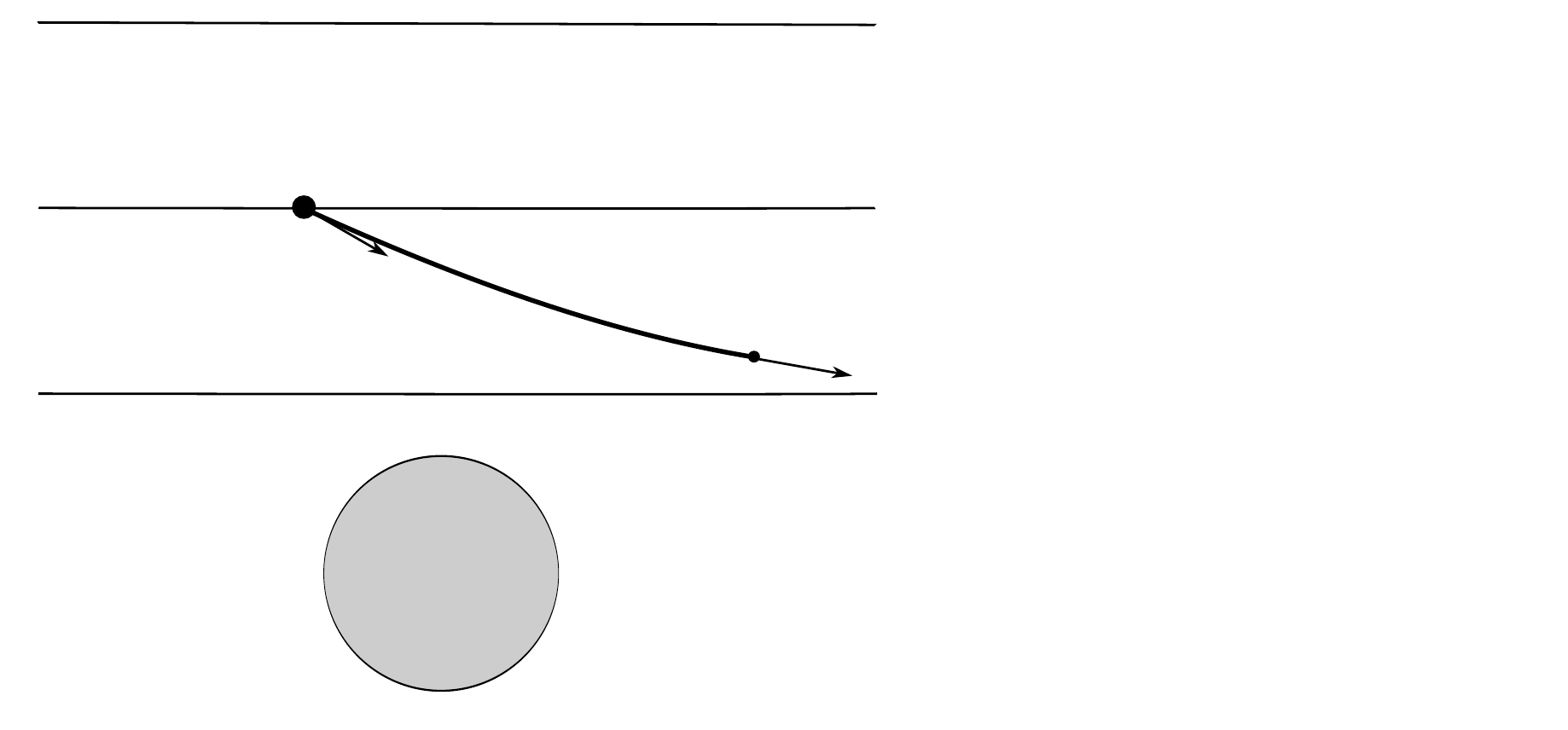
\end{figure}

In this direction, we will need the following estimate from Lemma \ref{l.Clairaut-ODE} above: given $\gamma(t)$ be a WP geodesic as above, and denoting by $r_{\alpha}(t)=r_{\alpha}(\dot{\gamma}(t))$ and $f_{\alpha}(t)=\ell_{\alpha}^{1/2}(\gamma(t))$,  then 
$$r_{\alpha}'(t) = O(f_{\alpha}(t)^3)$$

From this inequality, it is not hard to estimate the amount of time spent by a geodesic $\gamma_v(t)$ near $\mathcal{T}_{\alpha}/MCG_{g,n}$ for an arbitrary 
$v\in V_{\varepsilon}$:

\begin{lemma}\label{l.WP-mixing-time-lower-bound} There exists a constant $C_0>0$ (depending only on $g$ and $n$) such that 
$$\ell_{\alpha}^{1/2}(\gamma_v(t))=f_{\alpha}(t)\leq 2\varepsilon$$
for all $v\in V_{\varepsilon}$ and $0\leq t\leq 1/C_0\varepsilon$. 
\end{lemma}

\begin{proof} By definition, $v\in V_{\varepsilon}$ implies that $f_{\alpha}(0)\leq \varepsilon$. Thus, it makes sense to consider the maximal interval $[0,T]$ of time such that $f_{\alpha}(t)\leq 2\varepsilon$ for all $0\leq t\leq T$. 

By Lemma \ref{l.Clairaut-ODE}, we have that $r_{\alpha}'(s)=O(f_{\alpha}(s)^3)$, i.e., $|r_{\alpha}'(s)|\leq B f_{\alpha}(s)^3$ for some constant $B>1/4$ depending only on $g$ and $n$. In particular, $|r_{\alpha}'(s)|\leq B f_{\alpha}(s)^3\leq B (2\varepsilon)^3$ for all $0\leq s\leq T$. From this estimate, we deduce that 
$$r_{\alpha}(t) = r_{\alpha}(0)+\int_0^t r_{\alpha}'(s)\,ds\leq r_{\alpha}(0)+B(2\varepsilon)^3 t = r_{\alpha}(0)+8B\varepsilon^3t$$
for all $0\leq t\leq T$. Since the fact that $v\in V_{\varepsilon}$ implies that $r_{\alpha}(0)\leq\varepsilon^2$, the previous inequality tell us that 
$$r_{\alpha}(t)\leq \varepsilon^2+8B\varepsilon^3t$$
for all $0\leq t\leq T$. 

Next, we observe that, by definition, $f_{\alpha}'(t)=\langle\dot{\gamma}(t),\textrm{grad}\ell_{\alpha}^{1/2}\rangle = \langle\dot{\gamma}(t),\lambda_{\alpha}\rangle$. Hence, 
$$|f_{\alpha}'(t)|=|\langle\dot{\gamma}(t),\lambda_{\alpha}\rangle|\leq \sqrt{\langle\dot{\gamma}(t),\lambda_{\alpha}\rangle^2 + \langle\dot{\gamma}(t),J\lambda_{\alpha}\rangle^2} = r_{\alpha}(t)$$ 
By putting together the previous two inequalities and the fact that $f_{\alpha}(0)\leq\varepsilon$ (as $v\in V_{\varepsilon}$), we conclude that 
$$f_{\alpha}(T) = f_{\alpha}(0)+\int_0^T f_{\alpha}'(t) \, dt\leq \varepsilon + \varepsilon^2 T + 4B\varepsilon^3T^2$$
Since $T>0$ was chosen so that $[0,T]$ is the maximal interval with $f_{\alpha}(t)\leq 2\varepsilon$ for all $0\leq t\leq T$, we have that $f_{\alpha}(T)=2\varepsilon$. Therefore, the previous estimate can be rewritten as 
$$2\varepsilon\leq \varepsilon + \varepsilon^2T + 4B\varepsilon^3T^2$$
Because $B>1/4$, it follows from this inequality that $T\geq 1/C_0\varepsilon$ where $C_0:=8B$. 

In other words, we showed that $[0,1/C_0\varepsilon]\subset [0,T]$, and, \emph{a fortiori}, $f_{\alpha}(t)\leq 2\varepsilon$ for all $0\leq t\leq 1/C_0\varepsilon$. This completes the proof of the lemma.
\end{proof}

Once we have Lemma \ref{l.WP-mixing-time-lower-bound} in our toolbox, it is not hard to infer some upper bounds on the rate of mixing of the WP flow on $T^1\mathcal{M}_{g,n}$ when $3g-3+n>1$. 

\begin{proposition}\label{p.WP-poly-mixing} Suppose that the WP flow $\varphi_t$ on $T^1\mathcal{M}_{g,n}$ has a rate of mixing of the form 
$$C_t(a,b) = \left|\int a\cdot b\circ\varphi_t - \left(\int a\right)\left(\int b\right)\right|\leq C t^{-\gamma}\|a\|_{C^1}\|b\|_{C^1}$$
for some constants $C>0$, $\gamma>0$, for all $t\geq 1$, and for all choices of $C^1$-observables $a$ and $b$.

Then, $\gamma\leq 10$, i.e., the rate of mixing of the WP flow is at most polynomial. 
\end{proposition}

\begin{proof} Let us fix once and for all an open ball $U$ (with respect to the WP metric) contained in the compact part of $\mathcal{M}_{g,n}$: this means that there 
exists $\varepsilon_0>0$ such that the systoles of all Riemann surfaces in $U$ are $\geq\varepsilon_0^2$. 

Take a $C^1$ function $a$ supported on the set $T^1U$ of unit vectors with footprints on $U$ with values $0\leq a\leq 1$ such that $\int a\geq \textrm{vol}(U)/2$ and $\|a\|_{C^1}=O(1)$: such a function $a$ can be easily constructed by smoothing the characteristic function of $U$ with the aid of bump functions. 
  
Next, for each $\varepsilon>0$, take a $C^1$ function $b_{\varepsilon}$ supported on the set $V_{\varepsilon}$ with values $0\leq b_{\varepsilon}\leq 1$ such that $\int b_{\varepsilon}\geq \textrm{vol}(V_{\varepsilon})/2$ and $\|b_{\varepsilon}\|_{C^1}=O(1/\varepsilon^2)$: such a function $b_{\varepsilon}$ can also be constructed by smoothing the characteristic function of $V_{\varepsilon}$ after taking into account the description of the WP metric near $\mathcal{T}_{\alpha}/MCG_{g,n}$ given by Theorems \ref{t.Wolpert2008} and \ref{t.Wolpert-expansions} above and the definition of $V_{\varepsilon}$ (in terms of the conditions $\ell_{\alpha}^{1/2}\leq \varepsilon$ and $r_{\alpha}\leq\varepsilon^2$). Furthermore, this description of the WP metric $g_{WP}$ near $\mathcal{T}_{\alpha}/MCG_{g,n}$ combined with the asymptotic expansion $g_{WP}\sim 4dx_{\alpha}^2+x_{\alpha}^6d\tau_{\alpha}$ where $x_{\alpha}:=\ell_{\alpha}^{1/2}/\sqrt{2\pi^2}$ and $\tau_{\alpha}$ is a twist parameter (see the proof of Lemma \ref{l.wp-DM-bdry-nbhd-vol}) says that $\textrm{vol}(V_{\varepsilon})\sim \varepsilon^8$: indeed, the condition $f_{\alpha}=\ell_{\alpha}^{1/2}\leq\varepsilon$ on footprints of unit tangent vectors in $V_{\varepsilon}$ provides a set of volume $\sim \varepsilon^4$ (cf. the proof of Lemma 4 of the aforementioned post for details) and the condition $r_{\alpha}\leq\varepsilon^2$ on unit tangent vectors in $V_{\varepsilon}$ with a fixed footprint provides a set of volume comparable to the Euclidean area $\pi\varepsilon^4$ of the Euclidean ball $\{\vec{v}\in\mathbb{R}^2: |v|\leq\varepsilon^2\}$ (cf. Theorem \ref{t.Wolpert2008}), so that 
$$\textrm{vol}(V_{\varepsilon}) = \int_{\{\ell_{\alpha}^{1/2}(p)\leq\varepsilon\}} \textrm{vol}(\{v\in T^1_p\mathcal{M}_{g,n}: r_{\alpha}(v)\leq\varepsilon^2\})\sim (\pi\varepsilon^4)\cdot \varepsilon^4\sim \varepsilon^8$$ 

In summary, for each $\varepsilon>0$, we have a $C^1$ function $b_{\varepsilon}$ supported on $V_{\varepsilon}$ with $0\leq b\leq 1$, $\|b_{\varepsilon}\|_{C^1} = O(1/\varepsilon^2)$ and $\int b_{\varepsilon}\geq c_0\varepsilon^8$ for some constant $c_0>0$ depending only on $g$ and $n$. 

Our plan is to use the observables $a$ and $b_{\varepsilon}$ to give some upper bounds on the mixing rate of the WP flow $\varphi_t$. For this sake, suppose that there are constants $C>0$ and $\gamma>0$ such that 
$$C_t(a,b_{\varepsilon})=\left|\int a\cdot b_{\varepsilon}\circ\varphi_t - \left(\int a\right)\left(\int b_{\varepsilon}\right)\right|\leq 
C t^{-\gamma}\|a\|_{C^1}\|b_{\varepsilon}\|_{C^1}$$ 
for all $t\geq 1$ and $\varepsilon>0$. 

By Lemma \ref{l.WP-mixing-time-lower-bound}, there exists a constant $C_0>0$ such that $V_{\varepsilon}\cap \varphi_{\frac{1}{C_0\varepsilon}}(T^1U)=\emptyset$ whenever $2\varepsilon<\varepsilon_0$. Indeed, since $V_{\varepsilon}$ is a symmetric set (i.e., $v\in V_{\varepsilon}$ if and only if $-v\in V_{\varepsilon}$), it follows from Lemma \ref{l.WP-mixing-time-lower-bound} that all Riemann surfaces in the footprints of $\varphi_{-\frac{1}{C_0\varepsilon}}(V_{\varepsilon})$ have a systole $\leq (2\varepsilon)^2<\varepsilon_0^2$. Because we took $U$ in such a way that all Riemann surfaces in $U$ have systole $\geq \varepsilon_0^2$, we obtain $\varphi_{-\frac{1}{C_0\varepsilon}}(V_{\varepsilon})\cap T^1U=\emptyset$, that is, $V_{\varepsilon}\cap \varphi_{\frac{1}{C_0\varepsilon}}(T^1U)=\emptyset$, as it was claimed. 

Now, let us observe that the function $a\cdot b_{\varepsilon}\circ \varphi_t$ is supported on $V_{\varepsilon}\cap \varphi_t(T^1U)$ because $a$ is supported on $T^1U$ and $b_{\varepsilon}$ is supported on $V_{\varepsilon}$. By putting together this fact and the claim in the previous paragraph (that $V_{\varepsilon}\cap \varphi_{\frac{1}{C_0\varepsilon}}(T^1U)=\emptyset$ for $2\varepsilon<\varepsilon_0$), we deduce that $a\cdot b_{\varepsilon}\circ \varphi_{\frac{1}{C_0\varepsilon}}\equiv 0$ whenever $2\varepsilon<\varepsilon_0$. Thus,   
$$C_{\frac{1}{C_0\varepsilon}}(a,b_{\varepsilon}) := \left|\int a\cdot b_{\varepsilon}\circ\varphi_{\frac{1}{C_0\varepsilon}} - \left(\int a\right)\left(\int b_{\varepsilon}\right)\right| = 
\left(\int a\right)\left(\int b_{\varepsilon}\right)$$ 

By plugging this identity into the polynomial decay of correlations estimate $C_t(a,b_{\varepsilon})\leq Ct^{-\gamma}\|a\|_{C^1}\|b_{\varepsilon}\|_{C^1}$, we get 
$$\left(\int a\right)\left(\int b_{\varepsilon}\right) = C_{\frac{1}{C_0\varepsilon}}(a,b_{\varepsilon})\leq CC_0^{\gamma}\varepsilon^{\gamma}\|a\|_{C^1}\|b_{\varepsilon}\|_{C^1}$$
whenever $2\varepsilon<\varepsilon_0$ and $1/C_0\varepsilon\geq 1$. 

We affirm that the previous estimate implies that $\gamma\leq 10$. In fact, recall that our choices were made so that $\int a\geq \textrm{vol}(U)/2$ where $U$ is a fixed ball, $\|a\|_{C^1}=O(1)$, $\int b_{\varepsilon} \geq c_0\varepsilon^8$ for some constant $c_0>0$ and $\|b_{\varepsilon}\|_{C^1}=O(1/\varepsilon^2)$. Hence, by combining these facts and the previous mixing rate estimate, we get that 
$$\left(\frac{\textrm{vol}(U)}{2}\right) c_0\varepsilon^8\leq \left(\int a\right)\left(\int b_{\varepsilon}\right)\leq CC_0^{\gamma}\varepsilon^{\gamma}\|a\|_{C^1}\|b_{\varepsilon}\|_{C^1} = O(\varepsilon^{\gamma}\frac{1}{\varepsilon^2}),$$
that is, $\varepsilon^{10}\leq D \varepsilon^{\gamma}$, for some constant $D>0$ and for all $\varepsilon>0$ sufficiently small (so that $2\varepsilon<\varepsilon_0$ and $1/C_0\varepsilon\geq 1$). It follows that $\gamma\leq 10$, as we claimed. This completes the proof of the proposition. 
\end{proof}

\begin{remark}\label{r.poly-mixing-Ck} In the statement of the previous proposition, the choice of $C^1$-norms to measure the rate of mixing of the WP flow is not very important. Indeed, an inspection of the construction of the functions $b_{\varepsilon}$ in the argument above reveals that $\|b_{\varepsilon}\|_{C^{k+\alpha}} = O(1/\varepsilon^{k+\alpha})$ for any $k\in\mathbb{N}$, $0\leq\alpha<1$. In particular, the proof of the previous proposition is sufficiently robust to show also that a rate of mixing of the form
$$C_t(a,b) = \left|\int a\cdot b\circ\varphi_t - \left(\int a\right)\left(\int b\right)\right|\leq C t^{-\gamma}\|a\|_{C^{k+\alpha}}\|b\|_{C^{k+\alpha}}$$
for some constants $C>0$, $\gamma>0$, for all $t\geq 1$, and for all choices of $C^1$-observables $a$ and $b$ holds only if $\gamma\leq 8+2(k+\alpha)$. 

In other words, even if we replace $C^1$-norms by (stronger, smoother) $C^{k+\alpha}$-norms in our measurements of rates of mixing of the WP flow (on $T^1\mathcal{M}_{g,n}$ for $3g-3+n>1$), our discussions so far will always give polynomial upper bounds for the decay of correlations. 
\end{remark}

At this point, our discussion of the proof of Theorem \ref{t.BMMW-1} (i.e., the first item of Theorem \ref{t.BMMW}) is complete thanks to Proposition \ref{p.WP-poly-mixing} and Remark \ref{r.poly-mixing-Ck}. So, we will now move on to the next subsection where we give some of the key ideas in the proof of Theorem \ref{t.BMMW-2} (i.e., the second item of Theorem \ref{t.BMMW}).

\subsection{Rates of mixing of the WP flow on $T^1\mathcal{M}_{g,n}$ II: Proof of Theorem \ref{t.BMMW-2}}

Let us consider the WP flow on $T^1\mathcal{M}_{g,n}$ when $3g-3+n=1$, that is, when $(g,n)=(0,4)$ or $(1,1)$. 

Actually, we will restrict our attention to the case $(g,n)=(1,1)$ because the remaining case $(g,n)=(0,4)$ is very similar to $(g,n)=(1,1)$. 

Indeed, the moduli space $\mathcal{M}_{0,4}$ of four-times punctured spheres is a \emph{finite} cover of the moduli space $\mathcal{M}_{1,1}\simeq \mathbb{H}/SL(2,\mathbb{Z})$: this can be seen by sending each four-punctured sphere $\overline{\mathbb{C}}-\{x_1,\dots, x_4\}$ to the elliptic curve $y^2=(x-x_1)\dots(x-x_4)$, so that $\mathcal{M}_{0,4}$ becomes naturally isomorphic to $\mathbb{H}/\Gamma_0(2)$ where $\Gamma_0(2)$ is a congruence subgroup of $SL(2,\mathbb{Z})$ of level $2$ with index $3$. Since all arguments towards rapid mixing of geodesic flows in this section still work after taking finite covers, it suffices to prove Theorem \ref{t.BMMW-2} for the WP flow on $T^1\mathcal{M}_{1,1}$. 

The rate of mixing of a geodesic flow on the unit tangent bundle of a negatively curved \emph{compact} surface is known to be \emph{fast}: indeed, Chernov \cite{Chernov} used his technique of ``Markov approximations'' to show \emph{stretched exponential} decay of correlations, and Dolgopyat \cite{Dolgopyat} added a new crucial ingredient (``Dolgopyat's estimate'') to Chernov's work to prove \emph{exponential} decay of correlations. 

Evidently, these works of Chernov and Dolgopyat can not be applied to the WP flow on $T^1\mathcal{M}_{1,1}$ because of the non-compactness of $\mathcal{M}_{1,1}\sim\mathbb{H}/SL(2,\mathbb{Z})$ due to the presence of a (single) cusp (at infinity). Nevertheless, this suggests that we should be able to determine the rate of mixing of the WP flow on $T^1\mathcal{M}_{1,1}$ provided we have enough control of the geometry of the WP metric near the cusp. 

Fortunately, as we mentioned in Example \ref{ex.wolpert-asymptotics} above, Wolpert showed that the WP metric $g_{WP}$ on $\mathcal{M}_{1,1}\simeq \mathbb{H}/SL(2,\mathbb{Z})$ has an \emph{asymptotic} expansion $g_{WP}^2\sim \frac{|dz|^2}{\textrm{Im}(z)}$ at a point $z\in\mathbb{H}$. Thus, the WP metric on neighborhoods $\{z=x+iy\in\mathbb{H}: |x|\leq 1/2, y>y_0\}/SL(2,\mathbb{Z})$ (with $y_0>1$) of the cusp at infinity of $\mathcal{M}_{1,1}$ becomes closer (as $y_0\to\infty$) to the metric of surface of revolution of the profile $v=u^3$ on neighborhoods $\{v=u^3: 0\leq u<u_0\}$ of the cusp at $0$ (as $u_0\to 0$). 

Partly motivated by the scenario of the previous paragraph, from now on we will \emph{pretend} that the WP metric on $\mathbb{H}/PSL(2,\mathbb{Z})$ looks \emph{exactly} like the metric $\frac{|dz|^2}{\textrm{Im}(z)}$ at all points $\{z\in\mathbb{H}: \textrm{Im}(z)>y_0\}$ for some $y_0\gg 1$. In other words, instead of studying the WP flow on $T^1\mathcal{M}_{1,1}$, we will focus on the rates of mixing of the following \emph{toy model}: the geodesic flow on a negatively curved surface $S$ with a single cusp possessing a neighborhood where the metric is isometric to the surface of revolution of a profile $\{v=u^r\}$ for a fixed real number $r>3$. 

\begin{remark} The surface of revolution modeling the WP metric on $T^1\mathcal{M}_{1,1}$ is obtained by rotating the profile $\{v=u^3\}$. In other words, we see that the study of rates of mixing of the surface of revolution approximating the WP metric on $T^1\mathcal{M}_{1,1}$ is a ``borderline case'' in our subsequent discussion. 
\end{remark}

Here, our main motivations to replace the WP flow $\varphi_t$ on $T^1\mathcal{M}_{1,1}$ by the toy model described above are:
\begin{itemize}
\item all important ideas for the study of rates of mixing of $\varphi_t$ are also present in the case of the toy model, and
\item even though the WP metric on $\mathcal{M}_{1,1}$ is a perturbation of a surface of revolution, the verification of the fact that the arguments used to estimate the decay of correlations of the geodesic flow on the toy model surfaces are robust enough so that they can be carried over the WP metric situation is somewhat boring: basically, besides performing a slight modification of the proofs to include the borderline case $r=3$, one has to introduce ``error terms'' in the whole discussion below and, after that, one has to check that these errors terms do not change the qualitative nature of all estimates. 
\end{itemize} 

In summary, the remainder of this subsection will contain a proof of the following ``toy model version'' of Theorem \ref{t.BMMW-2}.

\begin{theorem}\label{t.rapid-mixing-toy-model} Let $\overline{S}$ be a compact surface and fix $0\in\overline{S}$. Suppose that $S=\overline{S}-\{0\}$ is equipped with a negatively curved Riemannian metric $g$ such that the restriction of $g$ to a neighborhood of $\{p\in S: d(p,0) < \rho_0\}$ is isometric to a surface of revolution of a profile $\{v=u^r:0<u\leq u_0\}$ (for some choices of $\rho_0>0$ and $u_0>0$). 

Then, the geodesic flow (associated to $g$) on $T^1S$ is rapid (faster than polynomial) mixing in the sense that, for all $n\in\mathbb{N}$, one can choose an adequate Banach space $X_n$ of ``reasonably smooth'' observables and a constant $C_n>0$ so that 
$$C_t(a,b)=\left|\int a\cdot b\circ\varphi_t - \left(\int a\right)\left(\int b\right)\right|\leq C_n t^{-n}\|a\|_{X_n}\|b\|_{X_n}$$ 
for all $t\geq 1$. 
\end{theorem}

\begin{remark} The arguments below show that the statement above also holds when $S=\overline{S}-\{0_1,\dots, 0_k\}$ is equipped with a negatively curved metric that is isometric to a surface of revolution $\{v=u^{r_i}\}$, $r_i>3$, near $0_i$ for each $i=1,\dots, k$.
\end{remark}

\begin{remark} The Riemannian metric $g$ is incomplete because the surface of revolution of $\{v=u^r\}$ is incomplete when $r>1$ (as the reader can check via a simple calculation). 
\end{remark}

Recall that, in the setting of Theorem \ref{t.rapid-mixing-toy-model}, we want to understand the dynamics of the excursions of the geodesic flow near the cusp $0$ (in order to get rapid mixing). For this sake, we describe these excursions by rewriting the geodesic flow (near $0$) as a \emph{suspension flow}. 

\subsubsection{Excursions near the cusp and suspension flows}

Consider a small neighborhood in $S$ of $0$ where the metric is isometric to the surface of revolution of the profile $\{v=u^r: 0<u\leq u_0\}$, i.e., 
$$\{(x, x^r\cos y, x^r\sin y)\in\mathbb{R}^3: 0<x\leq u_0, 0\leq y\leq 2\pi\}$$ 
Next, take $0<d_0<u_0$ a small parameter and consider the parallel $C=C(d_0) = \{(d_0, d_0^r\cos y, d_0^r\sin y)\in \mathbb{R}^3: 0\leq y\leq 2\pi\}$. We parametrize unit tangent vectors to the surface of revolution with footprints in $C$ as follows. 

Given $q=(d_0,d_0^r\cos y_0, d_0^r\sin y_0)\in C$, we denote by $V=V(q)\in T_q^1S$ the unique unit tangent vector pointing towards to the cusp $O$ at $x=0$. Equivalently, $V$ is the unit vector tangent to the meridian $\{(d_0-t, (d_0-t)^r\cos y_0, (d_0-t)^r\sin y_0)\in\mathbb{R}^3: 0\leq t< d_0\}$ at time $t=0$, or, alternatively, $V(q)=-\nabla d(q)$ where $d(p)=\textrm{dist}(O,p)$ is the distance function from the cusp $O$ to a point $p$. Also, we let $JV=JV(q)$ be the unit vector obtained by rotating $V$ by $\pi/2$ in the counterclockwise sense (i.e., by applying the natural almost complex structure $J$). 

\begin{figure}[htb!]
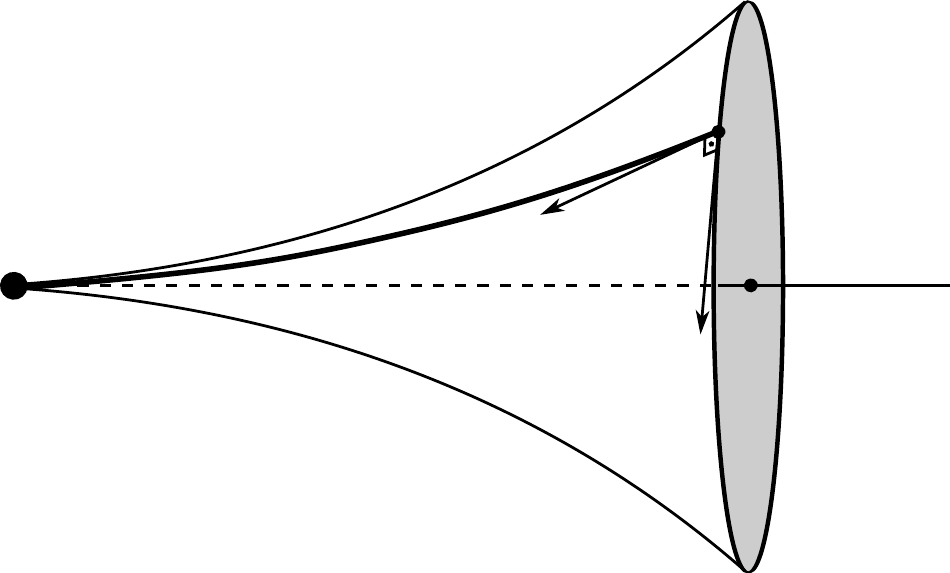
\end{figure}

\newpage

In this setting, an unit vector $v\in T_q^1S$ pointing towards the cusp $O$ is completely determined by a real number $\beta\in (-\pi/2, \pi/2)$ such that $\langle v, V\rangle = \cos\beta$ and $\langle v, JV \rangle = \sin\beta$, i.e., 
$$v = \cos\beta\cdot V + \sin\beta\cdot JV:=v(\beta)$$

The \emph{qualitative} behavior of the excursion of a geodesic $\gamma(t)=(x(t),x(t)^r\cos y(t),x(t)^r\sin y(t))$ starting at $\dot{\gamma}(0)=v(\beta)\in T^1_qS$ can be easily determined in terms of the parameter $\beta$ thanks to the classical results in Differential Geometry about surfaces of revolutions. Indeed, it is well-known (see, e.g., Do Carmo's book \cite{DoCarmo}) that such a geodesic $\gamma(t)$ satisfies 
$$x(t)^{2r}y'(t) = c$$
and 
$$(1+r^2 x(t)^{2(r-1)})x'(t)^2 + \frac{c^2}{x(t)^{2r}}=1$$
for a certain constant $c$, and, furthermore, these relations imply the famous \emph{Clairaut's relation}:
\begin{equation}
x(t)^r \cos|\frac{\pi}{2}-|\beta(t)|| = c = constant
\end{equation}
where $\beta(t)$ is the parameter attached to $\dot{\gamma}(t)$ (i.e., $\dot{\gamma}(t)=v(\beta(t))\in T^1_{\gamma(t)}C(x(t))$). In particular, except for the geodesic going directly to the cusp (i.e., the geodesic starting at $V(q)$ associated to $\beta=0$), all geodesics $\gamma(t)$ (starting at $v(\beta)$ with $\beta\neq0$) behave qualitatively in a simple way. In the first part $t\in [0,T(\beta)/2]$ of its excursion towards the cusp, the angle $\beta(t)$ increases (resp. decreases) from $\beta>0$ to $\pi/2$ (resp. from $\beta<0$ to $-\pi/2$) while the value of $x(t)$ diminishes in order to keep up with Clairaut's relation. Then, the geodesic $\gamma(t)$ reaches its closest position to the cusp at time $t=T(\beta)/2$: here, 
$\beta(t)=\pm\pi/2$ (i.e., $\dot{\gamma}(T(\beta)/2)$ is tangent to the parallel $C(x(T(\beta)/2))$ containing $\gamma(T(\beta)/2)$) and, hence,  
$$x(T(\beta)/2)^r = x(0)^r\sin\beta = d_0^r\sin\beta:=x_{\min}(\beta)^r$$ 
Finally, in the second part $t\in [T(\beta)/2, T(\beta)]$, $\gamma(t)$ does the ``opposite'' from the first part: the angle $\beta(t)$ goes from $\pm\pi/2$ to $\pm\pi/2-\beta$ and $x(t)$ increases from $x_{\min}(\beta)$ back to $x(0)=d_0$. The following picture summarizes the discussion of this paragraph: 

\begin{figure}[htb!]
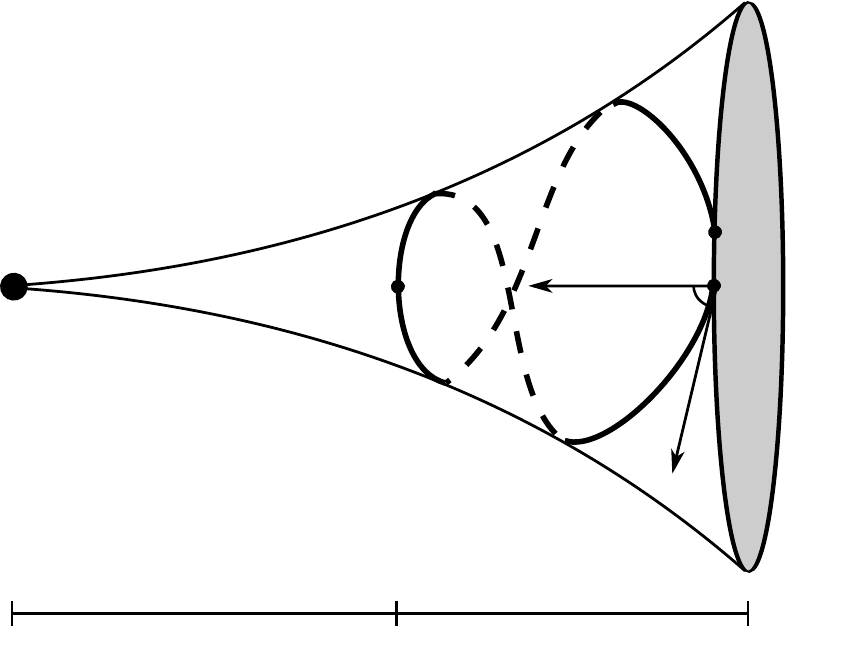
\end{figure}

\begin{remark}\label{r.return-time} Note that the time $T(\beta)$ taken by the geodesic $\gamma(t)$ to go from the parallel $C=C(d_0)$ to $C(x_{\min}(\beta))$ and then from $C(x_{\min}(\beta))$ back to $C$ is \emph{independent} of the basepoint $q=\gamma(0)\in C$. Indeed, this is a direct consequence of the rotational symmetry of our surface. Alternatively, this can be easily seen from the formula 
$$\frac{T(\beta)}{2} = \int_{x_{\min}(\beta)}^{x(0)}x^{r}\sqrt{\frac{1+(rx^{r-1})^2}{x^{2r}-c^2}} \, dx = \int_{d_0(\sin\beta)^{1/r}}^{d_0}x^{r}\sqrt{\frac{1+(rx^{r-1})^2}{x^{2r}-(d_0^r\sin\beta)^2}} \, dx$$
deduced by integration of the ODE satisfied by $x(t)$. Observe that this formula also shows that $T(\beta)$ is uniformly bounded, i.e., $T(\beta)=O_{d_0, r}(1)$ for all $\beta\neq 0$. Geometrically, this means that all geodesics $\gamma(t)$ starting at $C$ must return to $C$ in bounded time unless they go directly into the cusp. 
\end{remark}

This description of the excursions of geodesics near the cusp permits to build a \emph{suspension-flow} model of the geodesic flow near $O$. Indeed, let us consider the \emph{cross-section} $N=T^1_CS = T^1_{C(d_0)}S$. As we saw above, an element of the surface $N$ is parametrized by two angular coordinates $y$ and $\beta$: the value of $y$ determines a point $q=(d_0, d_0^r\cos y, d_0^r\sin y)\in C$ and the value of $\beta$ determines an unit tangent vector $v(\beta)\in T^1_qS$ making angle $\beta$ with $V(q)$. The subset $M$ of $N$ consisting of those elements $v(\beta)$ with angular coordinate $-\pi/2<\beta<\pi/2$ corresponds to the unit vectors with footprint in $C$ pointing towards the cusp at $O$. The equation $\beta=0$ determines a circle $\Sigma$ inside $M$ corresponding to geodesics going straight into the cusp, and, furthermore, we have a natural ``first-return map'' $F:M-\Sigma\to N$ defined by $F(v(\beta)) = \dot{\gamma}_{v(\beta)}(T(\beta))$ where $\gamma_{v(\beta)}$ is the geodesic starting at $v(\beta)$ at time $t=0$. 

\begin{figure}[htb!]
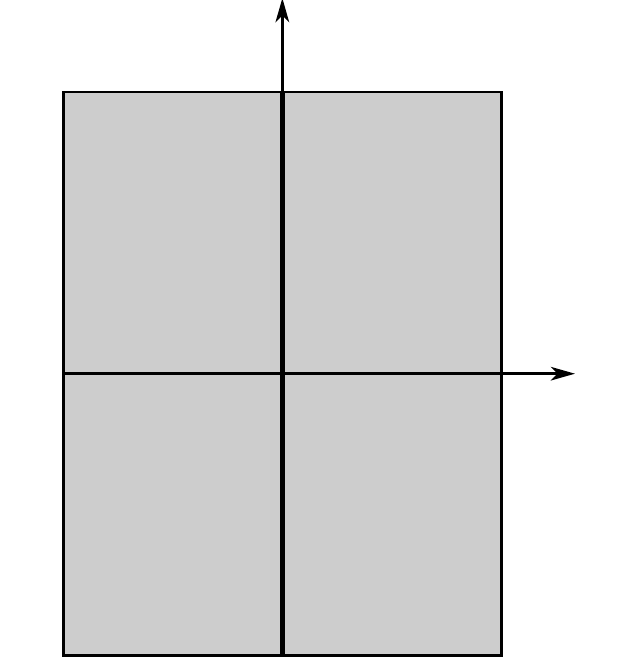
\end{figure}

\newpage

In this setting, the orbits $\gamma_{v(\beta)}(t)$, $t\in [0,T(\beta)]$ are modeled by the ``suspension flow'' $\varphi_t(v(\beta),s)=(v(\beta), s+t)$ if $0\leq s+t<T(\beta)$, $\varphi_{T(\beta)}(v(\beta),0) = (F(v(\beta)),0)$ over the \emph{base map} $F$ with \emph{roof} function $T:M-\Sigma\to\mathbb{R}$, $T(v(\beta))=T(\beta)$. 

\begin{remark}\label{r.suspension} Technically speaking, one needs to ``complete'' the definition of $F$ and $r$ by including the dynamics of the geodesic flow on the compact part of $S$ in order to properly write the geodesic flow on $S$ as a suspension flow. Nevertheless, since the major technical difficulty in the proof of Theorem \ref{t.rapid-mixing-toy-model} comes from the presence of the cusp, we will ignore the excursions of geodesics in the compact part $S$ and we will pretend that the (partially defined) flow $\varphi_t$ is a ``genuine'' suspension flow model. 
\end{remark}

\subsubsection{Rapid mixing of contact suspension flows}

One of the advantages about thinking of the geodesic flow on $S$ as a suspension flow comes from the fact that several authors have previously studied the interplay between the rates of mixing of this class of flows and the features of $F$ and $r$: see, e.g., these papers of Avila-Gou\"ezel-Yoccoz \cite{AvilaGouezelYoccoz} and Melbourne \cite{Melbourne} for some results in this direction (and also for a precise definition of suspension flows).  

For our current purposes, it is worth to recall that B\'alint and Melbourne (cf. Theorem 2.1 [and Remarks 2.3 and 2.5] of \cite{BalintMelbourne}) proved the rapid mixing property for \emph{contact} suspension flows whose base map is modeled by a \emph{Young tower} with \emph{exponential tails} and whose roof function is bounded and \emph{uniformly piecewise H\"older continuous} on each subset of the basis of the Young tower. In particular, the proof of Theorem \ref{t.rapid-mixing-toy-model} is complete once we prove that the base map $F:M-\Sigma\to N$ is modeled by Young towers and the roof function $T:M-\Sigma\to \mathbb{R}$ is bounded and uniformly piecewise H\"older continuous on each element of the basis of the Young tower (whatever this means). 

As it turns out, the theory of Young towers (introduced by Young \cite{Young98}, \cite{Young99}) is a \emph{double-edged sword}: while it provides an adequate setup for the study of statistical properties of systems with some hyperbolicity \emph{once} the so-called \emph{Young towers} were built, it has the \emph{drawback} that the construction of Young towers (satisfying all five natural but technical axioms in Young's definition) is usually a delicate issue: indeed, one has to find a countable Markov partition of a positive measure subset (working as the basis of the Young tower) so that the return maps associated to this Markov partition verify several hyperbolicity and distortion controls, and it is not always clear where one could possibly find such a Markov partition for a given dynamical system. 

Fortunately, Chernov and Zhang \cite{ChernovZhang} gave a list of \emph{sufficient} geometric properties for a \emph{two-dimensional} map like $F:M-\Sigma\to N$ to be modeled by Young towers with exponential tails: in fact, Theorem 10 in Chernov-Zhang paper is a sort of ``black-box'' producing Young towers with exponential tails whenever seven \emph{geometrical} conditions are fulfilled. For the sake of exposition, we will not attempt to check all seven conditions for $F:M-\Sigma\to N$: instead, we will focus on two main conditions called \emph{distortion bounds} and \emph{one-step growth condition}. 

Before we discuss the distortion bounds and the one-step growth condition, we need to recall the concept of \emph{homogeneity strips} (originally introduced by Bunimovich-Chernov-Sinai \cite{BunimovichChernovSinai91}). In our setting, we take $k_0\in\mathbb{N}$ and $\nu=\nu(r)\in\mathbb{N}$ (to be chosen later) and we make a partition of a neighborhood of the \emph{singular set} $\Sigma$ (of geodesics going straight into the cusp) into countably many strips: 
$$H_k:=\left\{(y,\beta)\in M: \frac{1}{(k+1)^{\nu}}<|\beta|<\frac{1}{k^{\nu}}\right\}$$
for all $k\in\mathbb{N}$, $k\geq k_0$. (Actually, $H_k$ has two connected components, but we will slightly abuse of notation by denoting these connected components by $H_k$.)

\begin{figure}[htb!]
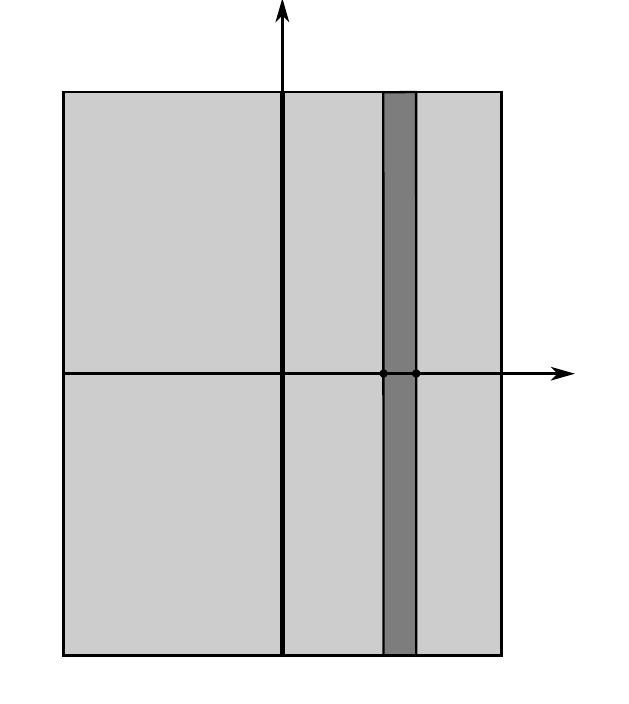
\end{figure}

Intuitively, the partition $H_k$ into polynomial scales $1/k^\nu$ in the parameter $\beta$ is useful in our context because the relevant quantities (such as Gaussian curvature, first and second derivatives, etc.) for the study of the geodesic flow of the surface of revolution blows up with a polynomial speed as the excursions of geodesics get closer the cusp (that is, as $\beta\to 0$). Thus, the important quantities for the analysis of the geodesic flow near the cusp become ``almost constant'' when restricted to one of the homogeneity strips $H_k$. 

Also, another advantage of the homogeneity strips is the fact that they give a rough control of the elements of the countable Markov partition at the basis of the Young tower produced by Chernov-Zhang: indeed, the arguments of Chernov-Zhang show that each element of the basis of their Young tower is completely contained in a homogeneity strip. In particular, the verification of the uniform piecewise H\"older continuity of the roof function $T:M-\Sigma\to N$ follows once we prove that the restriction $T|_{H_k}$of the roof function to each homogeneity strip $H_k$ is uniformly H\"older continuous (in the sense that, for some $0<\alpha=\alpha(r)\leq 1$, the H\"older norms $\|T|_{H_k}\|_{C^\alpha}$ are bounded by a constant \emph{independent} of $k$). 

Coming back to the one-step growth and distortion bounds, let us content ourselves to formulate simpler \emph{versions} of them (while referring to Section 4 and 5 of Chernov-Zhang paper for precise definitions): indeed, the actual definitions of these notions involve the properties of the derivative along unstable manifolds, and, in our current setting, we have just a \emph{partially defined} map $F:M-\Sigma\to N$, so that we can not talk about future iterates and unstable manifolds unless we ``complete'' the definition of $F$. 

Nevertheless, even if $F$ is only partially defined, we still can give crude analogs to unstable directions for $F$ by noticing that the vector field $w^u:=\partial/\partial\beta$ on $M-\Sigma$ (whose leaves are $\{y=constant\}$) morally works like an unstable direction: in fact, this vector field is transverse to the singular set $\Sigma=\{\beta=0\}$ which is a sort of ``stable set'' because all trajectories of the geodesic flow starting at $\Sigma$ converge in the future to the same point, namely, the cusp at $O$. In terms of the ``unstable direction'' $w^u=\partial/\partial\beta$, we define the \emph{expansion factor} $\Lambda(v)$ of $F$ at a point $v=(y,\beta)\in M-\Sigma$ as $\Lambda(v):=\|DF(v)w^u\|/\|w^u\|$, that is, the amount of expansion of the ``unstable'' vector field $w^u$ under $DF(v)$. Note that, from the definitions, the expansion factor $\Lambda(v)$ depends only on the $\beta$-coordinate of $v=(y,\beta)$. So, from now on, we will think of expansion factors as a function $\Lambda(\beta)$ of $\beta$. 

In terms of expansion factors, the (variant of the) distortion bound condition is 
\begin{equation}\label{e.distortion-bounds'}
\frac{d\log\Lambda}{d\beta}(\beta_0)=\frac{\Lambda'(\beta_0)}{\Lambda(\beta_0)}\leq C\frac{1}{\beta_0^{\theta}}
\end{equation}
where $\theta=\theta(r)>0$ satisfies $\nu\theta < \nu+1$, and the (variant of the) one-step growth condition is 
\begin{equation}\label{e.step-growth}
\sum\limits_{k=k_0}^{\infty}\Lambda_k^{-1}<1
\end{equation}
where $\Lambda_k:=\min\limits_{v\in H_k}\Lambda(v) = \min\limits_{\frac{1}{(k+1)^{\nu}}\leq |\beta|\leq\frac{1}{k^{\nu}}} \Lambda(\beta)$.

\begin{remark} The one-step growth condition above is very close to the original version in Chernov-Zhang work (compare \eqref{e.step-growth} with Equation (5.5) in \cite{ChernovZhang}). On the other hand, the distortion bound condition \eqref{e.distortion-bounds'} differs slightly from its original version in Equation (4.1) in Chernov-Zhang paper. Nevertheless, they can be related as follows. The original distortion condition essentially amounts to give estimates $\log\prod\limits_{i=0}^n\frac{\Lambda(F^{-i}(v_1))}{\Lambda(F^{-i}(v_2))}\leq \psi(dist(v_1,v_2))$ (where $\psi$ is a smooth function such that $\psi(s)\to 0$ as $s\to 0$) whenever $x$ and $y$ belong to the same \emph{homogenous unstable manifold} $W$ (i.e., a piece $W$ of unstable manifold such that $F^{-j}(W)$ never intersects the boundaries of the homogeneity strips $H_k$ for all $j\geq 0$ and $k\geq k_0$; the existence of homogenous unstable manifolds through almost every point is guaranteed by a Borel-Cantelli type argument described in Appendix 2 of Bunimovich-Chernov-Sinai's paper \cite{BunimovichChernovSinai91}). Here, one sees that 
$$\log\prod\limits_{i=0}^n\frac{\Lambda(F^{-i}(v_1))}{\Lambda(F^{-i}(v_2))} = \sum\limits_{i=0}^n \frac{\Lambda'(z_i)}{\Lambda(z_i)} dist(F^{-i}(x), F^{-i(y)})$$
for some $z_i\in F^{-i}(W)$. Using the facts that $dist(F^{-i}(x), F^{-i}(y))$ decays exponentially fast (as $x$ and $y$ are in the same unstable manifold $W$) and $F^{-i}(W)$ is always contained in a homogeneity strip $H_{k_i}$ (as $W$ is a homogenous unstable manifold), one can check that the estimate in \eqref{e.distortion-bounds'} implies the desired uniform bound on the previous expression in terms of a smooth function $\psi(s)$ such that $\psi(s)\to 0$ as $s\to 0$. In other words, the estimate \eqref{e.distortion-bounds'} can be shown to imply the original version of distortion bounds, so that we can safely concentrate on the proof of \eqref{e.distortion-bounds'}. 
\end{remark}

At this point, we can summarize the discussion so far as follows. By Melbourne's criterion for rapid mixing for contact suspension flows and Chernov-Zhang criterion for the existence of Young towers with exponential tails for the map $F:M-\Sigma\to N$, we have ``reduced'' the proof of Theorem \ref{t.rapid-mixing-toy-model} to the following statements:

\begin{proposition}\label{p.Holder-roof-function} Given $\nu>0$ and $0<\alpha<1/(\nu+1)$, one has the following ``uniform H\"older estimate'' 
$$\sup\limits_{k\in\mathbb{N}} \|T|_{H_k}\|_{C^\alpha}<\infty$$
whenever $d_0$ is sufficiently small (depending on $r$, $\nu$ and $\alpha$). 
\end{proposition}

\begin{proposition}\label{p.step-growth-distortion-bound} The expansion factor function $\Lambda(\beta)$ satisfies:
\begin{itemize}
\item given $\nu>r/(r-1)$, we can choose $k_0\in\mathbb{N}$ large (and $d_0$ sufficiently small) so that 
$$\sum\limits_{k=k_0}^{\infty}\Lambda_k^{-1}<1$$
where $\Lambda_k = \min\limits_{\frac{1}{(k+1)^{\nu}}\leq |\beta|\leq\frac{1}{k^{\nu}}} \Lambda(\beta)$;
\item given $r>3$, we can choose $\nu>r/(r-1)$ and $\theta>1+2/r$ such that $\nu\theta<\nu+1$ and 
$$\frac{\Lambda'(\beta)}{\Lambda(\beta)}\leq C\frac{1}{\beta^{\theta}}$$
for some (sufficiently large) constant $C>0$ and for all $\beta$.
\end{itemize}
\end{proposition}

The proofs of these two propositions are given in the next two subsections and they are based on the study of perpendicular unstable Jacobi fields related to the variations of geodesics of the form $\gamma_{v(\beta)}(t)$, $0<\beta<\pi/2$. 

\subsubsection{The derivative of the roof function}

From now on, we fix $q\in C=C(d_0)$ (e.g., $q=(d_0, d_0^r, 0)$) and, for the sake of simplicity, we will denote a geodesic $\gamma_{v(\beta)}(t)$ corresponding to an initial vector $v(\beta)\in T^1_qS$ by $\gamma_{\beta}(t)$. Of course, there is no loss of generality here because of the rotational symmetry of the surface $S$. Also, we will suppose that $\beta>0$ as the case $\beta<0$ is symmetric.

Note that the roof function $T(\beta)$ is defined by the condition $\gamma_{\beta}(T(\beta))\in C=C(d_0)$, or, equivalently, 
$$d(\gamma_{\beta}(T(\beta))) = I(d_0) := \int_0^{d_0}\sqrt{1+(rx^{r-1})^2}dx$$
where $d(.)$ denotes the distance from a point to the cusp at $O$ and $I(d_0)$ is the distance from $C(d_0)$ to $O$. By taking the derivative with respect to $\beta$ at $\beta=\beta_0$ and by recalling that $-\nabla d = V$, we obtain that 
$$0=\langle\nabla d(c(\beta_0)), \dot{c}(\beta_0)\rangle = -\langle V(c(\beta_0)), \dot{c}(\beta_0)\rangle$$
where $c(\beta):=\gamma_{\beta}(T(\beta))$. Since $c(\beta) = C(\beta,T(\beta))$ where $C(\beta, t):=\gamma_{\beta}(t)$, we have $\dot{c}(\beta) = \frac{D \gamma_{\beta}}{\partial\beta} (T(\beta)) + \dot{\gamma}_{\beta}(T(\beta)) T'(\beta)$, and, \emph{a fortiori}, 
$$0=\langle V(\gamma_{\beta_0}(T(\beta_0))), \frac {D \gamma_{\beta}}{\partial\beta}|_{\beta=\beta_0}(T(\beta_0))\rangle + \langle V(\gamma_{\beta_0}(T(\beta_0))), \dot{\gamma}_{\beta_0}(T(\beta_0))\rangle T'(\beta_0)$$

Let us compute the two inner products above. By definition of the parameter $\beta$ and the symmetry of the revolution surface $S$, we have $\langle V(\gamma_{\beta}(T(\beta))), \dot{\gamma}_{\beta}(T(\beta))\rangle = - \cos\beta = -\langle V(\gamma_{\beta}(0)), \dot{\gamma}_{\beta}(0)\rangle$. Also, if we denote by $J(t)=\frac {D \gamma_{\beta}}{\partial\beta}(t):= j(t)\cdot J\dot{\gamma}_{\beta}(t)$ the perpendicular (``unstable'') \emph{Jacobi field}\footnote{See the paper \cite{Burns} in this volume for background material on Jacobi fields.} along the geodesic $\gamma_{\beta_0}(t)$ associated to the variation of $C(\beta,t)=\gamma_{\beta}(t)$ with initial conditions $j(0)=0$ and $j'(0)=1$, then 
\begin{eqnarray*}
\langle V(\gamma_{\beta_0}(T(\beta_0))), \frac {D \gamma_{\beta}}{\partial\beta}|_{\beta=\beta_0}(T(\beta_0))\rangle &=& j(T(\beta_0)) \langle V(\gamma_{\beta_0}(T(\beta_0))), J\dot{\gamma}_{\beta_0}(T(\beta_0))\rangle \\ &=& - j(T(\beta_0)) \langle V(\gamma_{\beta_0}(0)), J\dot{\gamma}_{\beta_0}(0)\rangle \\ & = & -j(T(\beta_0))\langle JV(\gamma_{\beta_0}(0)), \dot{\gamma}_{\beta_0}(0)\rangle \\ &=& - j(T(\beta_0))\sin\beta_0
\end{eqnarray*}

From the computation of the inner products above and the fact that they add up to zero, we deduce that $0 = - j(T(\beta_0))\sin\beta_0 - (\cos\beta_0) T'(\beta_0)$, that is,
\begin{equation}\label{e.T'-Jacobi-field}
T'(\beta_0)=-(\tan\beta_0) j(T(\beta_0))
\end{equation}

In other terms, the previous equation says that the derivative $T'(\beta_0)$ can be controlled via the quantity $j(T(\beta_0))$ measuring the growth of the perpendincular Jacobi field $J(t)$ at the return time $T(\beta_0)$. Here, it is worth to recall that Jacobi fields are driven by \emph{Jacobi's equation}:
$$j''(t)+K(t) j(t)=0$$ 
where $K(t)<0$ is the Gaussian curvature of the surface of revolution $S$ at the point $\gamma_{\beta_0}(t)$. Also, it is useful to keep in mind that Jacobi's equation implies that the quantity $u=j'/j$ satisfies \emph{Riccati's equation}
$$u'(t)+u(t)^2 = k(t)^2$$
where $-k(t)^2:=K(t)$. 

In the context of the surface of revolution $S$, these equations are important tools because we have the following explicit formula for the Gaussian curvature $K(q)$ at a point $q=(x,x^r\cos y, x^r\sin y)\in S$:
$$K(q) = \frac{-r(r-1)}{x^2(1+(rx^{r-1})^2)^2}$$ 
In particular, $k(q):=\sqrt{r(r-1)}/x(1+(rx^{r-1})^2)$ verifies $-k(q)^2=K(q)$. 

Next, we take $\varepsilon>0$ and we consider the following auxiliary function:
$$g(q):=\frac{r(1+\varepsilon)}{x}$$
By definition, $k(q)<g(q)$. Furthermore, 
$$k(t)^2-g(t)^2-g'(t)\leq \frac{r(r-1)}{x(t)^2} - \frac{r^2(1+\varepsilon)^2}{x(t)^2} - \frac{r(1+\varepsilon)x'(t)}{x(t)^2}$$
Since the equation $(1+rx(t)^{r-1})^2 x'(t)^2 = 1-c^2/x(t)^{2r}=\cos\beta(t)^2$ (describing the motion of geodesic on $S$) implies that $|x'(t)|\leq 1$, we deduce from the previous inequality that 
\begin{equation}\label{e.curvature-upper-bound}
k(t)^2 - g(t)^2 - g'(t)\leq \frac{1}{x(t)^2}(r(r-1)-(r(1+\varepsilon))^2+r(1+\varepsilon))<0
\end{equation}
for all times $t\in[0,T(\beta)]$.

This estimate allows to control the solution $u=j'/j$ of Riccati's equation along the following lines. The initial data of the Jacobi field $J(t)$ is $j(0)=0$ and $j'(0)$. Hence, 
$$\frac{j'(0)}{j(0)}=\infty>g(0)=\frac{r(1+\varepsilon)}{x(0)} = \frac{r(1+\varepsilon)}{d_0}$$
In particular, there exists a well-defined maximal interval $[0,t_0]\subset [0,T(\beta)]$ where $j'(t)/j(t)\geq g(t)$ for all $t\in[0, t_0]$. By plugging this estimate into Jacobi's equation, we get that 
$$\frac{j''(t)}{j'(t)}=\frac{k(t)^2j(t)}{j'(t)}\leq \frac{k(t)^2}{g(t)}\leq g(t)$$
for each $t\in [0, t_0]$. 

By integrating this inequality (and using the initial condition $j'(0)=1$), we obtain that 
$$\log j'(t_0)=\log \frac{j'(t_0)}{j'(0)} =\int_0^{t_0} \frac{j''(t)}{j'(t)} ds\leq \int_0^{t_0} g(t) dt.$$
Therefore, 
$$j(t_0)\leq \frac{j'(t_0)}{g(t_0)}\leq \frac{1}{g(t_0)} \exp\left(\int_0^{t_0} g(t) dt\right)$$

If $t_0=T(\beta)$, we deduce that $j(T(\beta))\leq \frac{1}{g(T(\beta))}\exp\left(\int_0^{T(\beta)} g(t) dt\right)\leq \frac{1}{k(0)}\exp\left(\int_0^{T(\beta)} g(t) dt\right)$ (as $k(0)=k(T(\beta))<g(T(\beta))$). Otherwise, $0<t_0<T(\beta)$ and $u(t_0)=j'(t_0)/j(t_0) = g(t_0)$. Since $u=j'/j$ satisfies Riccati's equation, we deduce from \eqref{e.curvature-upper-bound} that 
$$u'(t_1) - g'(t_1) = k(t_1)^2 - u(t_1)^2 - g'(t_1) = k(t_1)^2 - g(t_1)^2 - g'(t_1) < 0$$ 
at each time $t_1$ where $u(t_1)=g(t_1)$. It follows that $j'(t)/j(t):=u(t)\leq g(t)$ for all $t\in [t_0, T(\beta)]$. Hence, 
$$\log\frac{j(T(\beta))}{j(t_0)} = \int_{t_0}^{T(\beta)} \frac{j'(t)}{j(t)} dt\leq \int_{t_0}^{T(\beta)} g(t) dt,$$
and, \emph{a fortiori}, 
\begin{eqnarray*}
j(T(\beta))&\leq& j(t_0)\exp\left(\int_{t_0}^{T(\beta)} g(t) dt\right) \\ &\leq& \frac{1}{g(t_0)} \exp\left(\int_{0}^{t_0} g(t) dt\right) \exp\left(\int_{t_0}^{T(\beta)} g(t) dt\right) \\ 
&\leq & \frac{1}{k(0)} \exp\left(\int_0^{T(\beta)} g(t) dt\right).  
\end{eqnarray*}
In other words, we proved that 
\begin{equation}\label{e.jTbeta-estimate}
j(T(\beta))\leq \frac{1}{k(0)} \exp\left(\int_0^{T(\beta)} g(t) dt\right)  
\end{equation}
independently whether $t_0=T(\beta)$ or $0<t_0<T(\beta)$. 

Now, the quantity $\exp\left(\int_0^{T(\beta)} g(t) dt\right)$ can be estimated as follows. By deriving Clairaut's relation $x(t)^r\sin\beta(t)=c$, we get 
$$rx(t)^{r-1}x'(t)\sin\beta(t) + x(t)^r(\cos\beta(t))\beta'(t)=0,$$
that is, 
\begin{equation}\label{e.1/d-formula}
\frac{1}{x(t)} = -\frac{1}{r}\frac{\cos\beta(t)}{x'(t)}\frac{\beta'(t)}{\sin\beta(t)}
\end{equation}
Since $\sin\beta(t)\sim\beta(t)$ (as we are interested in small angles $|\beta|<k_0^{-\nu}$, $k_0$ large) and $\cos\beta(t)\sim x'(t)$ (thanks to the relation $(1+rx(t)^{r-1})^2 x'(t)^2 =1 - c^2/x(t)^{2r} = (\cos\beta(t))^2$ and the fact that $r>1$ and, thus, $1\leq 1+(rx(t)^{r-1})^2\leq 1+(rd_0^{r-1})^2\sim 1$ for $d_0$ small), we conclude that 
$$g(t)=\frac{r(1+\varepsilon)}{x(t)}\leq (1+2\varepsilon)\frac{\beta'(t)}{\beta(t)}$$
for $t\in [0, T(\beta)/2]$. Here, we used the fact that $x'(t)<0$ for $t\in[0, T(\beta)/2]$. Therefore, 
$$\int_0^{T(\beta)/2} g(t) dt \leq (1+2\varepsilon) \log\frac{\pi/2}{\beta(0)}$$
since $\beta(T(\beta)/2)=\pi/2$.  Also, the symmetry of the surface $S$ implies $x(t)=x(T(\beta)-t)$ and, hence, 
$$\int_0^{T(\beta)/2}g(t) dt = \int_{T(\beta)/2}^{T(\beta)} g(t) dt$$
In summary, we have shown that $\int_0^{T(\beta)}g(t) dt\leq 2(1+2\varepsilon)\log(\pi/2\beta(0))$, i.e.,  
\begin{equation}\label{e.exp-g-integral}
\exp\left(\int_0^{T(\beta)} g(t) dt\right)\leq (\pi/2)^{2(1+2\varepsilon)}\frac{1}{\beta(0)^{2(1+2\varepsilon)}}
\end{equation}

By putting together \eqref{e.T'-Jacobi-field}, \eqref{e.jTbeta-estimate} and \eqref{e.exp-g-integral}, we conclude that 
\begin{equation}\label{e.T'-estimate}
|T'(\beta_0)|\leq\frac{\tan\beta_0}{k(0)}\exp\left(\int_0^{T(\beta_0)} g(t) dt\right)\leq C\frac{\beta_0}{\beta_0^{2(1+2\varepsilon)}} = \frac{C}{\beta_0^{1+4\varepsilon}}
\end{equation}
for some constant $C>0$ depending on $r>1$ and $\varepsilon>0$. 

At this stage, we are ready to complete the proof of Proposition \ref{p.Holder-roof-function}. 

\begin{proof} Let us estimate the H\"older constant $\|T|_{H_k}\|_{C^{\alpha}}$. For this sake, we fix $\beta_1, \beta_2\in H_k$ and we write 
$$\frac{|T(\beta_1)-T(\beta_2)|}{|\beta_1-\beta_2|^{\alpha}} = |T'(\beta_3)|\cdot |\beta_1-\beta_2|^{1-\alpha}$$ 
for some $\beta_3\in H_k$ between $\beta_1$ and $\beta_2$. Since $|\beta_1-\beta_2|\leq k^{-\nu}-(k+1)^{-\nu}\leq \nu/k^{\nu+1}$ and $|\beta_3|\geq (k+1)^{-\nu}$, it follows from \eqref{e.T'-estimate} that  
$$\frac{|T(\beta_1)-T(\beta_2)|}{|\beta_1-\beta_2|^{\alpha}} \leq  C\nu^{1-\alpha}\frac{(k+1)^{\nu(1+4\varepsilon)}}{k^{(\nu+1)(1-\alpha)}}$$
Because $\beta_1$ and $\beta_2$ are arbitrary points in $H_k$, we have that
$$\|T|_{H_k}\|_{C^{\alpha}}\leq C\frac{(k+1)^{\nu(1+4\varepsilon)}}{k^{(\nu+1)(1-\alpha)}}$$
where $C>0$ is an appropriate constant.  

Now, our assumption $0<\alpha<1/\nu+1$ implies that we can choose $\varepsilon>0$ sufficiently small so that $\nu(1+4\varepsilon)\leq \nu(1-\alpha)$. By doing so, we see from the previous estimate that 
$$\sup\limits_{k\in\mathbb{N}}\|T|_{H_k}\|_{C^{\alpha}}<\infty$$
whenever $\varepsilon>0$, i.e., $d_0>0$, is sufficently small. This proves Proposition \ref{p.Holder-roof-function}.
\end{proof}

\subsubsection{Some estimates for the expansion factors $\Lambda(\beta)$}

Similarly to the previous subsection, the proof of Proposition \ref{p.step-growth-distortion-bound} uses the properties of Jacobi's and Riccati's equation to study 
\begin{equation}\label{e.expansion-factor-definition}
\Lambda(\beta):=j(T(\beta)) + j'(T(\beta))
\end{equation}
where $j(t)=j_{\beta}(t)$ is the scalar function (with $j(0)=0$ and $j'(0)=1$) measuring the size of the perpendicular ``unstable'' Jacobi field along $\gamma_{\beta}(t)$. 

We begin by giving a lower bound on $\Lambda(\beta)$.  Given $\varepsilon>0$, let us choose $d_0=d_0(\varepsilon, r)>0$ small so that 
$$\sqrt{1-\varepsilon}<\frac{1}{1+(rd_0^{r-1})^2}(\leq 1)$$
Of course, this choice of $d_0$ is possible because $r>1$. Next, we consider the auxiliary function:
$$h(q):=\frac{(r-1)(1-2\varepsilon)}{x}.$$

By definition, $h(q)<\sqrt{r(r-1)}/x(1+(rd_0^{r-1})^2)\leq k(q)$. Furthermore, 
$$h'(t) = -\frac{(r-1)(1-\varepsilon)}{x(t)^2}x'(t)$$
In particular, 
$$k(t)^2-h(t)^2-h'(t)>\frac{r(r-1)(1-\varepsilon)}{x(t)^2}-\frac{(r-1)^2(1-2\varepsilon)^2}{x(t)^2} - \frac{(r-1)(1-2\varepsilon)x'(t)}{x(t)^2}$$
Since $|x'(t)|\leq 1$ (cf. the paragraph before \eqref{e.curvature-upper-bound}), we deduce from the previous estimate that
$$k(t)^2-h(t)^2-h'(t)>0$$ 
This inequality implies that the solution $u(t)=j'(t)/j(t)$ of Riccati's equation satisfies $u(t)\geq h(t)$ for all $t\in [0, T(\beta)]$. Indeed, the initial condition $j'(0)=1$, $j(0)=0$ says that $u(0)=\infty>h(0)$ and the inequality above tells us that 
$$u'(t_1)-h'(t_1) = k(t_1)^2 - u(t_1)^2 - h'(t_1) = k(t_1)^2 - h(t_1)^2 - h'(t_1) >0$$
at any time $t_1$ where $u(t_1)=h(t_1)$. 

By integrating the estimate $u(t)=j'(t)/j(t)\geq h(t)$ over the interval $[t_0, T(\beta)]$, we obtain that 
$$\log \frac{j(T(\beta))}{j(t_0)} = \int_{t_0}^{T(\beta)} \frac{j'(t)}{j(t)} dt \geq \int_{t_0}^{T(\beta)} h(t) dt,$$
i.e., 
$$j(T(\beta))\geq j(t_0)\exp\left(\int_{t_0}^{T(\beta)} h(t) dt\right)$$

For sake of concreteness, let us set $t_0:=d_0/10$ and let us restrict our attention to geodesics whose initial angle $\beta=\beta(0)$ with the meridians of $S$ are sufficiently small so that $T(\beta)\geq d_0/2$. In this way, we have that $j(t_0)\geq t_0=d_0/10$ (thanks to Jacobi's equation $j''=k^2j$ and our initial conditions $j(0)=0$ and $j'(0)=1$). In this way, the inequality above becomes
$$j(T(\beta))\geq \frac{d_0}{10}\exp\left(\int_{t_0}^{T(\beta)} h(t) dt\right)$$

Next, we observe that $\exp\left(\int_{t_0}^{T(\beta)} h(t) dt\right)$ can be bounded from below in a similar way to our derivation of a bound from above to $\exp\left(\int_{0}^{T(\beta)} g(t) dt\right)$ in the previous subsection: in fact, by repeating the arguments appearing after \eqref{e.1/d-formula} above, one can show that 
$$h(t)\geq \frac{(r-1)(1-3\varepsilon)}{r}\frac{\beta'(t)}{\beta(t)}$$
and 
$$\exp\left(\int_{t_0}^{T(\beta)} h(t) dt \right)\geq \overline{c} \frac{1}{\beta(0)^{(r-1)(1-3\varepsilon)/r}}$$
where $\overline{c}>0$ is an adequate (small) constant depending on $r$, $d_0$ and $\varepsilon$. 

By putting together the estimates above, we deduce that 
$$\Lambda(T(\beta))\geq j(T(\beta))\geq c\frac{1}{\beta(0)^{(r-1)(1-3\varepsilon)/r}}$$
where $c=d_0\overline{c}/10$. 

This inequality shows that 
$$\sum\limits_{k=k_0}^{\infty} \Lambda_k^{-1}\leq \frac{1}{c} \sum\limits_{k=k_0}^{\infty}\frac{1}{(k+1)^{(r-1)\nu(1-3\varepsilon)/r}}$$ 
Thus, if $\nu>r/(r-1)$, then we can choose $\varepsilon>0$ small (with $(r-1)(1-3\varepsilon)\nu/r>1$) and $k_0\in\mathbb{N}$ large so that (our variant of) the one-step growth condition \eqref{e.step-growth} holds. This proves the first part of Proposition \ref{p.step-growth-distortion-bound}. 

Finally, we give an indication of the proof of the second part of Proposition \ref{p.step-growth-distortion-bound} (i.e., the distortion bound \eqref{e.distortion-bounds'}). We start by writing 
$$\frac{\Lambda'(\beta)}{\Lambda(\beta)} = \frac{d}{d\beta}\log\Lambda(\beta)$$
and by noticing that 
$$\log\Lambda(\beta) = \log(j(T(\beta))+j'(T(\beta))) = \log j(T(\beta)) + \log(1+u(T(\beta)))$$
Next, we take the derivative with respect to $\beta$ of the previous expression. Here, we obtain several terms involving some quantities already estimated above via Jacobi's and Riccati's equation (such as $j(T(\beta))$, $T'(\beta)$, etc.), but also a new quantity appears, namely, $u_{\beta}(t)$, i.e., the derivative with respect to $\beta$ of the family of solutions $u(t)=u(t,\beta)$ of Riccati's equation along $\gamma_{\beta}(t)$. Here, the ``trick'' to give bounds on $u_{\beta}(t)$ is to derive Riccati's equation 
$$u'(t)+u(t)^2 = k(t)^2$$
with respect to $\beta$ in order to get an ODE (in the time variable $t$) satisfied by $u_{\beta}(t)$. In this way, it is possible to see that one has reasonable bounds on $u_{\beta}(t)$ as soon as the derivative $k_{\beta}$ of the square root of the absolute value $-K$ of the Gaussian curvature. Here, $k_{\beta}$ can be bounded by recalling that we have an explicit formula 
$$K=-r(r-1)/x^2(1+(rx^{r-1})^2)^2$$ 
for the Gaussian curvature. By following these lines, one can prove that, for a given $\varepsilon>0$, the distortion bound  
$$\frac{\Lambda'(\beta)}{\Lambda(\beta)}\leq C\frac{1}{\beta(0)^{(1+2/r)(1+\varepsilon)}} = \frac{C}{\beta(0)^{\theta}}$$
holds whenever $d_0>0$ is taken sufficiently small. In other words, by taking $\theta = \theta(r) = (r+2)(1+\varepsilon)/r$, we have $\Lambda'(\beta)/\Lambda(\beta) 
\leq C\beta(0)^{-\theta}$. 

Note that the estimate in the previous paragraph gives the desired distortion bounds \eqref{e.distortion-bounds'} once we show that $\theta=\theta(r)=\frac{(r+2)}{r}+$ can be selected such that $\nu\theta<\nu+1$. In order to check this, it suffices to recall that $\nu-r/(r-1)>0$ can be taken arbitrarily small (cf. the proof of the first part of Proposition \ref{p.step-growth-distortion-bound}), i.e., $\nu=\frac{r}{r-1}+$. So, 
$$\nu\theta = \left(\frac{r}{r-1}+\right)\left(\frac{r+2}{r}+\right) = \frac{r+2}{r-1}+$$ 
and 
$$\nu+1 = \frac{r}{r-1}+1+ = \frac{2r-1}{r-1}+$$
Since $r+2<2r-1$ for $r>3$, it follows that $\nu\theta<\nu+1$ for adequate choices of $\theta$ and $\nu$. This completes our sketch of proof of the second part of Proposition \ref{p.step-growth-distortion-bound}.

\end{document}